\documentclass[11pt,oneside,english]{amsart}
\usepackage[T1]{fontenc}
\usepackage[latin9]{inputenc}
\usepackage{geometry}
\usepackage{float}
\usepackage{mathrsfs}
\usepackage{amsthm}
\usepackage{amsbsy}
\usepackage{amstext}
\usepackage{amssymb}
\usepackage{fixltx2e}
\usepackage{graphicx}
\usepackage[all]{xy}
\usepackage{amsmath}
\usepackage{mathtools}
\usepackage{setspace}
\usepackage{enumerate}
\usepackage{comment}
\usepackage{babel}
\usepackage{tikz}
\usetikzlibrary{arrows}
\usepackage{ifpdf}
\ifpdf
\IfFileExists{lmodern.sty}
 {\usepackage{lmodern}}{}
\makeatletter

\providecommand{\tabularnewline}{\\}

\numberwithin{equation}{section}
\numberwithin{figure}{section}
\numberwithin{table}{section}
\theoremstyle{plain}
\newtheorem{thm}{\protect\theoremname}[section]
  \theoremstyle{plain}
  \newtheorem{cor}[thm]{\protect\corollaryname}
  \theoremstyle{definition}
  \newtheorem{defn}[thm]{\protect\definitionname}
  \theoremstyle{remark}
  \newtheorem{rem}[thm]{\protect\remarkname}
  \theoremstyle{plain}
  \newtheorem{prop}[thm]{\protect\propositionname}
  \theoremstyle{definition}
  \newtheorem{example}[thm]{\protect\examplename}
  \theoremstyle{plain}
  \newtheorem{fact}[thm]{\protect\factname}
  \theoremstyle{plain}
  \newtheorem{lem}[thm]{\protect\lemmaname}
  \newtheorem{question}[thm]{Question}
  \newtheorem{notation}[thm]{Notation}
  \newtheorem{conjecture}[thm]{Conjecture}
  
  \newtheorem*{theorem*}{Theorem}

  \theoremstyle{plain}
  \newtheorem*{cor*}{\protect\corollaryname}

\newcommand{\bb}[1]{\mathbb{#1}}
\newcommand{\HFh}{\widehat{HF}}
\newcommand{\surg}[2]{S^3_{#1,#2}(L)}
\newcommand{\As}[2]{\mathfrak{A}^-_{#1,#2}}
\newcommand{\AAs}[2]{\mathfrak{A}^-(\mathcal{#1},\mathrm{#2})}
\newcommand{\Is}[2]{I_{#1}^{\overrightarrow{#2}}}
\newcommand{\V}[1]{\overrightarrow #1}
\newcommand{\bigchi}{\mbox{\Large$\chi$}}
\newcommand{\s}{\mathrm{\bf{s}}}
\newcommand{\im}{\mathrm{Im\ }}

\newcommand{\Ker}{\mathrm{Ker}}
\newcommand{\xyR}[1]{
  \xydef@\xymatrixrowsep@{#1}}
\newcommand{\xyC}[1]{
  \xydef@\xymatrixcolsep@{#1}}

\@ifundefined{date}{}{\date{}}

\makeatother

  \providecommand{\corollaryname}{Corollary}
  \providecommand{\definitionname}{Definition}
  \providecommand{\examplename}{Example}
  \providecommand{\factname}{Fact}
  \providecommand{\lemmaname}{Lemma}
  \providecommand{\propositionname}{Proposition}
  \providecommand{\remarkname}{Remark}
\providecommand{\theoremname}{Theorem}

\begin{document}
\title{$L$-space surgeries on links}

\author{Yajing Liu}

\begin{abstract}
An $L$-space link is a link in $S^3$ on which all large surgeries are $L$-spaces. In this paper, we initiate a general study of the definitions, properties, and examples of $L$-space links. In particular, we find many hyperbolic $L$-space links, including some chain links and two-bridge links; from them, we obtain many hyperbolic $L$-spaces by integral surgeries, including the Weeks manifold. We give bounds on the ranks of the link Floer homology of $L$-space links and on the coefficients in the multi-variable Alexander polynomials. We also describe the Floer homology of surgeries on any $L$-space link using the link surgery formula of Manolescu and Ozsv\'{a}th. As applications, we compute the graded Heegaard Floer homology of surgeries on 2-component $L$-space links in terms of only the Alexander polynomial and the surgery framing, and give a fast algorithm to classify $L$-space surgeries among them.
\end{abstract}
\maketitle

\section{Introduction}

\subsection{Background on $L$-spaces.}
Heegaard Floer homology is a package of invariants for 3-manifolds and links introduced by Ozsv\'{a}th and Szab\'{o} in \cite{[OS]HF1}.
It has many applications to topological questions.  See \cite{[O-S]detect_genus,[OS]ThurstonNorm,[YiNi]Fiber,[O-S]4-ball_genus,[OS]rational_surgery,[Wu]CosmeticSurgery,
  [Wu-Ni]CosmeticSurgery}.
An $L$-space is a rational homology sphere with the simplest Heegaard Floer homology. In this paper, for simplicity, we work in the field $\bb{F}=\bb{Z}/2\bb{Z}$, and then we use the following definition:

\begin{defn}[$\bb{Z}/2\bb{Z}$-$L$-space]
\label{defn:L-spaces}
A 3-manifold $M$ is called an $L$-space, if it is a rational homology sphere and
$\dim_{\bb{F}}(\widehat{HF}(M))=\left|H_1(M) \right|$.
\end{defn}

Examples of $L$-spaces include all 3-manifolds with elliptic geometry and double branched covers over quasi-alternating links. $L$-spaces are of interests in 3-manifold topology.
An $L$-space does not admit any co-oriented $C^2$ taut foliations; see Theorem 1.4 from \cite{[O-S]detect_genus}. Examples of closed hyperbolic manifolds admitting no taut foliations are very interesting and first found in \cite{[RSS]Infinitely_many_hyperbolic_mfld_without_taut_foliation} and \cite{[Calegari-Dunfield]Laminations&groups_of_Hom(S^1)} by considering their fundamental groups. In fact, any hyperbolic $\bb{Z}/2\bb{Z}$-$L$-space also provides an example of hyperbolic manifold admitting no co-oriented taut foliations. This is because in the proof of Theorem 1.4 of \cite{[O-S]detect_genus}, it is pointed out that any $\bb{Z}/p\bb{Z}$-$L$-space does not admit a co-oriented taut foliation for all prime numbers $p$. There is also a conjecture of Boyer-Gordon-Watson from \cite{[Boyer-Gordon-Watson]Left-orderablity}  relating $L$-spaces with left-orderability of the fundamental group.

In \cite{[OS]lens_space_surgery},  $L$-space knots were introduced by Oszv\'{a}th and Szab\'{o}, in order to study the Berge conjecture on
lens space surgeries on knots in $S^3$. For further results towards the Berge conjecture,  see \cite{[J.Greene]lens_space_surgery,[Hedden]07Berge}.

\begin{defn}[$L$-space knot]
\label{defn:L-spaces}
A knot $K\subset S^3$ is called an $L$-space knot, if there is a positive integer $n$, such that the $n$-surgery on
$K$ is an $L$-space.
\end{defn}

Since every 3-manifold is a surgery on a link in $S^3$, one can study $L$-spaces by surgeries on links.
In this paper, we focus on a class of links called \emph{$L$-space links}, whose large surgeries are all $L$-spaces. These links are natural generalizations
of $L$-space knots. The terminology of $L$-space links was introduced by  Gorsky and N\'{e}methi in \cite{[Gorsky_Nemithi]algebriac_links} to study algebraic links. Actually, Ozsv\'{a}th, Stipsicz and Szab\'{o} have shown that all plumbing trees are $L$-space links in \cite{[O-S-S]spectral_seq_lattice_homology}. The surgeries on algebraic links and plumbing trees are all graph manifolds. In this paper, we give many examples of hyperbolic $L$-space links, including some families of two-bridge links and chain links.  In turn, these hyperbolic $L$-space  links provide many examples of hyperbolic $L$-spaces, including the famous Weeks manifold; see Section 3. All of these hyperbolic $L$-spaces are derived from elliptic $L$-spaces, by using the surgery exact triangle of Floer homology.

It turns out that $L$-space links are rich in geometry and simple in algebra. All the generalized Floer complexes are chain homotopy equivalent to $\bb{F}[[U]]$ and the link Floer homology are controlled by their Alexander polynomials; see Section 5 and 4. Moreover, there are $L$-space links of all kinds of geometry with arbitrarily many components, including non-prime links, torus links, satellite links, and hyperbolic links; see Example \ref{eg:L-space_links}. There are also non-fibered prime $L$-space links, contrasting $L$-space knots.

Here, all the links are oriented links in $S^3$, and all Floer complexes are of the completed version, meaning over the completion $\bb{F}[[U]]$.

\subsection{$L$-space knots.}
Examples and properties of $L$-space knots have been extensively
studied in the literature. We list some of them here.

\begin{example}
\label{eg:examples_of_L-space_knots}
Examples of $L$-space knots include lens space knots such as Berge knots (up to mirror), algebraic
knots (which are torus knots and their cables), and $(-2,3,q)$ pretzel knots with $q>1$ odd
 (which are hyperbolic). See \cite{[OS]lens_space_surgery, [Hedden]Knot_Floer&cabling, [Lidman-Moore]prezel_L_knots, [Baker-Moore]Montesinos_L_knots}.
\end{example}

\begin{fact}
\label{fact:defn_L-space_knot}
In \cite{[OS]rational_surgery}, it is shown that a positive rational $L$-space surgery implies a positive
integer $L$-space surgery; a positive integer $L$-space surgery implies that all large surgeries are
$L$-spaces.
\end{fact}

\begin{fact}[\cite{[OS]lens_space_surgery}]
\label{fact:alternating_L-space_knot}
If $K$ is an  alternating $L$-space knot, then $K$ is a $T(2,2n+1)$ torus knot.
\end{fact}

\begin{fact}[\cite{[YiNi]Fiber}]
\label{fact:fiberedness_of_L-space_knot}
An $L$-space knot is a fibered knot.
\end{fact}

\begin{fact}[\cite{[OS]lens_space_surgery}]
\label{fact:AlexPoly_L-space_knot}
Let $K$ be an $L$-space knot. The knot Floer homology $\widehat{HFK}(K)$ is determined by the Alexander polynomial of $K$,
 and $\mathrm{rank}(\widehat{HFK}(K,s))\leq1, \forall s\in \bb{Z}.$
\end{fact}

These properties provide strong constraints on $L$-space knots. However, it turns out that  none of the above properties extends to \emph{$L$-space links} immediately.
\subsection{$L$-space links.}
In \cite{[Gorsky_Nemithi]algebriac_links},  Gorsky and N\'{e}methi define $L$-space links in terms of large surgeries.

\begin{defn}[$L$-space link]
\label{defn:L-space links}
An $l$-component link $L\subset S^3$ is called an \emph{$L$-space link}, if all of its positive large surgeries are
$L$-spaces, that is, there exist integers $p_1,...,p_l$, such that $S^3_{n_1,...,n_l}(L)$ is an $L$-space for
all $n_1,...,n_l$ with $n_i>p_i,\forall 1\leq i\leq l.$ Note that whether $L$ is an $L$-space link does not
depend on the orientation of $L$. A link $L$ is called a \emph{non-$L$-space link}, if neither $L$ nor its mirror is an $L$-space link.
\end{defn}

The large surgeries on the link $L$ are governed by the \emph{generalized Floer complexes} $\mathfrak{A}^-_{\mathrm{\bf{s}}}(L)$'s with $\mathrm{\bf{s}}\in \bb{H}(L)$, which were introduced by Manolescu and Ozsv\'{a}th in \cite{[MO]link_surgery}. Here, $\bb{H}(L)$ is defined below. Also, see Definition \ref{defn:A_s} for the generalized Floer complexes.

\begin{defn}[$\bb{H}(L)$]
For an oriented link $L$ with $l$ components, we define $\bb{H}(L)$ to be the affine lattice over  $\bb{Z}^l$,
\[ \bb{H}(L)=\bigoplus_{i=1}^{l}\bb{H}(L)_i, \ \ \bb{H}(L)_i=\bb{Z}+\frac{\mathrm{lk}(L_i,L-L_i)}{2}.\]
\end{defn}

Based on the knowledge of $\mathfrak{A}^-_{\mathrm{\bf{s}}}(L)$, we have the following necessary condition on $L$-space links.

\begin{lem}
\label{lem:L-space_sublinks}
If L is an $L$-space link, then all sublinks of $L$ are $L$-space links.
\end{lem}

We also describe $L$-space links in three other equivalent ways, which are easy to use. To this end, we study the relation between $L$-space surgeries and large surgeries on links. Using the \emph{$L$-space surgery induction lemma} (Lemma \ref{lem:L-space_surgery_iteration}) and the generalized Floer complexes, we prove the following result.

\begin{prop}
\label{prop:equivalent_defn_L-space_links}
The following conditions are equivalent:
\begin{enumerate}[(i)]
\item $L$ is an $L$-space link;
\item there exists a surgery framing $\Lambda(p_1,...,p_l)$, such that
for all sublink $L'\subseteq L$, $\det(\Lambda(p_1,...,p_l)|_{L'})>0$
and $S^3_{\Lambda|_{L'}}(L')$ is an $L$-space; (Notice that at this time $\Lambda$ is positive definite.)
\item $H_*(\mathfrak{A}^-_{\mathrm{\bf{s}}}(L))=\bb{F}[[U]], \forall \mathrm{\bf{s}}\in \bb{H}(L);$
\item $H_*(\hat{\mathfrak{A}}_{\mathrm{\bf{s}}}(L))=\bb{F},\forall \mathrm{\bf{s}}\in\bb{H}(L).  $
\end{enumerate}
\end{prop}

Using grid diagrams as in \cite{[MOS]Combinatorial_Knot_Floer}, one can compute $\mathfrak{A}^-_{\s}$ combinatorially and check condition (iii) or (iv). On the other hand, for special class of links, it is more convenient to use condition (ii). For instance, it follows immediately that an algebraically split link is an $L$-space link if and only if it admits a positive surgery $\Lambda$ such that the surgeries restricted to sublinks are all $L$-spaces. Note that if we work with $\bb{Z}$ coefficients, conditions (i) and (ii) are also equivalent.

In contrast to Fact \ref{fact:defn_L-space_knot}, a single $L$-space surgery (with positive surgery coefficients)
on $L$ fails to imply that all the large surgeries on $L$ are
$L$-spaces. See Example \ref{eg:figure8}. It leads
us to define \emph{weak $L$-space links}.

\begin{defn} [Weak $L$-space link]
\label{defn:strong/weak_L-space_links}
A link $L$ is called a \emph{weak $L$-space link}, if there exists an $L$-space surgery on $L$.
\end{defn}

There are generalizations of $L$-space links, called \emph{generalized $(\pm\cdots\pm) L$-space links}, by considering the corresponding types of generalized large surgeries. There are also parallel theories of $\mathfrak{A}^-_{\mathrm{\bf{s}}}$ for generalized large surgeries and the link surgery formula. See Section 2. An $L$-space link is literally a generalized $(+\cdots+)L$-space link. Note that there are generalized $(+-)L$-space links that are \emph{non-$L$-space} links.

\begin{example}
\label{eg:L-space_links}
We have the  following examples of $L$-space links and generalized $L$-space links.
\begin{enumerate}[(A)]
\item Split disjoint unions of $L$-space knots are $L$-space links.
\item Two-bridge links $b(rq-1,-q)$ with $r,q$ being positive odd integers are all $L$-space links, which include $T(2,2n)$ torus links. See Theorem \ref{thm:two-bridge_L-space_links}. Note that except for $T(2,2n)$, they are all hyperbolic links.
\item A 2-component $L$-space link: $L7n1$ in the Thistlethwaite link table.  See Example \ref{eg:small_links}.
\item Some 3-component $L$-space links: Borromean rings, $L6a5$, $L6n1$, $L7a7$ and a link in  Example \ref{eg:Weeks_mfld}. See Example \ref{eg:small_links}.
\item A hyperbolic 4-chain $L$-space link: See Example \ref{eg:4-component_L-space_link}.
\item A hyperbolic 5-chain generalized $(++++-)L$-space link: See Example \ref{eg:five_rings}.
\item Two families of hyperbolic $L$-space chain links: See Example \ref{eg:chain links1} and Example \ref{eg:chain links2}.
\item A sequence of plumbing graphs that are generalized $L$-space links: See Example \ref{eg:Hopf_plumbing}.
\item All plumbing trees of unknots are $L$-space links. This was proved by Ozsv\'{a}th and Szab\'{o} in \cite{[O-S]Plumbed_mfld}. See Example \ref{eg:plumbing_trees}.
\item All algebraic links are $L$-space links. This was proved by Gorsky and N\'{e}methi in \cite{[Gorsky_Nemithi]algebriac_links}.
\item See Table \ref{table:small links} for the list of which links with crossing number $\leq 7$ are $L$-space links.
\end{enumerate}
\end{example}

In contrast to Fact \ref{fact:alternating_L-space_knot}, there are alternating hyperbolic $L$-space links,
for example, all two-bridge links $b(rq-1,-q)$ with $r,q>1$ being positive odd integers.  Surgeries on these hyperbolic $L$-space links can give examples of hyperbolic $L$-spaces which are neither surgery nor double branched cover over any knot. See Example \ref{eg:whitehead and Borromean}. In fact, surgeries on these $L$-space two-bridge links are always double branched covers over some links. It is not clear to us whether those links are quasi-alternating or not.

In relation to Example \ref{eg:L-space_links} (B), we make the following conjecture:

\begin{conjecture}
\label{conj:L-space_two-bridge}
The set of all  $L$-space two-bridge links is
$$\{b(rq-1,-q):r,q\text{ are positive odd integers}\}.$$
\end{conjecture}

Using the algorithm from \cite{[YL]Surgery_2-bridge_link} for computing $\hat{\mathfrak{A}}_{\mathrm{\bf{s}}}(L)$ for two-bridge
links, we verify that Conjecture \ref{conj:L-space_two-bridge} is true for all two-bridge links $b(p,q)$ with $p\leq100.$

Compared with Fact \ref{fact:AlexPoly_L-space_knot}, we study the Alexander polynomials of $L$-space links using $\mathfrak{A}^-_{\mathrm{\bf{s}}}(L)$.

\begin{thm}
\label{thm: AlexPoly-of-L-space-link}
Suppose $L$ is an $l$-component $L$-space link with $l\geq 2$, and has the multi-variable Alexander polynomial as follows
\[ \Delta_L(x_1,...,x_l)=\sum_{i_1,...,i_l}a_{i_1,...,i_l}\cdot x_1^{i_1} \cdots x_l^{i_l}.
\]
 Then,
\begin{align}
\label{eq:ineq1}
\mathrm{rank}_{\bb{F}}(HFL^-(L,\mathrm{\bf{s}}))\leq 2^{l-1},\forall \mathrm{\bf{s}}\in \bb{H}(L),\\
\label{eq:ineq2}
-2^{l-2} \leq a_{i_1,...,i_l}\leq 2^{l-2},\forall i_1,...,i_l.
\end{align}

In particular, for a 2-component $L$-space link, the multi-variable Alexander polynomial has non-zero
coefficients $\pm 1$. Moreover, fixing $i_1$, the signs of non-zero $a_{i_1,*}$'s are alternating; similarly fixing $i_2$, the signs of non-zero $a_{*,i_2}$'s are alternating.
\end{thm}

\begin{rem}
Inequality \eqref{eq:ineq1} is sharp for $l=2$. For example, for the Whitehead link $\mathit{Wh}$, $HFL^-(\mathit{Wh},0,1)$ equals to  $\bb{F}\oplus \bb{F}$. Inequality \eqref{eq:ineq1} can also be deduced from a spectral sequence of Gorsky and  N{\'e}methi from \cite{[Gorsky_Nemithi]algebriac_links}.

Inequality \eqref{eq:ineq2} is sharp for $l=3$. The mirror of $L7a7$ is an $L$-space link with Alexander polynomial
\[\Delta_{L7a7}(u,v,w)=\frac{uvw-uv-uw+2u-2vw+v+w-1}{\sqrt{uvw}}.\]
\end{rem}

In contrast to knots, the Alexander polynomial condition does not give strong constraints for alternating links. In \cite{[OS]lens_space_surgery}, it is shown that if $K$ is an alternating knot with Alexander polynomial satisfying the condition in Fact \ref{fact:AlexPoly_L-space_knot}, then $K$ is a $T(2,2n+1)$ torus knot; see Proposition \ref{prop:Alternating_Alex_knots} and Theorem \ref{thm:alternating_L-space_knot}.   On the other hand, we find infinitely many hyperbolic alternating links with multi-variable Alexander polynomial satisfying Inequality \eqref{eq:ineq2}, including $L$-space links and non-$L$-space links. See Section 4.2.

Theorem \ref{thm: AlexPoly-of-L-space-link} also implies that
a Floer homologically thin $L$-space 2-component link $L$ has fibered link exterior.

In contrast to Fact \ref{fact:fiberedness_of_L-space_knot}, there are non-fibered $L$-space links. For example,
 the split disjoint union of two $L$-space knots is a non-fibered $L$-space link, since the complement is not
 irreducible any more.  In fact, there are also many non-fibered $L$-space
 links among hyperbolic two-bridge links. See Example \ref{eg:Non-fibered_hyperbolic_L-space_lk}.

Actually, there are additional constraints on the Alexander polynomials of an $L$-space link; see Theorem \ref{thm:normalization_of_L-space_link} and Theorem \ref{thm:U_powers of inclusion_maps} below for the precise statements. As a consequence, either of these theorems implies that $L7n2$ is not an $L$-space link, while Theorem \ref{thm: AlexPoly-of-L-space-link} fails to do so.

\subsection{Surgeries on $L$-space links.} Despite many algorithms on computing various versions of Heegaard Floer homology, explicit computations of plus/minus versions for 3-manifold invariants have only been done on a few cases, such as surgeries on knots and some mapping tori of surfaces, by exploiting surgery exact triangles.
In \cite{[Hom]L-space_surgeries}, Hom pointed out that the result from
 \cite{[OS]lens_space_surgery}  further implies that the whole package of Heegaard Floer homology of surgeries on an $L$-space knot $K$ is determined by the Alexander polynomial of $K$ and the surgery coefficients.

In this paper, we study the computation of Heegaard Floer invariants for integral surgeries on an $L$-space link $L$, including the completed Heegaard Floer homology $\bf{HF}^-$, absolute gradings, and the cobordism maps, using the link surgery formula of Manolescu-Ozsv\'{a}th from \cite{[MO]link_surgery}.  The Manolescu-Ozsv\'{a}th surgery complex is an object in the category of chain complexes of $\bb{F}[[U]]$-modules, while it can also be considered as an object in the homotopy category of chain complexes of $\bb{F}[[U]]$-modules.  In \cite{[YL]Surgery_2-bridge_link}, any representative in this chain homotopy equivalence class is called a \emph{perturbed surgery complex}. Some algebraic rigidity results are established in \cite{[YL]Surgery_2-bridge_link}, which imply that $\mathfrak{A}^-_{\mathrm{\bf{s}}}(L)$ is chain homotopic to $\bb{F}[[U]]$ by a $\bb{F}[[U]]$-linear chain map preserving the $\bb{Z}$-grading.

Thus, for an $L$-space link $L$, the perturbed surgery complex turns out to be largely simplified. When $L$ has 1 or 2 components, all the information needed in the perturbed surgery complex is completely determined by the Alexander polynomial and the surgery framing matrix.

\begin{thm}
\label{thm:L-space_surgery}
For a 2-component $L$-space link $\overrightarrow{L}=\overrightarrow{L_1}\cup \overrightarrow{L_2}$, all Heegaard Floer homology
${\bf{HF}}^-(S^3_\Lambda(L))$  together with the absolute gradings on them are determined by the following set of data:
\begin{itemize}
\item the multi-variable Alexander polynomial $\Delta_L(x,y)$,
\item the Alexander  polynomials $\Delta_{L_1}(t)$ and $\Delta_{L_2}(t)$,
\item the framing matrix $\Lambda$.
\end{itemize}
\end{thm}

\begin{rem}For $L$-space links with more components, besides the Alexander polynomials more information is needed to determine whether the higher diagonal maps vanish or not.
\end{rem}

Furthermore, we explicitly describe $\widehat{HF}$ of surgeries on an $L$-space link $L=L_1\cup L_2$ by a series of formulas in terms of the Alexander polynomials and the surgery framing matrix.  These formulas give a fast algorithm computing $\widehat{HF}$ of these surgeries. We also give a fast algorithm for classifying $L$-space surgeries. As an application, we study the classification of $L$-space surgeries on two-bridge links, and compute some examples explicitly: $(1,1)$-surgery on a family of $L$-space links with linking number zero, $L_n=b(4n^2+4n,-2n-1)$.

Instead of classifying $L$-space links with more than 2 components, we contend to show the prevalence of surgeries on $L$-space links among 3-manifolds:

\begin{question}
\label{ques:L-space_link_surgery}
Is every 3-manifold a surgery on  a (generalized) $L$-space link?
\end{question}

If Question \ref{ques:L-space_link_surgery} had a positive answer, one could hope to compute Heegaard Floer homology by $L$-space links. As a matter of fact, every 3-manifold $M$ can be realized by a surgery on an algebraically split link after connect sum with several lens spaces; see Corollary 2.5 from \cite{[Ohtsuki]Rational_homology_sphere}. It is also interesting to ask whether this algebraically split link can be chosen to be a generalized $L$-space link.

Regarding $L$-space surgeries, there is a more basic question:

\begin{question}
Is every $L$-space a surgery on  a (generalized) $L$-space link?
\end{question}

\subsection{Organization and conventions.}This paper is organized as follows.
In Section 2, we discuss the properties of $L$-space links and generalized $L$-space links. In Section 3, we present examples of $L$-space links and contrast them with $L$-space knots. Section 4 consists of the proof of Theorem \ref{thm: AlexPoly-of-L-space-link} and related discussions on fiberedness of $L$-space links. In Section 5, we prove Theorem \ref{thm:L-space_surgery}. In Section 6, we give the algorithm for computing $\widehat{HF}$ of surgeries on 2-component $L$-space links and compute some examples.

Since $L$-space links are sensitive to mirrors and the generalized Floer complexes are defined for oriented links, we describe our conventions about oriented two-bridge links $b(p,q)$ and oriented torus links $T(2,2n)$ in Section 3.  In addition, the Floer complex ${\bf{CF}}^-(S^3)$ is absolutely $\bb{Z}$-graded, and the top grading is 0. This convention is needed to compute the $d$-invariants from link surgery formula using minus version Floer complexes.

\subsection{Acknowledgements.}
I deeply appreciate Ciprian Manolescu for providing wonderful insights and for his continuing encouragements as my advisor. Some of the ideas in this paper have originated from Matthew Hedden, in particular, Theorem \ref{thm: AlexPoly-of-L-space-link}. Especially I wish to thank the referee for his careful reading and numerous suggestions and corrections. I am also grateful to Jiajun Wang for guiding me to $L$-space surgeries,  to Yi Ni for helpful conversations on fiberedness, to Tye Lidman for inspiring comments on chain links, to Thomas Mark and Bulent Tosun for inspiring comments on the Weeks manifold. I am informed that some results in this paper such as Theorem \ref{thm: AlexPoly-of-L-space-link} have been obtained by Nakul Dawra independently.

\section{$L$-space links}
In this section, we study the large $L$-space surgeries on a link $L$. Then, we introduce various notions of $L$-space links.

\subsection{$L$-space links.}
Let us recall the definition of generalized Floer complexes of a link $L$ in $S^3$ in \cite{[MO]link_surgery} Section 4, which govern the large surgeries on $L$. For simplicity, we only consider generic admissible multi-pointed Heegaard diagrams with each component $L_i$ having only two basepoints $w_i,z_i$. Here, we allow free basepoints.

\begin{defn}[Generalized Floer complexes]
\label{defn:A_s}
Let $L$ be a link in $S^3$ and choose a Heegaard diagram $\mathcal{H}$.
For $\text{s}\in\mathbb{H}(L)$, the \emph{generalized Floer complex}
$\mathfrak{A}^{-}(\mathcal{H},\text{s})$ is the free module over
$\mathcal{R}=\mathbb{F}[[U_{1},...,U_{l}]]$ generated by $\mathbb{T}_{\alpha}\cap\mathbb{T}_{\beta}\in \mathrm{Sym}^{g+k-1}(\Sigma)$,
and equipped with the differential:
\begin{equation}
\label{eq:A_s}
\partial_{\mathrm{\bf{s}}}^{-}{\bf x}=\underset{{\bf y}\in\mathbb{T}(\alpha)\cap\mathbb{T}(\beta)}{\displaystyle \sum}{\displaystyle \underset{\begin{array}{c}
\phi\in\pi_{2}({\bf x},{\bf y})\\
\mu(\phi)=1
\end{array}}{\displaystyle \sum}\#(\mathcal{M}(\phi)/\mathbb{R})}\cdot U_{1}^{E_{s_{1}}^{1}(\phi)}\cdots U_{l}^{E_{s_{l}}^{l}(\phi)}\cdot U_{l+1}^{n_{w_{l+1}}(\phi)}\cdots U_{k}^{n_{w_{k}}(\phi)}\cdot {\bf y},
\end{equation}
where $E_{s}^{i}(\phi)$ is defined by
\begin{align}
E_{s}^{i}(\phi)&=\max\{s-A_{i}({\bf x}),0\}-\max\{s-A_{i}({\bf y}),0\}+n_{z_{i}}(\phi)\\
&=\max\{A_{i}({\bf x})-s,0\}-\max\{A_{i}({\bf y})-s,0\}+n_{w_{i}}(\phi).\label{eq:E_s}
\end{align}
Here, $\mathcal{M}(\phi)$ denotes the moduli space of the pseudo-holomorphic disks in the homotopy class $\phi$, and $A_i({\bf x})$ denotes the $i$th Alexander grading of the intersection point $\bf x.$
The stable quasi-isomorphism type of $\mathfrak{A}^-({\mathcal H},{\mathrm{\bf{s}}})$ is an invariant of $L$.  For simplicity, we also write
$\mathfrak{A}^-(L,{\mathrm{\bf{s}}})$, $\mathfrak{A}^-_{\mathrm{\bf{s}}}(L)$, or $\mathfrak{A}^-_{\mathrm{\bf{s}}}$, when the context is clear.
\end{defn}

\begin{notation}
Let $L$ be an $l$-component link in $S^3$. In order to simplify the notation, we denote the $(p_1,...,p_l)$-surgery on $L$ by $S^3_{p_1,...,p_l}(L)$ and the surgery framing matrix by $\Lambda(p_1,...,p_l)$, where $p_1,...,p_l$ are surgery coefficients on the link; i.e. $\Lambda(p_1,...,p_l)$ is the matrix with $p_1,...,p_l$ on the diagonal and linking numbers off the diagonal.
\end{notation}

\begin{proof}[Proof of Lemma \ref{lem:L-space_sublinks}]
First, let us recall Theorem 10.1 in \cite{[MO]link_surgery}.
\begin{thm}
Let $\tilde{\Lambda}$ be a surgery framing on the link $L$.
For $\tilde{\Lambda}$ sufficiently large, there exist quasi-isomorphisms of relative $\bb{Z}$-graded complexes
\[{\bf{CF}}^-(S^3_{\tilde{\Lambda}}(L),\mathrm{\bf s})\to {\mathfrak A}^-_{\mathrm{\bf s}}(L)\]
for all $\mathrm{\bf s}$.
\end{thm}
Thus, $L$ is an $L$-space link if and only if $\mathfrak{A}^-_{\mathrm{\bf{s}}}(L)$ has the homology $\mathbb{F}[[U]]$ for all ${\mathrm{\bf{s}}}\in \bb{H}(L)$.  When the $i$th component of $\mathrm{\bf{s}}$, say $s_i$, equals to $+\infty$, there is a destabilization map between $\mathfrak{A}^-(L,{\mathrm{\bf{s}}})$ and $\mathfrak{A}^-(L-L_i,\psi^{+L_i}({\mathrm{\bf{s}}}))$, which is a quasi-isomorphism. See Example 7.2 in \cite{[MO]link_surgery}. Roughly, this is because the generalized Floer complexes of $L-L_i$ can be computed from the Heegaard diagram of $L$ by deleting the basepoint $z_i$, which is the same as putting $s_i=+\infty$ in $\mathfrak{A}^-(L,\mathrm{\bf{s}})$.   Thus,  $\mathfrak{A}^-(L-L_i,{\mathrm{\bf{s}}'})$ has homology $\bb{F}[[U]]$ for all $\mathrm{\bf{s}'}\in \bb{H}(L-L_i)$. So $L-L_i$ is an $L$-space link for $L_i\subset L.$ An induction will show that all sublinks are $L$-space links.
\end{proof}

In contrast to knots, a weak $L$-space link $L$ might be a non-$L$-space link.

\begin{example}
\label{eg:figure8}
Let $L=L_1\cup L_2$ be the link consisting of a Figure-8 knot $L_1$ and an unknot $L_2$ as in Figure \ref{fig:Figure 8}. Then by blowing down the unknot, the Figure-8 knot is then unknotted, and thus the surgery $S^3_{n,1}(L)$ is the lens space $L(n-4,1)$, when $n\neq 4$. However, the Figure-8 knot is not an $L$-space knot. Thus, by Lemma \ref{lem:L-space_sublinks}, $L$ is a weak $L$-space link but not an $L$-space link. Similarly, the mirror of $L$ is not a $L$-space link neither.
\end{example}

\begin{figure}
  \centering
  \centering
  \includegraphics[scale=0.3]{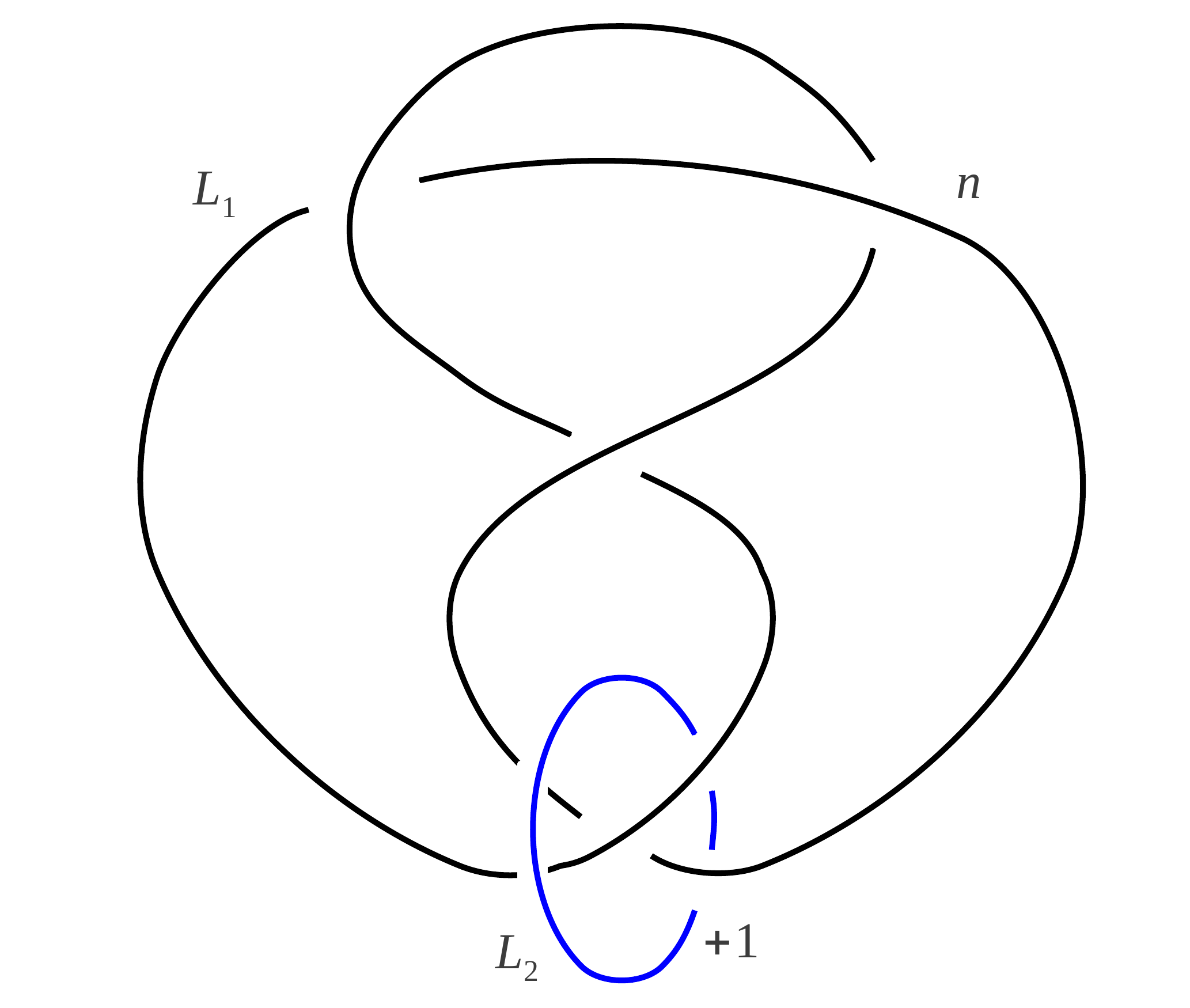}
  \caption{\textbf{An example of weak $L$-space link.} }
  \label{fig:Figure 8}
\end{figure}

\subsection{$L$-space induction and generalized large surgeries.}
In this subsection, we study how to characterize $L$-space links, by exploiting surgery exact triangles.

\begin{lem}[$L$-space surgery induction]
\label{lem:L-space_surgery_iteration}
Let $L=L_1\cup...\cup L_n$ be a link with $n$ components, and  $L'=L-L_1$. Let $\Lambda$ be the framing matrix of $L$ for the surgery $S^3_{p_1,...,p_n}(L)$, and denote by $\Lambda'$ the restriction of $\Lambda$ on $L'$. Suppose $S^3_{p_1,...,p_n}(L)$ and $S^3_{p_2,...,p_n}(L')$ are both $L$-spaces. Then,
\begin{description}
\item[Case I] if $\det(\Lambda)\cdot\det(\Lambda')>0$, then for all $k>0$, $S^3_{p_1+k,p_2,...,p_n}(L)$ is an $L$-space;
\item[Case II] if $\det(\Lambda)\cdot\det(\Lambda')<0$, then for all $k>0$, $S^3_{p_1-k,p_2,...,p_n}(L)$ is an $L$-space.
\end{description}
\end{lem}

\begin{proof}Let $\Lambda_k$ be the framing matrix of the surgery $S^3_{p_1+k,p_2,...,p_n}(L)$. Notice that $\det(\Lambda_k)=\det(\Lambda)+k\det(\Lambda')$.

For the case  $\det(\Lambda)\cdot \det(\Lambda')>0$, consider the following exact triangle of surgeries:
$$\xymatrix{\HFh(\surg{p_1}{p_2,...,p_n})\ar[rr] & &\HFh(\surg{p_1+1}{p_2,...,p_n})\ar[ld] \\
                                &\HFh(S^3_{p_2,...,p_n}(L')).\ar[lu] &}$$
Thus, from that $\det(\Lambda_1)=\det(\Lambda)+\det(\Lambda')$,  it follows that  $\surg{p_1+1}{p_2,...,p_n}$ is also an $L$-space.  Iterating this argument for all $k>0$, we can obtain that $S^3_{p_1+k,p_2,...,p_n}(L)$ is an $L$-space for all $k\geq0$.  The case where $\det(\Lambda)\cdot \det(\Lambda')<0$ is similar.
\end{proof}

\begin{lem}[Positive $L$-space surgery criterion]
\label{lem:positive_surgery_criterion}
An $l$-component link $L$ is an $L$-space link if and only if
there exists a surgery framing $\Lambda(p_1,...,p_l)$, such that
for all sublink $L'\subseteq L$, $\det(\Lambda(p_1,...,p_l)|_{L'})>0$
and $S^3_{\Lambda|_{L'}}(L')$ is an $L$-space.

In particular, if the surgery framing  $\Lambda(p_1,...,p_l)$ satisfies the above condition, then for any
surgery framing $\Lambda'=\Lambda(n_1,...,n_l)$ with $n_i\geq p_i$ for all $i$, the surgery $S^3_{\Lambda'}(L)$ is an $L$-space.
\end{lem}

\begin{proof}
If $L$ is an $L$-space link, then every sublink $L'$ is an $L$-space link, by Lemma \ref{lem:L-space_sublinks}. Thus,
there is a large $(p_1,...,p_l)$-surgery on $L$ such that for all $L'\subseteq L$, $\det(\Lambda(p_1,...,p_l))>0$ and $S^3_{\Lambda|_{L'}}(L')$ is an $L$-space.

Conversely, let $\Lambda(p_1,...,p_l)$ be the surgery framing satisfying the condition in the proposition.
Let $\Lambda'=\Lambda(p_1,...,p_i+1,...,p_l)$. By the $L$-space surgery induction lemma, we have that for all $L'\subseteq L$,
$S^3_{\Lambda'|_{L'}}(L')$ is an $L$-space and $\det(\Lambda'|_{L'})=\det(\Lambda|_{L'})+\varepsilon\det(\Lambda|_{L'-L_i})$,
where $\varepsilon=1$ if $L_i\subset L'$ and $\varepsilon=0$ if $L_i\nsubseteq L'$. Thus, by induction, we can show that for any surgery framing $\Lambda''=\Lambda(n_1,...,n_l)$ with $n_i\geq p_i$, the surgery $S^3_{\Lambda''|_{L'}}(L')$ is an $L$-space for all sublinks $L'\subset L$. In particular, $S^3_{\Lambda''}(L)$ is an $L$-space, and this finishes the proof.
\end{proof}

\begin{defn}
A link is called \emph{algebraically split}, if all the pairwise linking numbers are $0$.
\end{defn}

\begin{cor}
\label{cor:algebraic_split_links}
Let $L=L_1\cup...\cup L_l$ be an algebraically split link.
Then $L$ is an $L$-space link if and only if $\exists\  p_i>0,i=1,...,l$, such that
$S^3_{\Lambda|_{L'}}(L')$ is an $L$-space for all $L'\subseteq L$, where $\Lambda=\Lambda(p_1,...,p_l)$.
\end{cor}

Now we can prove Proposition \ref{prop:equivalent_defn_L-space_links}.

\begin{proof}[Proof of Proposition \ref{prop:equivalent_defn_L-space_links}]
Lemma \ref{lem:positive_surgery_criterion} implies that condition (i) and (ii) are equivalent. By Theorem 10.1
from \cite{[MO]link_surgery}, it follows $L$ is an $L$-space link if and only if $\mathfrak{A}^-_{\mathrm{\bf{s}}}(L)$ has homology
$\bb{F}[[U]]$ for all $\mathrm{\bf{s}}\in \bb{F}[[U]]$. Thus, condition (i) is equivalent to (iii) as well as (iv).
\end{proof}

\subsection{Generalized $L$-space links.}
We can enlarge our scope to generalized large surgeries on a link $L$.
Let us use $\pm$ signs to denote the type of the generalized large surgeries.

\begin{defn}[Generalized $L$-space links]
\label{defn:+-L-space links}
A $2$-component  link $L=L_1\cup L_2$ is called a \emph{generalized $(\pm\pm)L$-space link},
 if there exist integers $p_1,p_2$, such that for all positive integers $k_1,k_2> 0$, $S^3_{p_1\pm k_1,p_2\pm k_2}(L)$ is
an $L$-space. Similarly, we define an $l$-component \emph{generalized $(\pm\cdots\pm)L$-space link}.
\end{defn}

\begin{example} The split disjoint union of the left-handed trefoil and the right-handed trefoil is a generalized $(+-)L$-space link.
However, it is not an $L$-space link, and neither is its mirror.
\end{example}

Let us look at some examples of $2$-component generalized $L$-space links.

\begin{prop}
\label{prop:2-component Link surgery}
Suppose $L$ is a 2-component link $L=L_1\cup L_2$ with $L_1,L_2$ both being the unknots, and $\surg{p_1}{p_2}$ is an $L$-space. Then,
\begin{enumerate}
  \item if $p_1 p_2>\mathrm{lk}^2,p_1>0,p_2>0$, then $\surg{p_1+k_1}{p_2+k_2}$ are $L$-spaces for all $k_1,k_2\in \bb{N}$;
  \item if $p_1 p_2>\mathrm{lk}^2,p_1<0,p_2<0$, then $\surg{p_1-k_1}{p_2-k_2}$ are $L$-spaces for all $k_1,k_2\in \bb{N}$;
  \item if $p_1>0,p_2<0$, then $\surg{p_1+k_1}{p_2-k_2}$ are $L$-spaces for all $k_1,k_2\in \bb{N}$;
  \item if $p_1 p_2<\mathrm{lk}^2,p_1>0,p_2>0$, then the surgeries $\surg{p_1+k_1}{-1-k_2},\surg{-1-k_1}{p_2+k_2}$ with $k_1\geq0,k_2\geq0$ and $\surg{p'_1}{p'_2}$ with $0<p'_1\leq p_1,0<p'_2 \leq p_2$ are all $L$-spaces;
  \item if $p_1 p_2<\mathrm{lk}^2,p_1<0,p_2<0$, then the surgeries $\surg{p_1-k_1}{k_2},\surg{k_1}{p_2-k_2}$ with $k_1>0,k_2>0$ and $\surg{p'_1}{p'_2}$ with $0>p'_1\geq p_1,0>p'_2\geq p_2$ are all $L$-spaces.
\end{enumerate}
The above cases are shown in Figure \ref{fig:two_unknots}.
\end{prop}

\begin{figure}
  \centering
  \centering
  \includegraphics[scale=0.4]{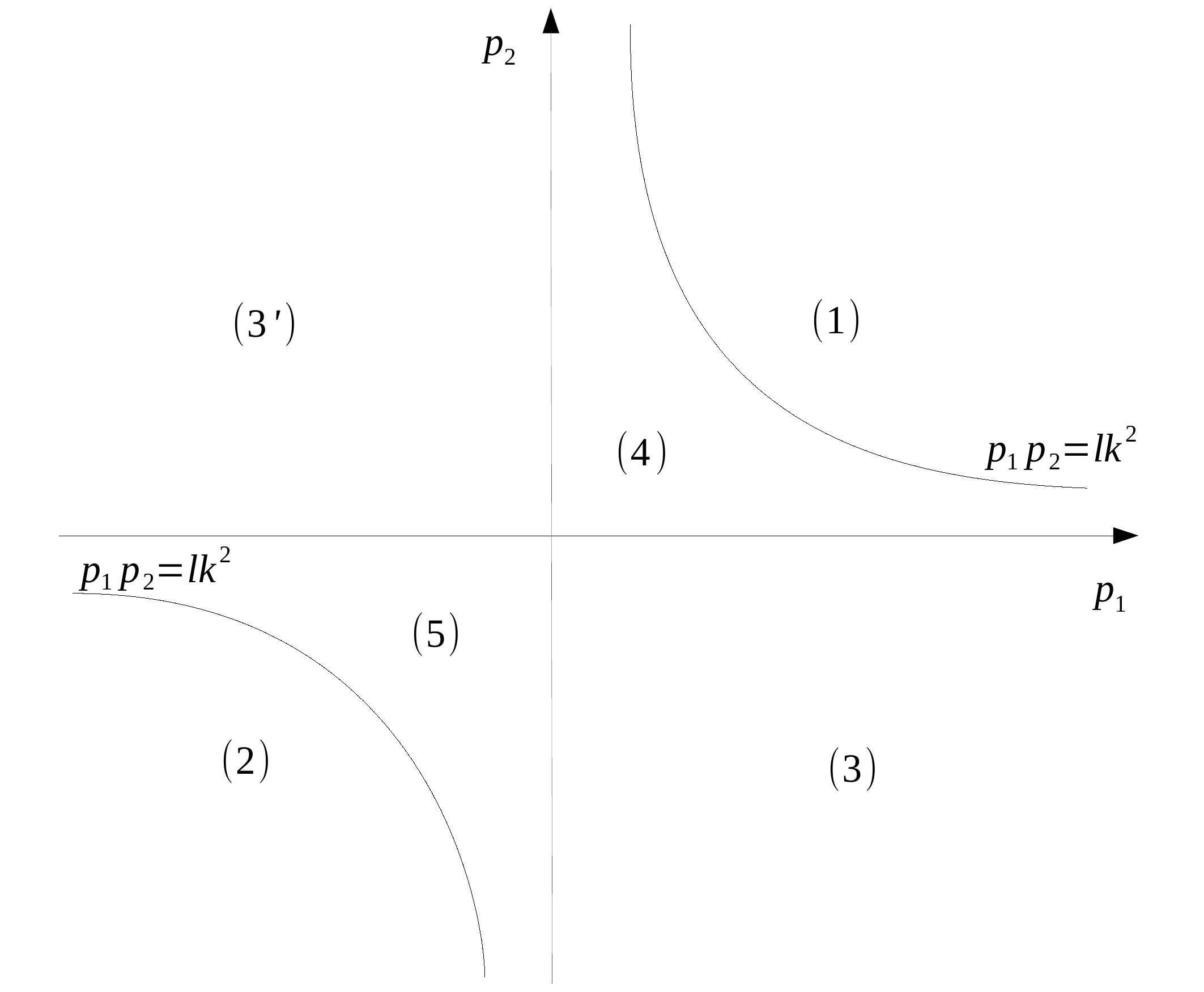}
  \caption{We illustrate the cases of the $(p_1,p_2)$-surgeries in Proposition \ref{prop:2-component Link surgery} on the $(p_1,p_2)$ plane, where the case $(3')$ is similar to case $(3)$.   }
  \label{fig:two_unknots}
\end{figure}

\begin{proof}
The cases (1), (2), and (3) are proved by induction using the long exact sequences for the surgery triple $\big(\surg{p}{q},\surg{p+1}{q},S^3_{q}(L_2)\big)$.

For the case (4), first by Lemma \ref{lem:L-space_surgery_iteration}, we have that
$\surg{p_1}{-1},\surg{-1}{p_2}$ are both $L$-space spaces. From (3), it follows
that $\surg{p_1+k_1}{-1-k_2},\surg{-1-k_1}{p_2+k_2}$ are all $L$-spaces for all
non-negative integers $k_1,k_2.$ Second, we can do induction to prove that
$\surg{p'_1}{p'_2}$ with $0<p'_1\leq p_1,0<p'_2 \leq p_2$ are all
$L$-spaces. The case (5) is similar to the case (4).
\end{proof}

Proposition \ref{prop:2-component Link surgery} implies that if $L$ is a 2-component link
with unknotted components, then $L$ is a weak $L$-space link if and only if $L$ is a generalized $L$-space link. The following proposition gives another example of generalized $L$-space links.

\begin{prop}
Let $L$ be an algebraically split link.  If there exists a surgery framing $\Lambda(p_1,...,p_l)$ on $L$, such that
for any sublink $L'\subseteq L$, $S^3_{\Lambda|_{L'}}(L')$ is an $L$-space, then $L$ is a generalized $L$-space link of
"$\epsilon_1\cdots \epsilon_l$"-type, where $\epsilon_i$ is the sign of $p_i$.
\end{prop}

\section{Examples of $L$-space links and generalized $L$-space links}
In this section, we use the lemmas and propositions in Section 2 to  show some examples of $L$-space links and generalized $L$-space links.

\begin{example}[Two hyperbolic links:  the Whitehead link and the Borromean rings]
\label{eg:whitehead and Borromean}
The Whitehead link and the Borromean rings are two well-known hyperbolic links. In fact, they are both $L$-space links.

The $(1,1)$-surgery on the Whitehead link is the Poincar\'{e} sphere. See Example 8 on Page 263 in \cite{[Rolfsen]}. The $(1,1,1)$-surgery on the Borromean rings is also the Poincar\'{e} sphere. See Exercise 4 on Page 269 in \cite{[Rolfsen]}.
By Corollary \ref{cor:algebraic_split_links}, they are both $L$-space links.
\end{example}

\begin{rem}
There are no alternating hyperbolic $L$-space knots. See Theorem \ref{thm:alternating_L-space_knot} below cited from \cite{[OS]lens_space_surgery}. However, Example \ref{eg:whitehead and Borromean} shows that there are $L$-space alternating hyperbolic links. In fact, there are many, see Theorem \ref{thm:two-bridge_L-space_links}.

Moreover, these hyperbolic links provide many examples of hyperbolic $L$-spaces which are neither surgery over knots nor double branched cover over knots. For example, surgeries on the Whitehead link $S^3_{n,2n}(\mathit{Wh})$ with $n>0$ are all $L$-spaces but not surgeries nor double branched cover on a knot. The reason is that the first homology of these surgeries is neither cyclic nor of odd order.
\end{rem}

\begin{figure}
  \centering
  \centering
  \includegraphics[scale=0.35]{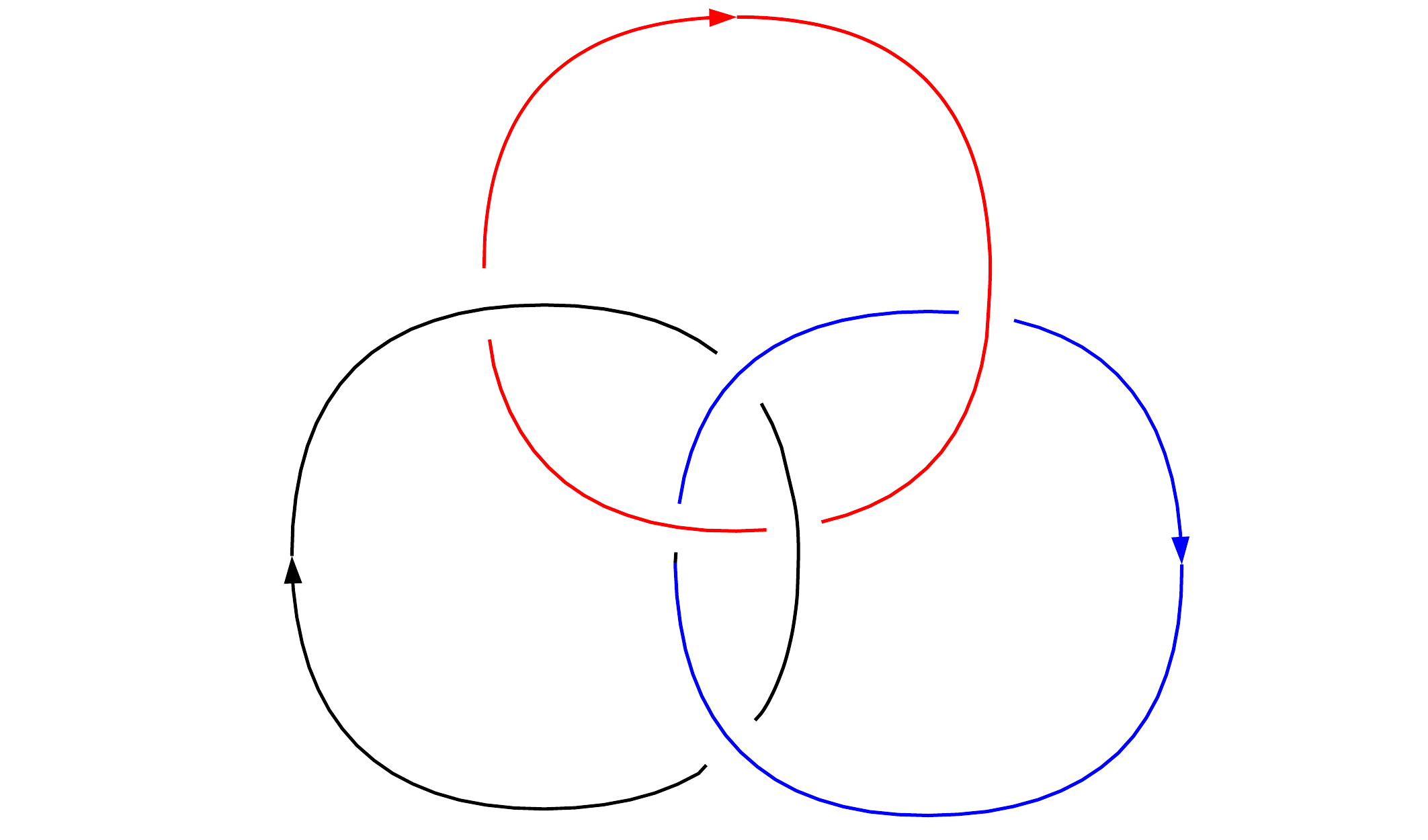}
  \caption{\textbf{The Borromean ring.} The $(1,1,1)$-surgery on the Borromean link is the Poincar\'{e} sphere.}
  \label{fig:borromean ring}
\end{figure}

\begin{example}[An $L$-space link providing the Weeks manifold]
\label{eg:Weeks_mfld}
Consider the link $L=L_1\cup L_2\cup L_3$ in Figure \ref{fig:Link_for_Weeks_mfld}, where $L_1\cup L_2$ is the Whitehead link (using the convention in \cite{[Rolfsen]}) and $L_3$ is the meridian of $L_2$.  The $(1,2,1)$-surgery is the Poincar\'{e} sphere, and it satisfies  the positive $L$-space surgery criterion. Thus, it is an $L$-space link.

By Lemma \ref{lem:positive_surgery_criterion}, we have that for any $n_1\geq1, n_2\geq 2, n_3\geq 1$, the $(n_1,n_2,n_3)$-surgery on $L$ is an $L$-space. Thus, the $(5,3,2)$-surgery is an $L$-space, which is the $(5,5/2)$-surgery on the Whitehead link. This surgery is the Weeks manifold; see \cite{[Calegari-Dunfield]Laminations&groups_of_Hom(S^1)}. The Weeks manifold has the smallest hyperbolic volume among closed hyperbolic 3-manifolds; see \cite{[Gabai-Meyerhoff-Milley]Weeks_manifold}. Thus, we confirm that the Weeks manifold does not admit a taut foliation.

The fact that the Weeks manifold is an $L$-space was already known by experts such as \cite{[KMOS]Monopole_lens_space} and  \cite{Low-dim-top-blog}.
\end{example}

\begin{figure}
  \centering
  \centering
  \includegraphics[scale=0.35]{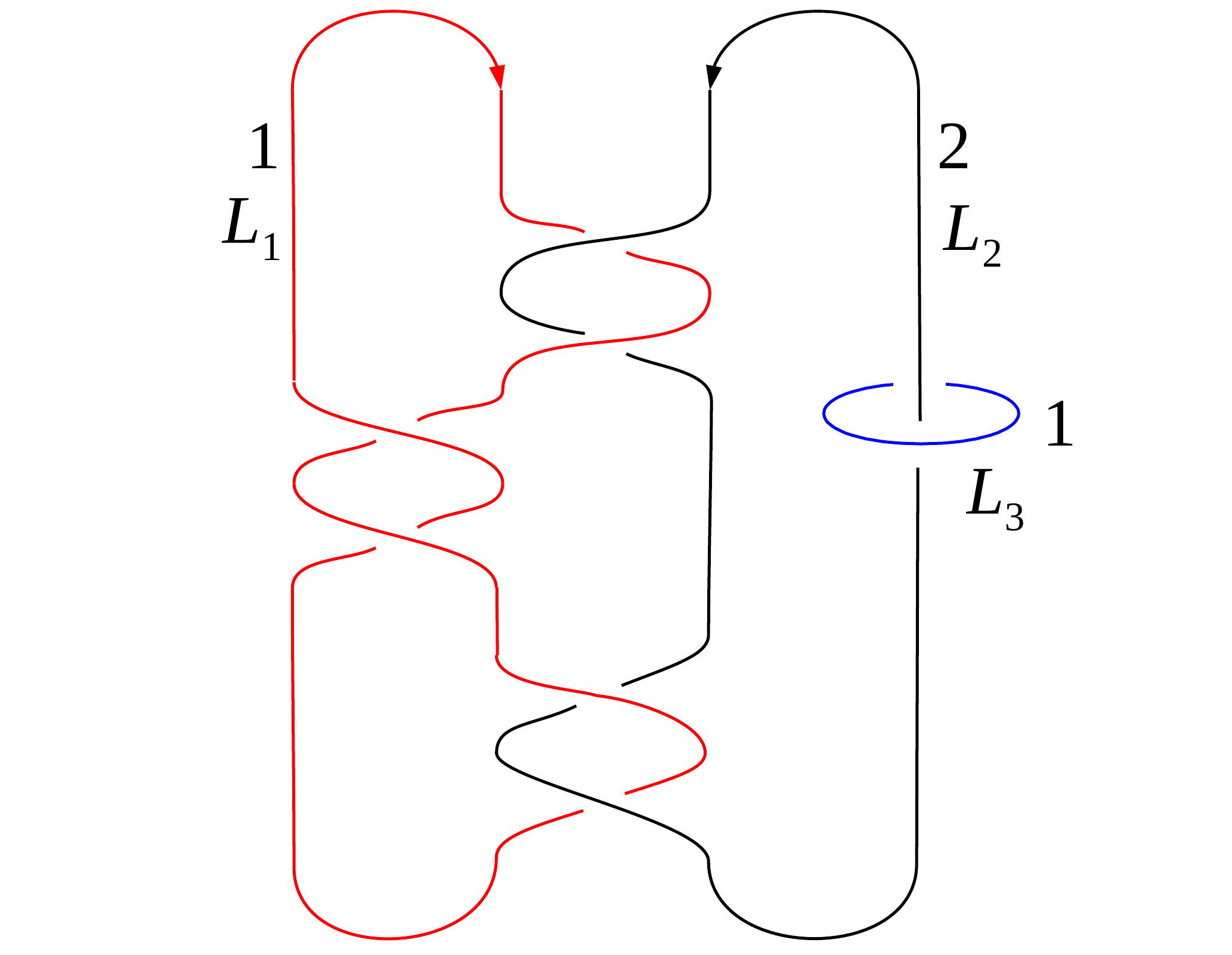}
  \caption{\textbf{An $L$-space link giving the Weeks manifold.} }
  \label{fig:Link_for_Weeks_mfld}
\end{figure}

\begin{example}[$T(2,2n)$ torus links] The oriented torus links $T(2,2n)$  are $L$-space links as Corollary \ref{cor:L-space surgery on T(2,2n)} below shows. We need to distinguish them from their mirrors, so see Figure \ref{fig:(n+1,n-1)surgery_on T(2,2n)} for the precise definitions of $T(2,2n)$.
\end{example}

\begin{figure}
  \centering
  \centering
  \includegraphics[scale=0.5]{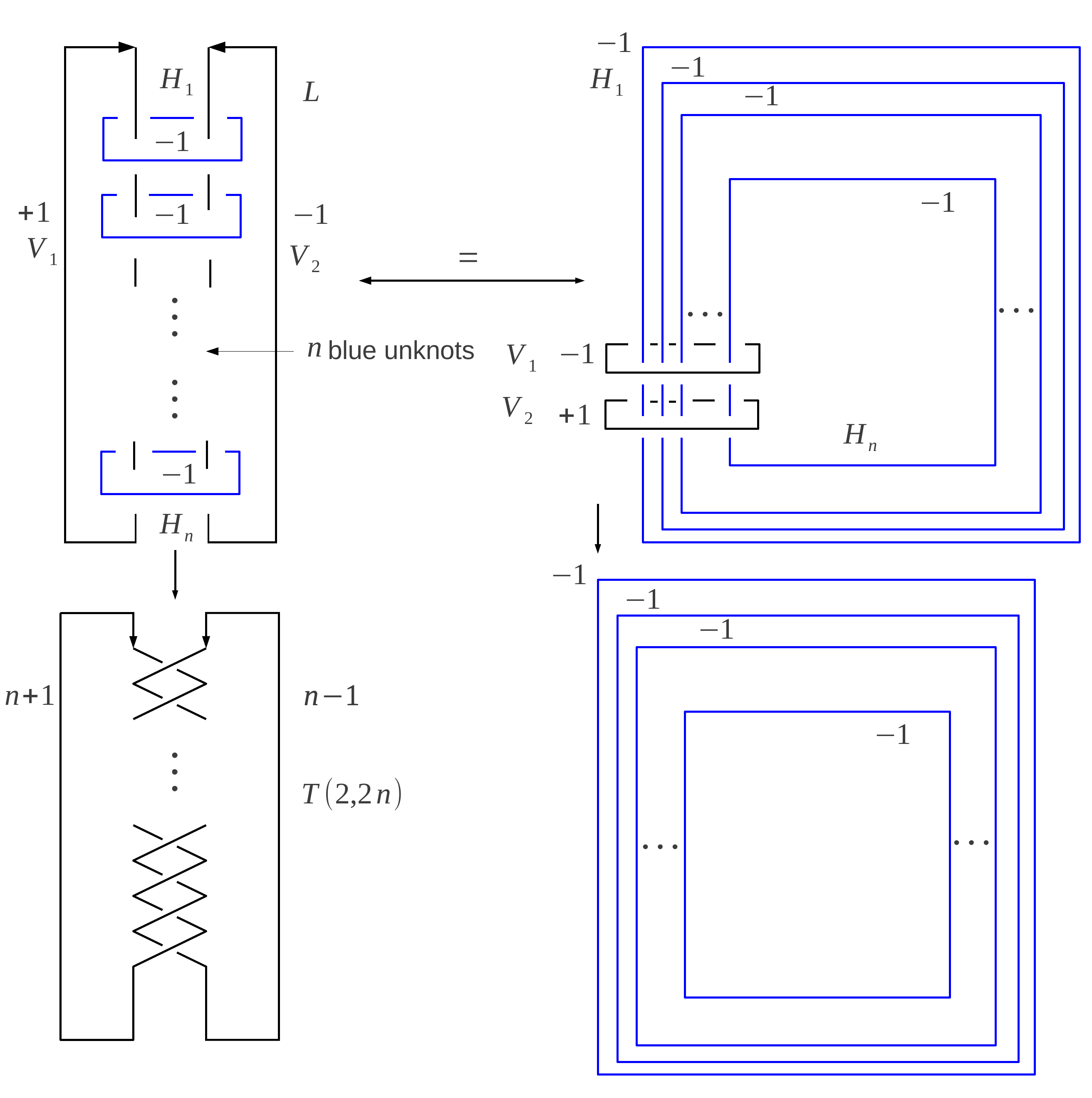}
  \caption{\textbf{The $(n+1,n-1)$-surgery on $T(2,2n)$.} Consider the surgery on the upper-left link $L$, which is a plumbing of unknots. By blowing down the horizontal unknots $H_i$'s, we get the surgery on the lower-left link $T(2,2n)$. While blowing down the
  black unknots $V_j$'s, we can get the surgery on the lower-right link, which is $S^3.$ }
  \label{fig:(n+1,n-1)surgery_on T(2,2n)}
\end{figure}

\begin{figure}
\centering
\includegraphics[scale=0.4]{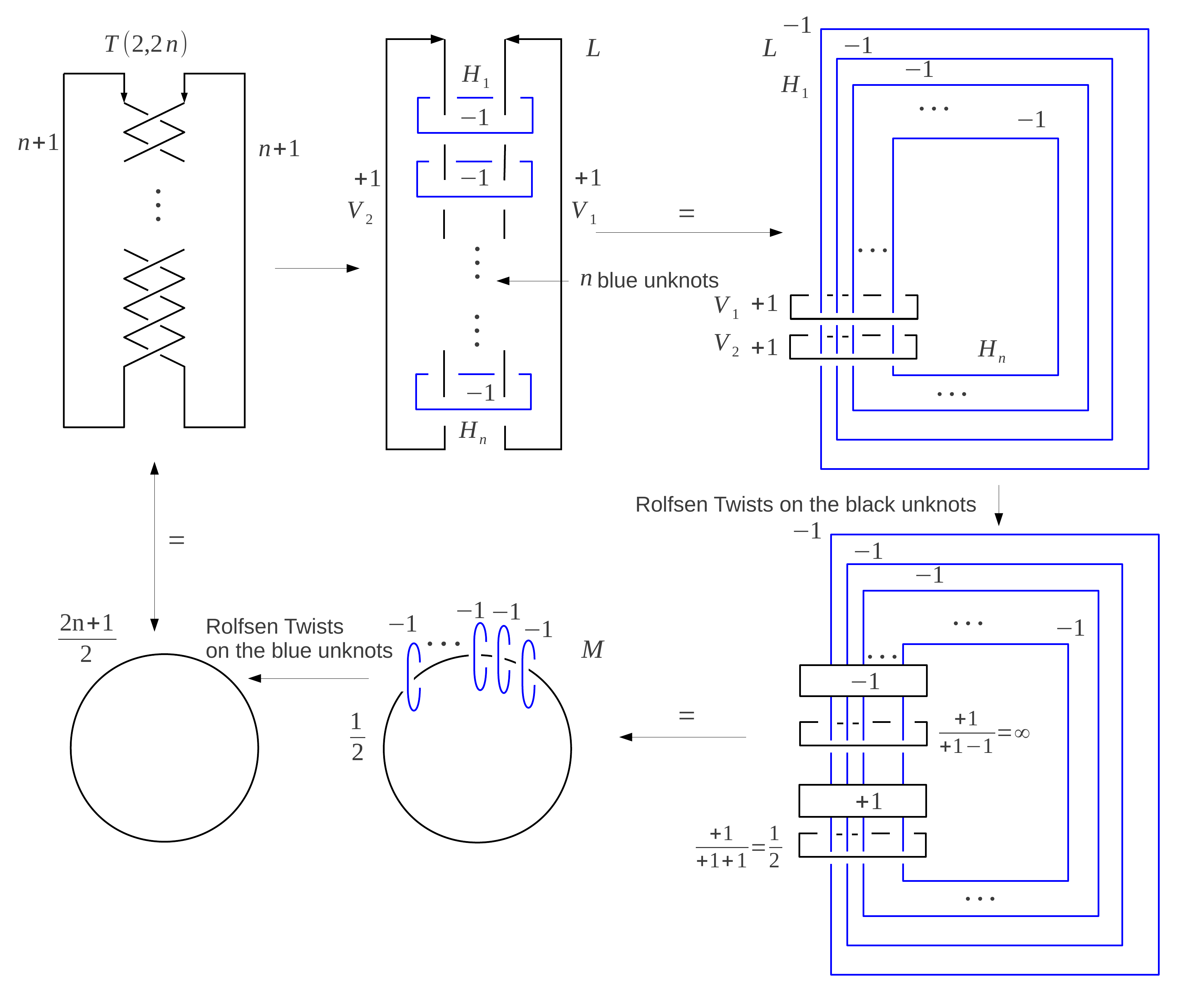}
\caption{\textbf{The $(n+1,n+1)$-surgery on the $T(2,2n)$ torus link.} Consider the
surgery on upper-middle link $L$, which is a plumbing of unknots. After
blowing down the horizontal (blue) unknots $H_i$'s, we get the $(n+1,n+1)$-surgery on
the upper-left link $T(2,2n)$. While after doing
Rolfsen twists on the black unknots $V_j$'s, we can get a rational surgery on the lower-middle link $M$, which is a lens space by blowing-down the blue unknots using Rolfsen twists again.}
\label{fig:(n+1,n+1)-surgery_on T(2,2n)}
\end{figure}

\begin{lem}
\label{lem:T(2,2n)}
For the torus links $T(2,2n)$, we have the following identifications of surgeries $$S^3_{n+1,n-1}(T(2,2n))=S^3,\ S^3_{n+1,n+1}(T(2,2n))=L(2n+1,2),\ S^3_{n,n+1}(T(2,2n))=L(n,1).$$
\end{lem}

\begin{proof}
First, for the $(n+1,n-1)$-surgery on $T(2,2n)$, we consider a surgery on the upper-left link $L$ in Figure \ref{fig:(n+1,n+1)-surgery_on T(2,2n)}, where $L$ is a plumbing of unknots. After two different blowing-down procedures, we get the identification of $S^3_{n+1,n-1}(T(2,2n))$ with $S^3$.

Second, for the $(n+1,n+1)$-surgery on $T(2,2n)$, we similarly consider a different surgery on $L$, which is drawn in Figure \ref{fig:(n+1,n+1)-surgery_on T(2,2n)}. After two different processes of doing Rolfsen twists, we can obtain the identification of $S^3_{n+1,n+1}(T(2,2n))$ with $L(2n+1,2)$. See Figure \ref{fig:(n+1,n+1)-surgery_on T(2,2n)}.
As is similar to the $(n+1,n+1)$-surgery, the $(n,n+1)$-surgery is $L(n,1)$.
\end{proof}

\begin{cor}
\label{cor:L-space surgery on T(2,2n)}
The following surgeries on the torus link $T(2,2n)$ are all $L$-spaces:
\begin{itemize}
\item $S^3_{n+1+k_1,n+1+k_2}(T(2,2n))$, $\forall k_1\geq 0,\forall k_2\geq 0$,
\item $S^3_{n+1-k_1,n-1}(T(2,2n))$, $\forall k_1\geq 0$,
\item $S^3_{-1-k_1,n-1+k_2}(T(2,2n))$, $\forall k_1\geq 0,\forall k_2\geq 0$,
\item $S^3_{n,q}(T(2,2n))$ with $q\neq n$.
\end{itemize}
\end{cor}

\begin{proof}
We combine Proposition \ref{prop:2-component Link surgery} and Lemma \ref{lem:T(2,2n)}.

From $S^3_{n+1,n+1}(T(2,2n))=L(2n+1,2)$, it follows that $S^3_{n+1+k_1,n+1+k_2}(T(2,2n))$ are all $L$-spaces for $k_1,k_2\geq0.$

From $S^3_{n+1,n-1}(T(2,2n))=S^3$, it follows that $S^3_{n+1-k_1,n-1}(T(2,2n))$ are all $L$-spaces by Lemma \ref{lem:L-space_surgery_iteration}. Thus, $(-1,n-1)$-surgery is an $L$-space, and so is any $S^3_{-1-k_1,n-1+k_2}(T(2,2n))$ with $k_1,k_2\geq 0$.

From $S^3_{n+1,n-1}(T(2,2n))=S^3$, it follows that $(n,n-1)$-surgery is an $L$-space
and thus all $(n,q)$-surgeries with $q\leq n-1$ are $L$-spaces.

From $S^3_{n,n+1}(T(2,2n))=L(n,1),$ it follows that all $(n,q)$-surgery with $q\geq n+1$ are $L$-spaces.
\end{proof}

\begin{example}[Algebraic links]
Gorsky and N\'{e}methi showed in  \cite{[Gorsky_Nemithi]algebriac_links} that every algebraic link is an $L$-space link.  For example, torus links are algebraic links. In \cite{[Gorsky_Nemithi]algebriac_links}, they also classify all the $L$-space surgeries on the $T(pr,qr)$ torus links, with $p,q\geq2,r\geq1.$ We also describe all possible $L$-space surgeries on the $T(2,2n)$ torus links; see Proposition \ref{prop:classification of L-space surgeries on T(2,2n)}.
\end{example}

\begin{figure}
\centering
\includegraphics[scale=0.3]{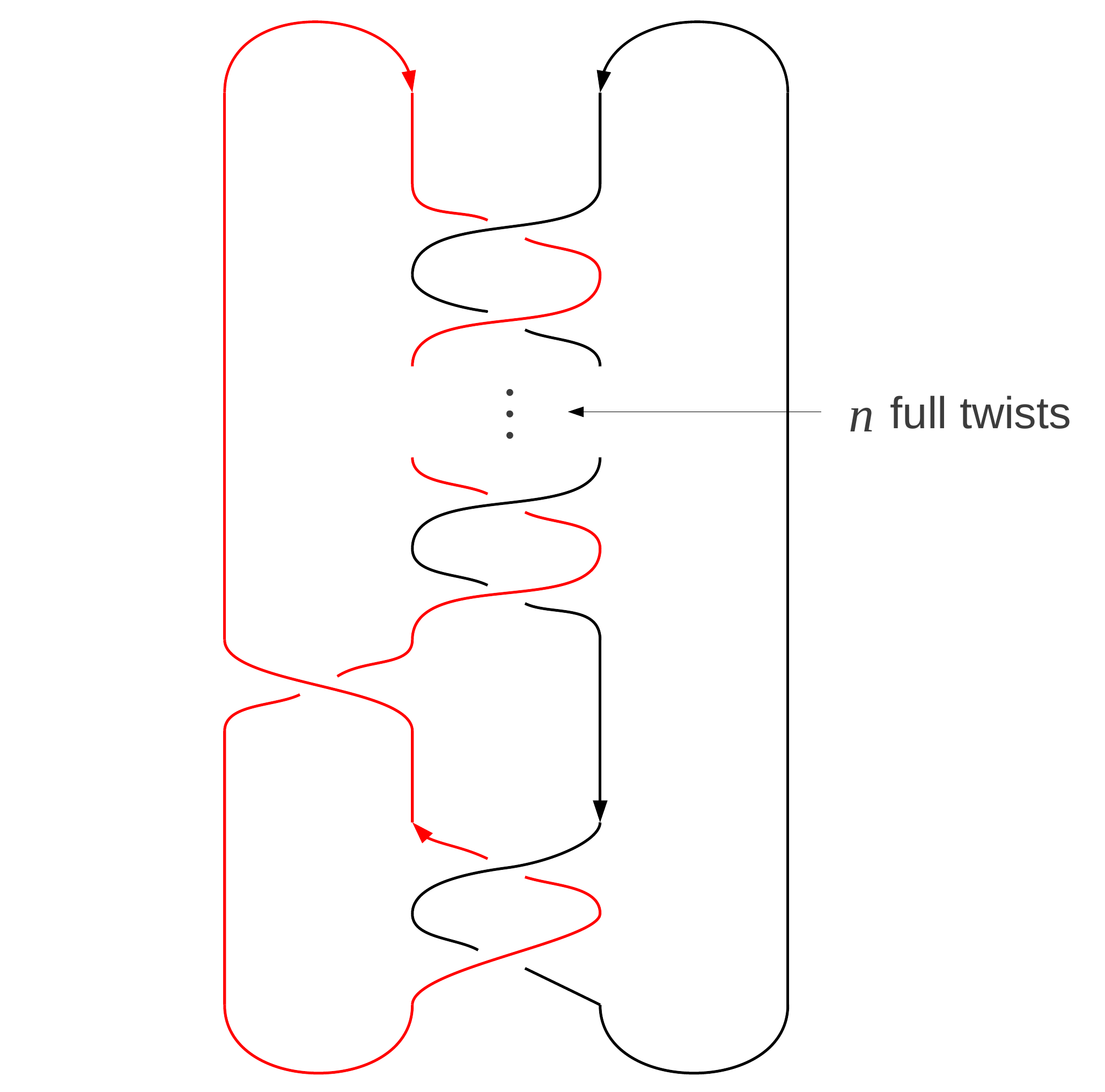}
\caption{\textbf{The two-bridge link $b(6n+2,-3)$.} }
\label{fig:b(6n+2,-3)}
\end{figure}

The following theorem provides an infinite set of two-bridge links, which are hyperbolic $L$-space links.
Before proving this theorem, let us clarify some conventions  for two-bridge links. First, the notation $b(p,q)$
denotes an oriented two-bridge link of slope $\frac{q}{p}$. For any
 continued fraction of $\frac{q}{p}:$
\[\frac{q}{p}=[a_1,a_2,...,a_m]=\cfrac{1}{a_1+\cfrac{1}{\cdots{\begin{array}{c}
\\
+\cfrac{1}{a_{m-1}+\cfrac{1}{a_m}}
\end{array}}}},
\]
a 4-plat projection of $b(p,q)$ can be obtained in the following ways:
\begin{description}
 \item[Case I]  If $m$ is odd, then the 4-plat is obtained by closing the 4-braid
 \[B=\sigma_2^{a_1}\sigma_1^{-a_2}\cdots \sigma_2^{a_m}\]
 in the way shown in Figure \ref{fig:4-plat}(a).
 \item[Case II]  If $m$ is even, then the 4-plat is obtained by closing the 4-braid
  \[B=\sigma_2^{a_1}\sigma_1^{-a_2}\cdots \sigma_1^{-a_m}\]
 in the way shown in Figure \ref{fig:4-plat}(b).
\end{description}
Here, we follow \cite{[burde_Zieschang]knots} Chapter 12B. We prescribe an orientation on $b(p,q)$
shown in Figure \ref{fig:4-plat}. Note that this orientation convention is different from \cite{[burde_Zieschang]knots}.

\begin{figure}
\centering
\includegraphics[scale=0.5]{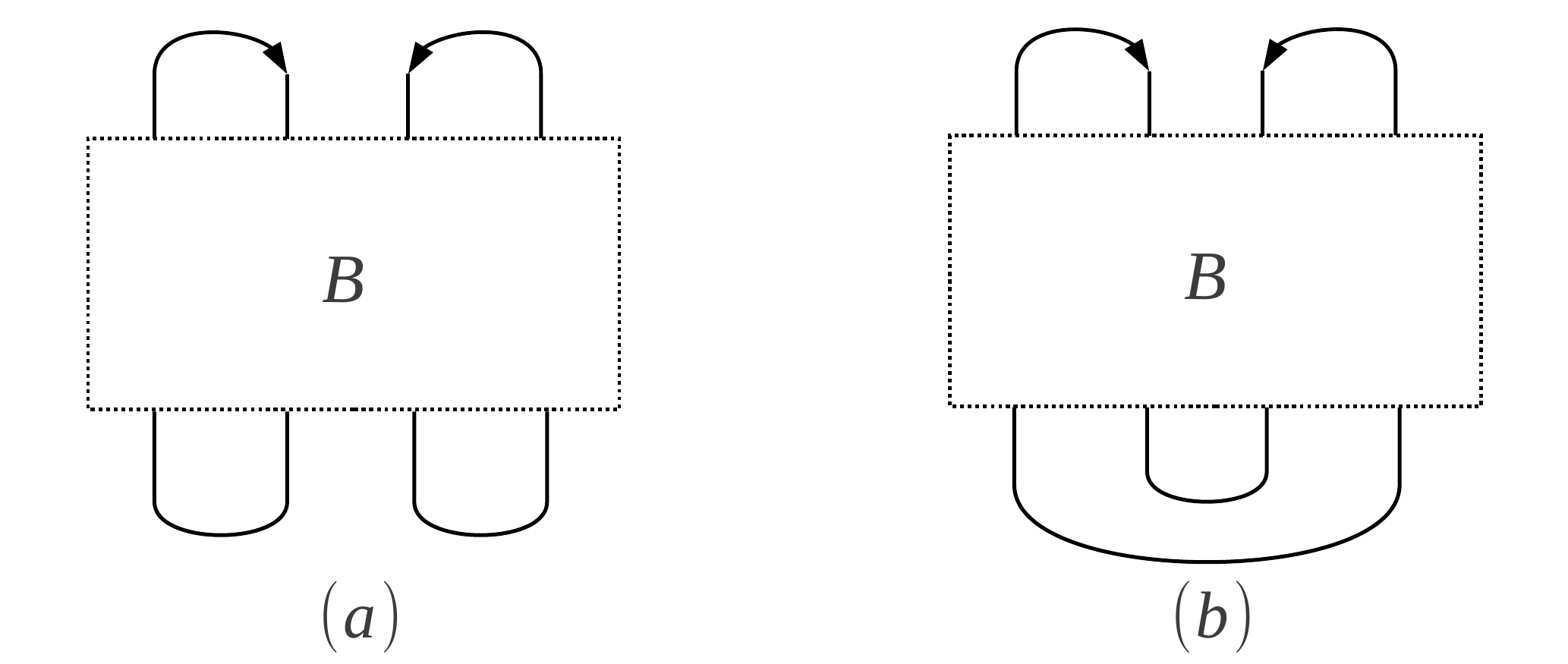}
\caption{{\textbf{The 4-plat presentations of two-bridge links.}} For any continued fraction $[a_1,...,a_m]=q/p$, there is a 4-plat projection of the two-bridge link $b(p,q)$. When $m$ is odd, we use (a) to close the 4-braid $B$ in the box; when $m$ is even, we use (b) to
close the 4-braid $B$.}
\label{fig:4-plat}
\end{figure}

\begin{thm}
\label{thm:two-bridge_L-space_links}
For all positive odd integers $r,q$, the two-bridge link $b(rq-1,-q)$ is an $L$-space link.
\end{thm}

\begin{figure}
\centering
\includegraphics[scale=0.3]{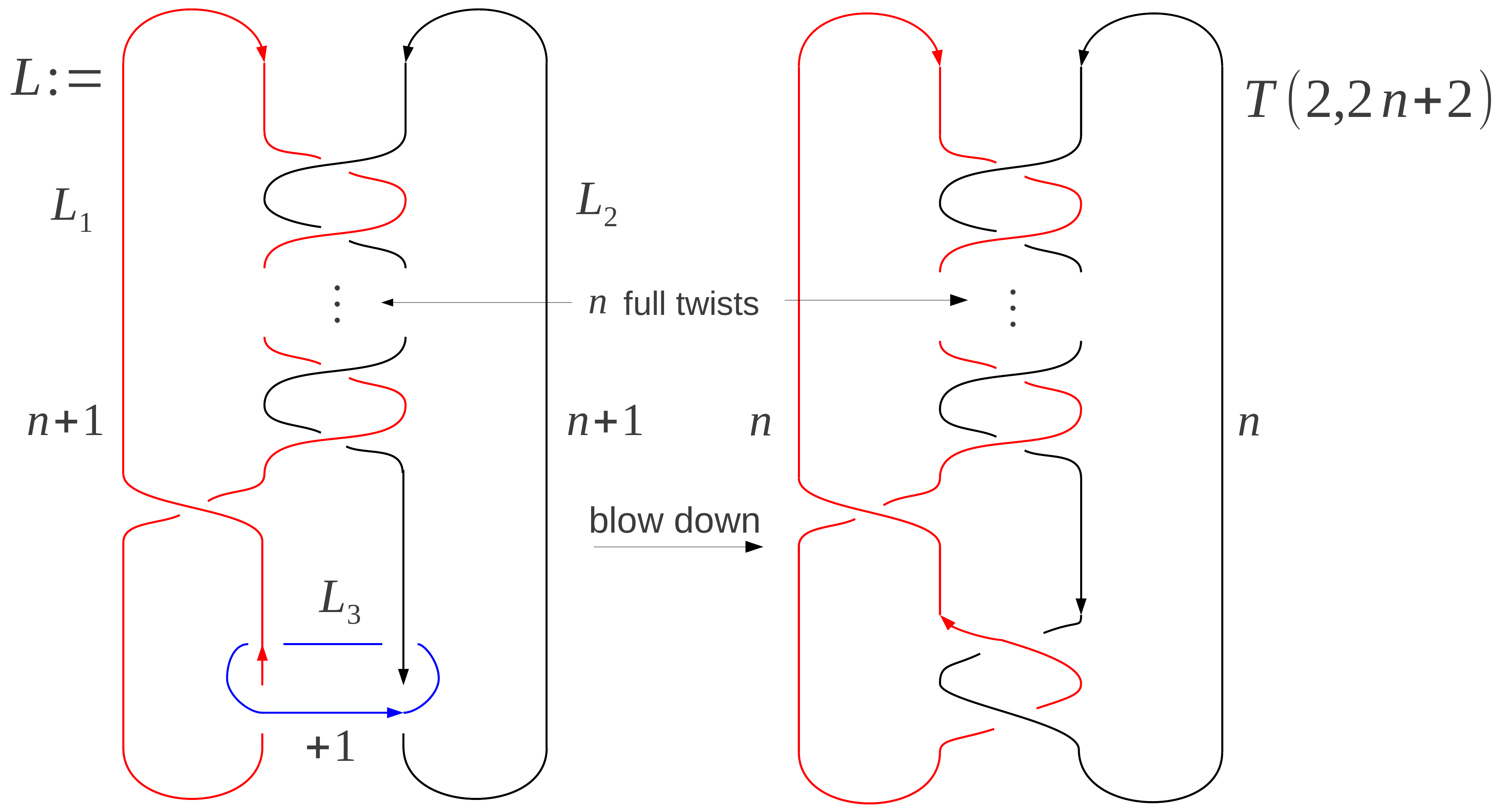}
\caption{\textbf{A 3-component link used to study the surgeries on $b(6n+2,-3)$.} The left link $L$ is used to study the surgeries on $b(6n+2,-3)$. After blowing down the $(-1)$-framed $L_3$, we can get the two-bridge link $b(6n+2,-3)$. While if we consider the $(n+1,n+1,1)$-surgery on $L$, after blowing down the $(+1)$-framed component $L_3$, we get the $(n,n)$-surgery on $T(2,2n+2)$, which is an $L$-space.}
\label{fig:(n+1,n+1)_surgery_b(6n+2,-3)}
\end{figure}

\begin{figure}
\centering
\includegraphics[scale=0.3]{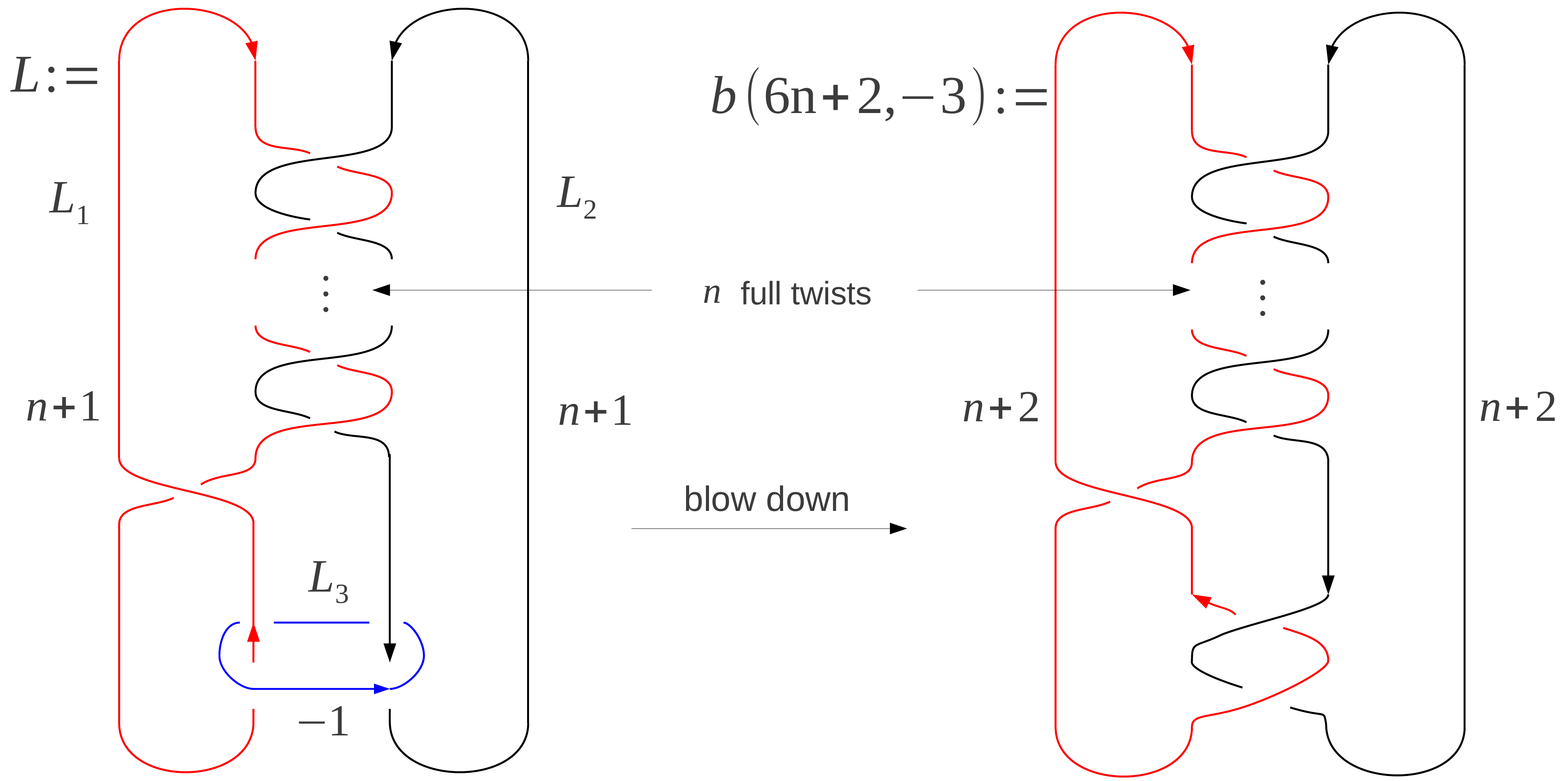}
\caption{\textbf{The $(n+2,n+2)$-surgery on the two-bridge link $b(6n+2,-3)$.} Consider the $(n+1,n+1,-1)$-surgery on the left 3-component link $L$. After blowing down the $(-1)$-framed component $L_3$, we get the $(n+2,n+2)$-surgery on the two-bridge link $b(6n+2,-3)$.}
\label{fig:(n+2,n+2)surgery on b(6n+2,-3)}
\end{figure}

\begin{figure}
\centering
\includegraphics[scale=0.3]{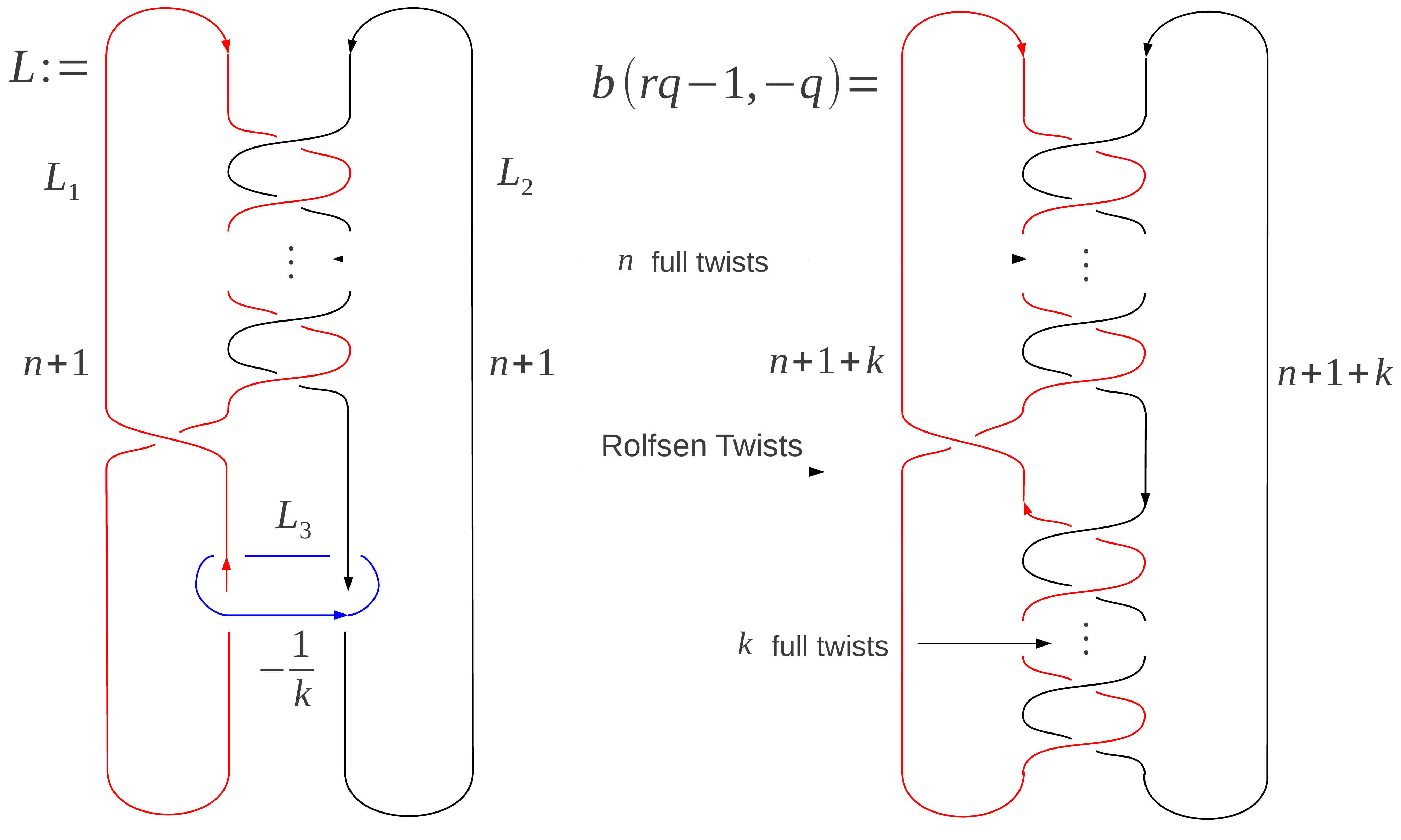}
\caption{\textbf{The $(n+1+k,n+1+k)$-surgery on the two-bridge link $b(rq-1,-q)$ with $r=2n+1,q=2k+1$.} Consider the $(n+1,n+1,-\frac{1}{k})$-surgery on the left 3-component link $L$. After doing the Rolfsen twists on the $(-1)$-framed component $L_3$, we get the $(n+1+k,n+1+k)$-surgery on the two-bridge link $b(rq-1,-q)$.}
\label{fig:Two-bridge_b((2k+1)(2n+1)-1,-(2k+1)))}
\end{figure}

\begin{proof}
Let $r=2n+1$ and $q=2k+1.$ Let us do induction on $k$.

First, for $k=1$, we need to show the family of two-bridge links $b(6n+2,-3)$ drawn in Figure \ref{fig:b(6n+2,-3)} are all $L$-space links. We claim that for any integer $n\geq 1$, the $(n+2,n+2)$-surgery on the two-bridge link $b(6n+2,-3)$ is an $L$-space. Consider the 3-component link $L=L_1\cup L_2\cup L_3$ drawn in Figure \ref{fig:(n+1,n+1)_surgery_b(6n+2,-3)}. We see that $L_1\cup L_2$ is the $T(2,2n)$ torus link with the linking number $n$.
Now consider the $(n+1,n+1,1)$-surgery on $L$, $S^3_{n+1,n+1,1}(L)$. By blowing down the $(+1)$-framed component $L_3$, we get the $(n,n)$-surgery on the $T(2,2n+2)$ torus link, $S^3_{n,n}(T(2,2n+2))$, which is an $L$-space by Corollary \ref{cor:L-space surgery on T(2,2n)}. While the $(n+1,n+1)$-surgery on the $T(2,2n)$ torus link $L_1\cup L_2$, $S^3_{n+1,n+1}(L_1\cup L_2)$, is also an $L$-space.
In addition, since
\begin{align*}
\det\left(\begin{array}{ccc}
n+1 & n & 1\\
n & n+1 & -1\\
1 & -1 & 1
\end{array}\right) &= -1-2n,\\
 \det\left(\begin{array}{cc}
n+1 & n \\
n & n+1
\end{array}\right) &=2n+1,
\end{align*}
from Lemma \ref{lem:L-space_surgery_iteration} it follows that the surgeries $S^3_{n+1,n+1,0}(L),S^3_{n+1,n+1,-1}(L)$ are both $L$-spaces.
By blowing down the $(-1)$-framed component $L_3$ on the $(n+1,n+1,-1)$-surgery on $L$, we get the $(n+2,n+2)$-surgery on the two-bridge link $b(6n+2,-3).$ See Figure \ref{fig:(n+2,n+2)surgery on b(6n+2,-3)}.

Since $\det \left(\begin{array}{cc}
n+2 & n-1\\
n-1 & n+2
\end{array}\right)>0$ and $n+2>0$,
 it follows from Lemma \ref{lem:L-space_surgery_iteration} that the two-bridge links $b(6n+2,-3)$ are all $L$-space links.

Fixing $n$, for $k>1$, consider rational surgeries on the 3-component link $L$ in Figure \ref{fig:(n+1,n+1)_surgery_b(6n+2,-3)}, with a rational coefficient on $L_3$. Then we have an exact triangle for the triple $S^3_{n+1,n+1,0}(L)$, $S^3_{n+1,n+1,-1/k}(L)$, and $S^3_{n+1,n+1,-1/(k+1)},$
$$\xymatrix{\HFh\left(S^3_{n+1,n+1,0}(L)\right)\ar[rr] & &\HFh\left(S^3_{n+1,n+1,-\frac{1}{k}}(L)\right)\ar[ld] \\
                                &\HFh\left(S^3_{n+1,n+1,-\frac{1}{k+1}}(L)\right).\ar[lu] &}$$
We claim that $S^3_{n+1,n+1,-\frac{1}{k+1}}(L)$ is an $L$-space for all positive integers $k$. We have shown that $S^3_{n+1,n+1,0}(L)$ is an $L$-space in the first step, and by the induction hypothesis, we can assume $S^3_{n+1,n+1,-\frac{1}{k}}(L)$ is an $L$-space. Moreover, we have
\[
\left| H_1(S^3_{n+1,n+1,-\frac{1}{k+1}}(L)) \right|=\left| H_1(S^3_{n+1,n+1,-\frac{1}{k}}(L)) \right|+\left| H_1(S^3_{n+1,n+1,0}(L)) \right|,
\]
since
\[
\left| H_1(S^3_{n+1,n+1,-\frac{1}{k}}(L)) \right|=
\det\left(\begin{array}{ccc}
n+1 & n & 1\\
n & n+1 & -1\\
k & -k & -1
\end{array}\right) = -1-2n-2k-4kn.
\]
Hence, from the above exact triangle it follows that $S^3_{n+1,n+1,-\frac{1}{k+1}}(L)$ is an $L$-space.

Now by doing Rolfsen twists on $L_3$, we get a $(n+1+k,n+1+k)$-surgery on the two-bridge link $b(pq-1,-q)=b(4kn+2k+2n,-2k-1).$ See Figure \ref{fig:Two-bridge_b((2k+1)(2n+1)-1,-(2k+1)))}. Since the linking number of $b(4kn+2k+2n,-2k-1)$ is $\pm(n-k)$, the determinant $\det\left(\begin{array}{cc}
n+1+k & \pm(n-k) \\
\pm(n-k) & n+1+k
\end{array}\right)$ is positive. Thus, by Lemma \ref{lem:positive_surgery_criterion}, we get $b(rq-1,-q)$ is an $L$-space link for all positive odd integers $r,q$.
\end{proof}

\begin{example}[Non-fibered hyperbolic $L$-space links]
\label{eg:Non-fibered_hyperbolic_L-space_lk}
The two-bridge links $b(10n+4,-5)$ with $n\in\bb{N}$ are $L$-space links, by  Theorem \ref{thm:two-bridge_L-space_links}. At least for $2\leq n\leq 6$, they are not fibered links, i.e., there does not exist any Seifert surface $F$ such that the link complement fibers over circle with fiber $F$. The fiberedness of links is detected by the knot Floer homology. See Corollary 1.2 in \cite{[YiNi]Fiber}: An oriented link $\overrightarrow{L}$ in $S^3$ is fibered if and only if the knot Floer homology $\widehat{HFK}(\overrightarrow{L})$ has a single copy of $\bb{Z}$ at the top Alexander grading.  Thus, for a homologically thin link $L$, the link $L$ is fibered if and only if its single-variable Alexander polynomial has leading coefficient $\pm1$. Note that two-bridge links are alternating and thus homologically thin; see Theorem 1.3 in \cite{[OS]link_FLoer}. We compute the multi-variable polynomials $\Delta_{L}(x,y)$ using the algorithm in \cite{[YL]Surgery_2-bridge_link} , and plug both $(t,t)$ and $(t,t^{-1})$ for $(x,y)$ so as to get the single-variable Alexander polynomials with both possible orientations.  It turns out that $b(10n+4,-5)$ is not fibered with any orientation, when $2\leq n\leq 6$. See Table \ref{table:non-fibered_hyperbolic_L-space_links}.
In fact, the fibered two-bridge knots and links are also classified by using continued fractions due to Gabai. See \cite{[Gabai]Fiberedness}. One should be able to generalize this to all $n\geq 2$ using number theoretic arguments.
\begin{table}
\begin{spacing}{2}
\begin{tabular}{|p{1in}|p{1.5in}|p{3.5in}|}
  \hline
  $\overrightarrow{L}=L_1\cup L_2$ & $\Delta_{-L_1\cup L_2}(t)=$  & $\Delta_{L_1\cup L_2}(t)=$
   \tabularnewline
  \hline
  $b(24,-5)$ & $2t^{2}-3t+2-\frac{3}{t}+\frac{2}{t^{2}}$ & $\frac{1}{t^{3}}(2t^{6}-3t^{5}+2t^{4}-3t^{3}+2t^{2}-3t+2)$
   \tabularnewline
  \hline
  $b(34,-5)$  & $ 3t^{2}-4t+3-\frac{4}{t}+\frac{3}{t^{2}}$  & $\frac{-1}{t^{4}}(2t^{8}-3t^{7}+2t^{6}-3t^{5}+2t^{4}-3t^{3}+2t^{2}-3t+2)$
   \tabularnewline
  \hline
  $b(44,-5)$ & $4t^{2}-5t+4-\frac{5}{t}+\frac{4}{t^{2}}$ & $\frac{1}{t^{5}}(2t^{10}-3t^{9}+2t^{8}-3t^{7}+2t^{6}-3t^{5}+2t^{4}-3t^{3}+2t^{2}-3t+2)$
   \tabularnewline
  \hline
  $b(54,-5)$& $5t^{2}-6t+5-\frac{6}{t}+\frac{5}{t^{2}}$ & $\frac{-1}{t^{6}}(2t^{12}-3t^{11}+2t^{10}-3t^{9}+2t^{8}-3t^{7}+2t^{6}-3t^{5}+2t^{4}-3t^{3}+2t^{2}-3t+2)$
   \tabularnewline
  \hline
   $b(64,-5)$& $6t^{2}-7t+6-\frac{7}{t}+\frac{6}{t^{2}}$ & $\frac{1}{t^{7}}(2t^{14}-3t^{13}+2t^{12}-3t^{11}+2t^{10}-3t^{9}+2t^{8}-3t^{7}+2t^{6}-3t^{5}+2t^{4}-3t^{3}+2t^{2}-3t+2)$
    \tabularnewline
  \hline
\end{tabular}
\end{spacing}
\caption{\textbf{Alexander polynomials of non-fibered hyperbolic $L$-space links.} Here, we consider the single variable Alexander polynomials for the two different orientations on the above $L$-space links. None of them has leading coefficient $1$, although the multi-variable Alexander polynomials do have coefficients $\pm1$. Thus, they are not fibered with any orientation. }
\label{table:non-fibered_hyperbolic_L-space_links}
\end{table}
\end{example}

\begin{figure}
\centering
\includegraphics[scale=0.3]{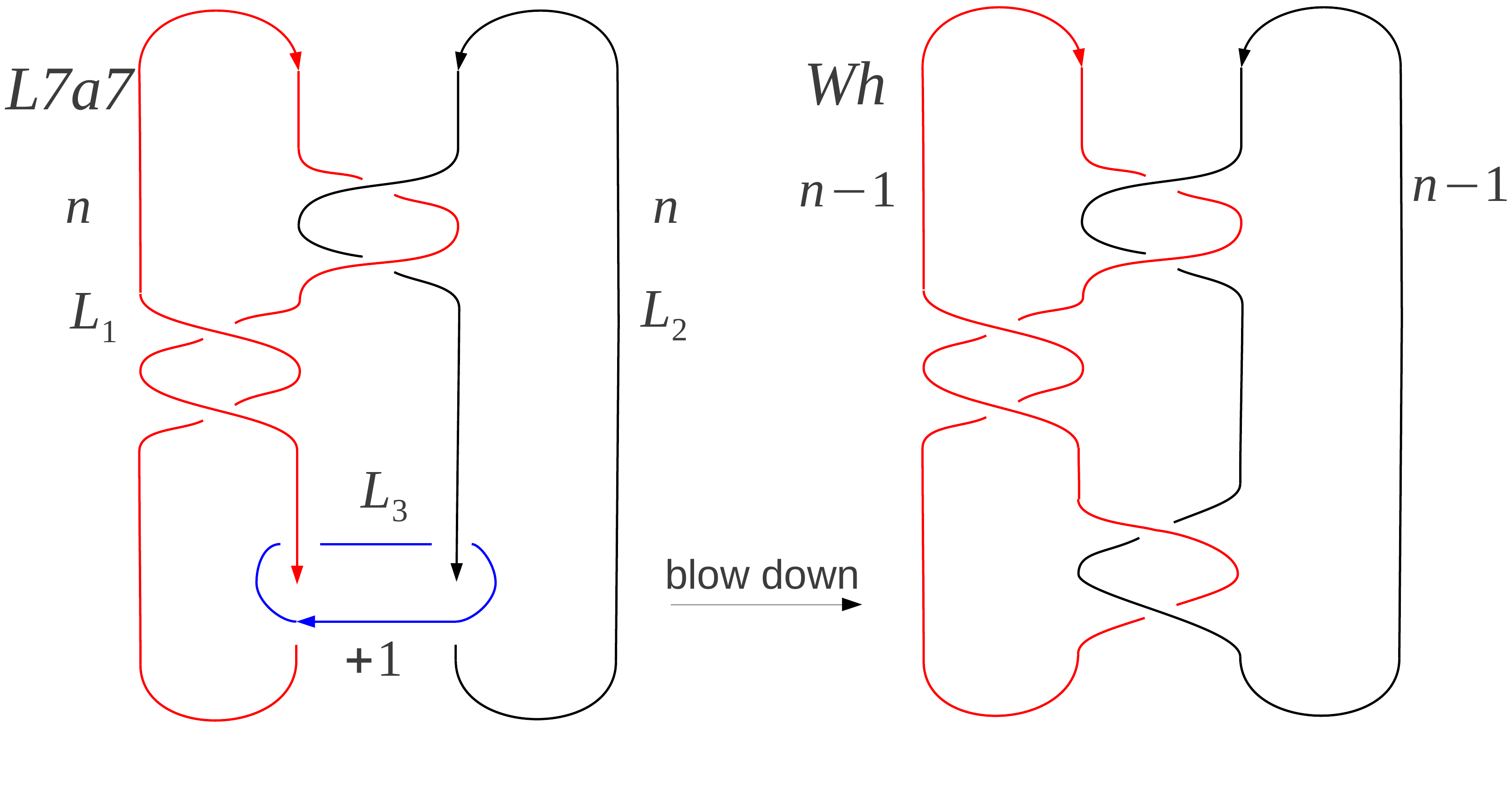}
\caption{\textbf{The 3-component link $L7a7$.} The 3-component link $L$ drawn above on the left is the mirror of $L7a7$ drawn in the Thistlethwaite Link Table on Knot Atlas. Consider the $(n,n,1)$-surgery on  $L$. After blowing down the $1$-framed component $L_3$, we get the $(n-1,n-1)$-surgery on the Whitehead link $\mathit{Wh}$.}
\label{fig:L7a7}
\end{figure}

\begin{example}[Plumbing trees]
\label{eg:plumbing_trees}
Any plumbing tree $L$ of unknots is an $L$-space link. In fact, any sufficiently
negative surgery on $L$ is a negative definite graph without bad vertices, and thus is an $L$-space,
by \cite{[O-S]Plumbed_mfld} Lemma 2.6. Since the plumbing tree is amphichiral, the sufficiently positive
surgeries are also $L$-spaces. Note that if $M$ is an $L$-space, then so is $-M$.
\end{example}

\begin{example}[$L7a7$ in the Thistlethwaite Link Table]
\label{eg:L7a7}
The link $L$ drawn in Figure \ref{fig:L7a7} is an $L$-space link.
It is actually the mirror of $L7a7$ drawn in the Thistlethwaite Link Table.
Consider the $(n,n,1)$-surgery on $L$. It is an $L$-space when $n\gg0.$ This is because after blowing down the $(+1)$-framed knot $L_3$, we get
the $L$-space Whitehead link $b(8,-3)$. Then, it follows from Lemma \ref{lem:positive_surgery_criterion} that $L$ is an $L$-space link.
\end{example}

\begin{example}
\label{eg:4-component_L-space_link}
The plumbing of unknots $L$ in Figure \ref{fig:plumbing_L-space_link} is
a hyperbolic $L$-space link. In fact, consider the $(3,1,3,1)$-surgery on $L$, which is $S^3$. By
Lemma \ref{lem:positive_surgery_criterion}, $L$ is an $L$-space link. In fact, this link is derived by
resolving the Whitehead link. Thus, all the surgeries on the Whitehead link are surgeries on this link.
\end{example}

\begin{figure}
\centering
\includegraphics[scale=0.45]{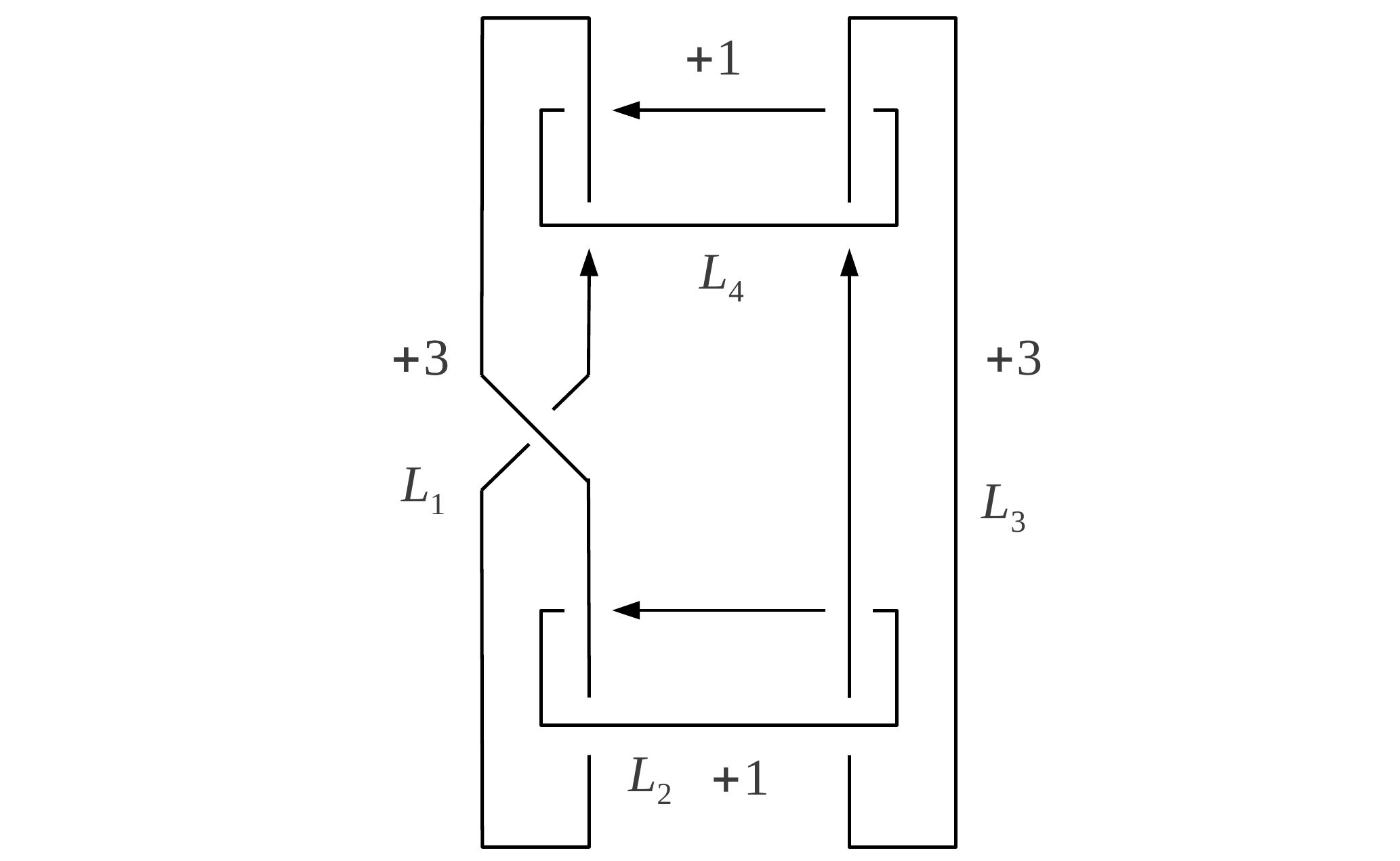}
\caption{\textbf{A plumbing graph $L$-space link.} Consider the link $L=L_1\cup...\cup L_4$
in the figure which is a plumbing of unknots. By blowing down $L_2,L_4$, we see that the surgery
shown  is $S^3.$}
\label{fig:plumbing_L-space_link}
\end{figure}

\begin{example}
\label{eg:five_rings}
The plumbing shown in Figure \ref{fig:five_rings_L-space_link} is a generalized $(++++-)L$-space link.
The $(1,1,1,1,1)$-surgery is the Poincar\'{e} sphere.
See \cite{[Rolfsen]} page 309.  In fact, every proper sublink is an $L$-space link, since the surgeries on them
are lens spaces. Thus, by Lemma \ref{lem:L-space_surgery_iteration}, the $(p_1,1,1,1,1)$-surgery is an $L$-space for all $p_1\geq1$, since $\det(\Lambda(1,1,1,1,1))=-1$
and $\det(\Lambda(1,1,1,1,1)|_{L-L_1})=-1.$ Next, from that $S^3_{p_1,1,1,1}(L-L_2)=L(p_1,1)$ and
$\det(\Lambda(p_1,1,1,1,1))=\det(\Lambda(p_1,1,1,1,1)|_{L-L_2})=-p_1$, it follows that
$(p_1,p_2,1,1,1)$-surgery on $L$ is an $L$-space for all $p_1\geq 1,p_2\geq 1.$ Similarly, we can get
the $(p_1,p_2,p_3,1,1)$-surgery is an $L$-space for all $p_1\geq 1,p_2\geq 1,p_3\geq 1.$ This is because
$S^3_{p_1,p_2,1,1}(L-L_3)=L(p_1,1)$, and $\det(\Lambda(p_1,p_2,1,1,1))=-p_1p_2$, $\det(\Lambda(p_1,p_2,1,1,1)|_{L-L_3})=-p_1.$
Now, we can get that the $(p_1,p_2,p_3,p_4,1)$-surgery on $L$ is an $L$-space, for all $p_1\geq 3,p_2\geq 3,p_3\geq 2,p_4\leq 1,$
since $\det(\Lambda(p_1,p_2,p_3,1,1))=p_2-p_1p_2-p_2p_3<0$,
$\det(\Lambda(p_1,p_2,p_3,1,1)|_{L-L_4})=1-p_1-p_3-p_1p_2+p_1p_2p_3>0.$ Finally, we can obtain that
the $(p_1,p_2,p_3,p_4,p_5)$-surgery on $L$ is an $L$-space for all $p_1\gg0,p_2\gg0,p_3\gg0,p_4\ll 0,p_5\geq1,$ due to
\begin{align*}
&\det(\Lambda(p_1,p_2,p_3,p_4,1))=p_1p_2p_3p_4+ \text{ lower terms}<0,\\
&\det(\Lambda(p_1,p_2,p_3,p_4,1)|_{L-L_5})=p_1p_2p_3p_4+\text{ lower terms}<0.
\end{align*}
\end{example}

\begin{figure}
\centering
\includegraphics[scale=0.37]{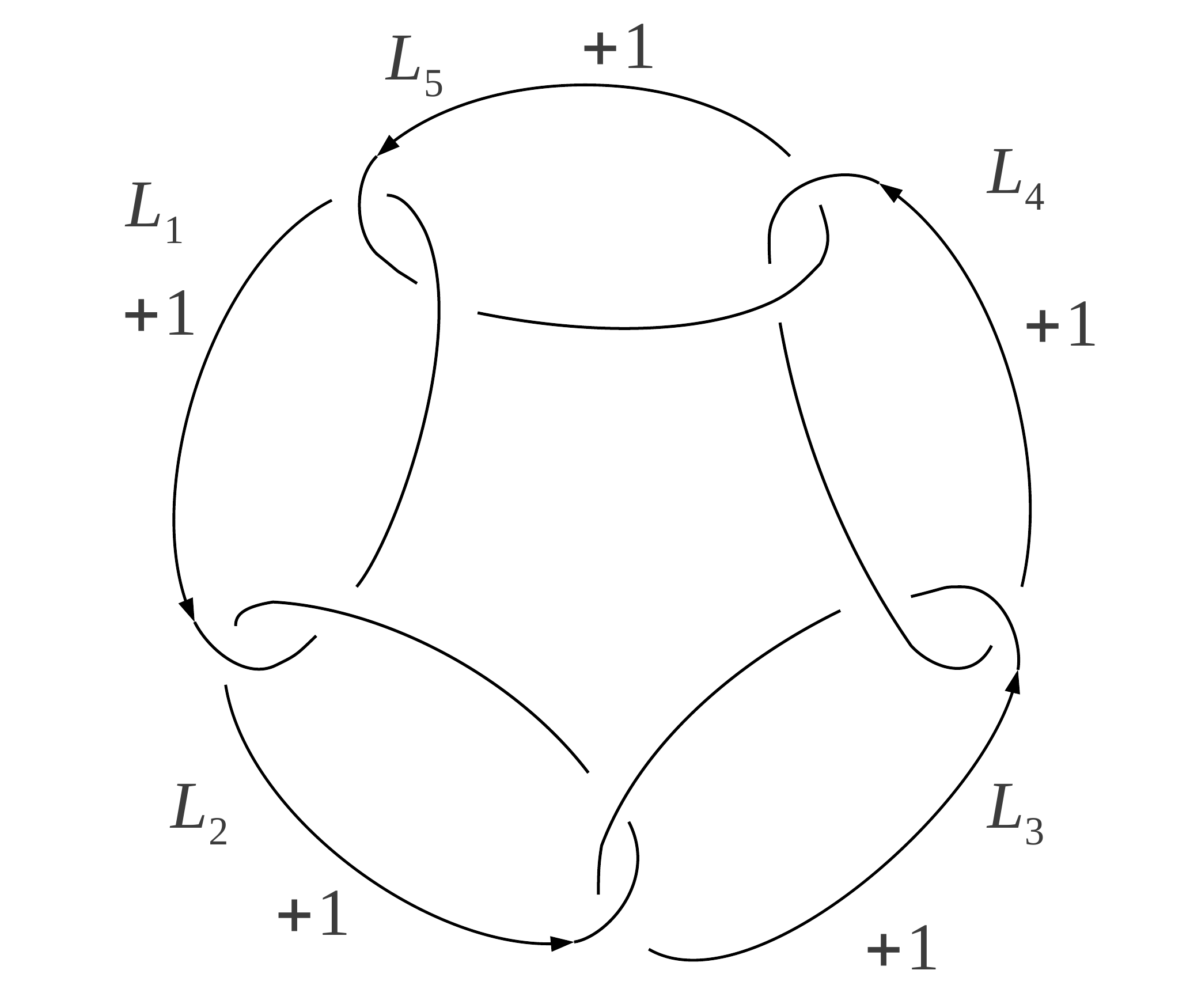}
\caption{\textbf{A generalized $(++++-)L$-space link.} Consider the link $L=L_1\cup...\cup L_5$
in the figure which is a plumbing of unknots. The surgery shown  is the Poincar\'{e} sphere.}
\label{fig:five_rings_L-space_link}
\end{figure}

\begin{example}[A family of $L$-space chain links]
\label{eg:chain links1}
A $n$-chain link consists of $n$ unknotted circles, linked together in a closed chain. Hyperbolic structures on $n$-chain link complements have been studied, for example, by Neumann and Reid  \cite{Arithmetic_hyperbolic_mfld}. They show that when $l\geq 5$ they are hyperbolic links.

The family of $l$-component chain links in Figure  \ref{fig:chain_link1} are all $L$-space links.  In fact, the $(1,2,...,2,l-2)$-surgery satisfies the positive $L$-space surgery criterion. First, if we blow down $L_1,L_2,...,L_{l-2}$ successively, then we get the $(1,1)$-surgery on the Whitehead link, the Poincar\'{e} manifold. Moreover, every proper sublink is a union of linear plumbings of unknots, and their surgeries are all connected sum of lens spaces. Thus, we only need to check the positive determinant condition.

Since a handle slide does not change the determinants of the surgery framing matrices, blowing down a $+1$ framed unknot does not change the determinants of the surgery framing matrices. Thus, after successively blowing down $L_1,...,L_{l-2}$ from $L$, we see that $\det(\Lambda(1,2,...,2,l-1))=1.$  For the proper sublinks, we only need to consider a linear plumbing $L'\subset L$. Since the determinant of the surgery framing matrix does not depend on the orientations, we can always orient $L'$ such that all the linking numbers of adjacent components are $-1.$ Let $M(k,n)$ denote the following $k\times k$ matrix
\[M(k,n)=\left(\begin{array}{cccccc}
n&-1&&&&\\
-1&2&-1&&&\\
&-1&2&&&\\
&&&\cdots&&
\\&&&&2&-1
\\&&&&-1&2\end{array}\right),\]
which is the surgery framing matrix of the linear plumbing in Figure \ref{fig:linear_plumbing1}.
There are four cases for computing $\det(\Lambda|_{L'})$:
\begin{itemize}
\item if $L_1\notin L', L_{l}\notin L'$, then $\Lambda|_{L'}=M(k,2)$ with $k$ being the number of components in $L'$;
\item if $L_1\notin L', L_{l}\in L'$, then $\Lambda|_{L'}=M(k,l-1)$;
\item if $L_1\in L', L_{l}\notin L'$, then after successively blowing down $L_1,L_2,...$, we can see $\det(\Lambda|_{L'})=1;$
\item if $L_1\in L', L_{l}\in L'$, then after successively blowing down $L_1,L_2,...$ inside $L'$, we can see $\det(\Lambda|_{L'})$ equals to $\det(M(k,n))$ with $k\leq l-2,n\geq 1.$
\end{itemize}
It is not hard to see $\det(M(k,2))=k+1$ by induction, and thus $\det(M(k,n))=nk-k+1.$ Therefore, all determinants are positive.
\end{example}

\begin{figure}
\centering
\includegraphics[scale=0.45]{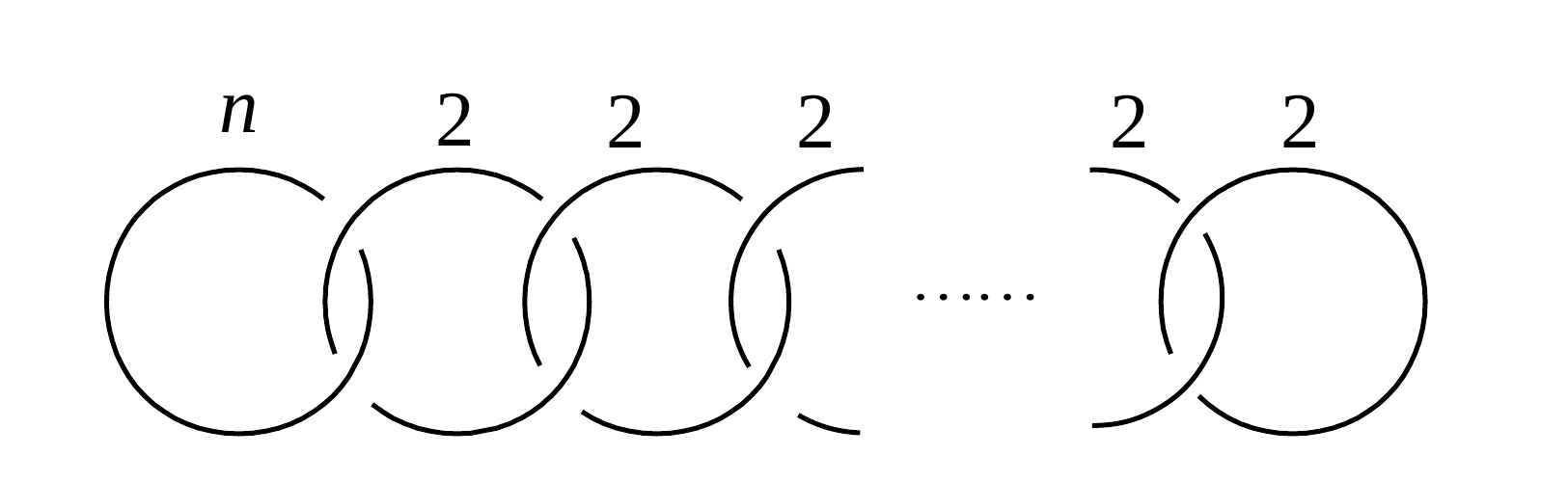}
\caption{\textbf{A linear plubming.}}
\label{fig:linear_plumbing1}
\end{figure}

\begin{figure}
\centering
\includegraphics[scale=0.35]{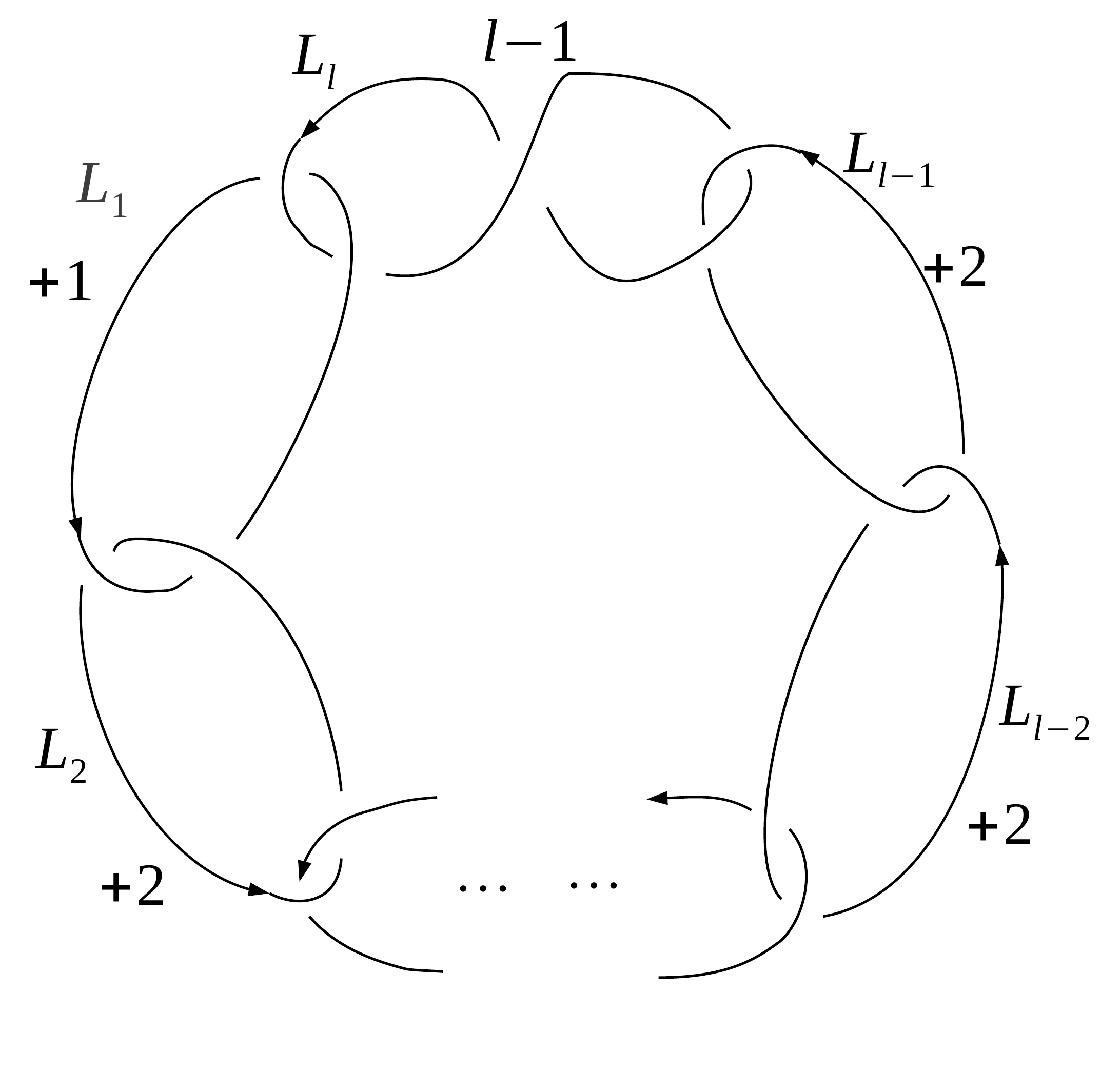}
\caption{\textbf{A family of hyperbolic $L$-space chain links.} The surgery labelled above satisfies the positive $L$-space surgery criterion.}
\label{fig:chain_link1}
\end{figure}

\begin{example}[Another sequence of $L$-space chain links]
\label{eg:chain links2}
Similarly, the family of $l$-component chain links in Figure  \ref{fig:chain_link2} are also all $L$-space links for $l\geq 3$.  In fact, when $n_1,n_2$ are large enough, the $(1,2,...,2,n_2,n_1)$-surgery satisfies the positive $L$-space surgery criterion. This is because after blowing down $L_1,...,L_{l-2}$, we have an $(n_1-l+2,n_2-1)$ framed $T(2,4)$
torus link. Thus, when $n_1,n_2$ are both large, this surgery is an $L$-space, since $T(2,4)$ is an $L$-space link. As is similar in Example \ref{eg:chain links1}, we only need to show when $n_1,n_2$ are large enough, $\det(\Lambda(1,2,...,2,n_2,n_1)|_{L'})$ is positive for any sublink $L'.$  For any sublink $L'$, we can blow down the circles on the side of $L_1$, and then obtain a linear plumbing as in Figure \ref{fig:linear_plumbing2}. The surgery matrix is a $k\times k$ matrix in the form of
\[\left(\begin{array}{cccccc}
n_1-c&-1&&&&\\
-1&n_2&-1&&&\\
&-1&2&&&\\
&&&\cdots&&
\\&&&&2&-1
\\&&&&-1&2\end{array}\right),\]
where $c$ is the number of times for blowing down $+1$-framed unknots.
The determinant of the above matrix is a polynomial of $n_1,n_2$, and the leading term is $\det(M(k-2,2))n_1n_2=(k-1)n_1n_2$. Thus, for $n_1,n_2$ large enough, all the determinants are positive.

Note that the link in Example \ref{eg:five_rings} is the same as the link here for $l=5.$
\end{example}

\begin{figure}
\centering
\includegraphics[scale=0.45]{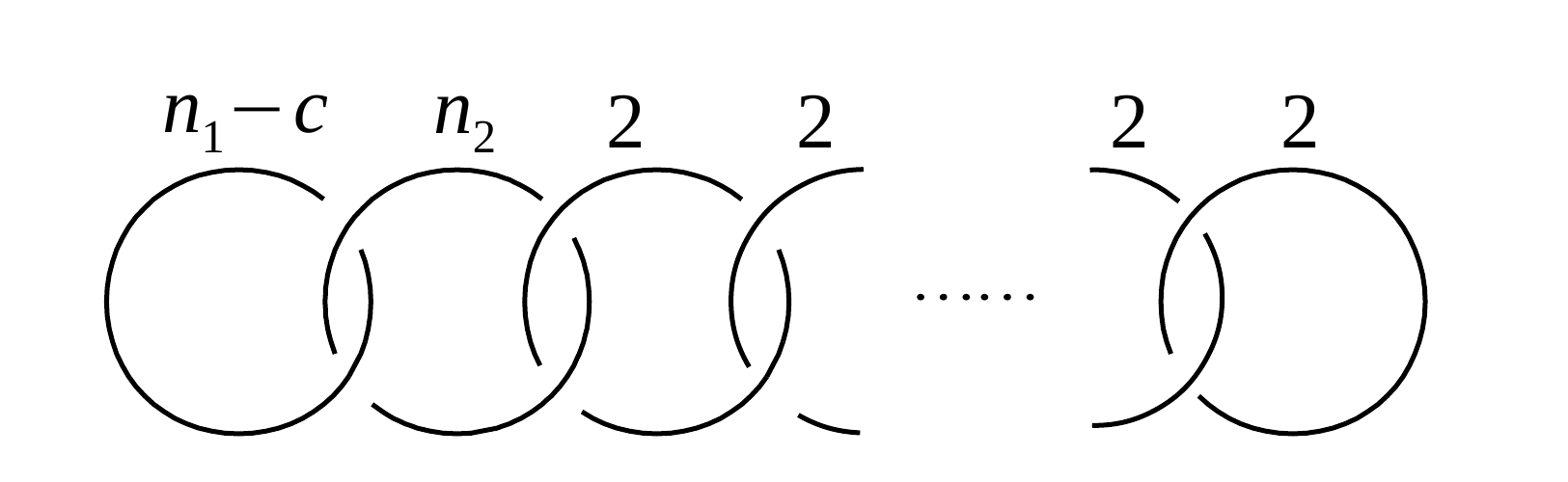}
\caption{\textbf{A linear plubming.}}
\label{fig:linear_plumbing2}
\end{figure}

\begin{figure}
\centering
\includegraphics[scale=0.35]{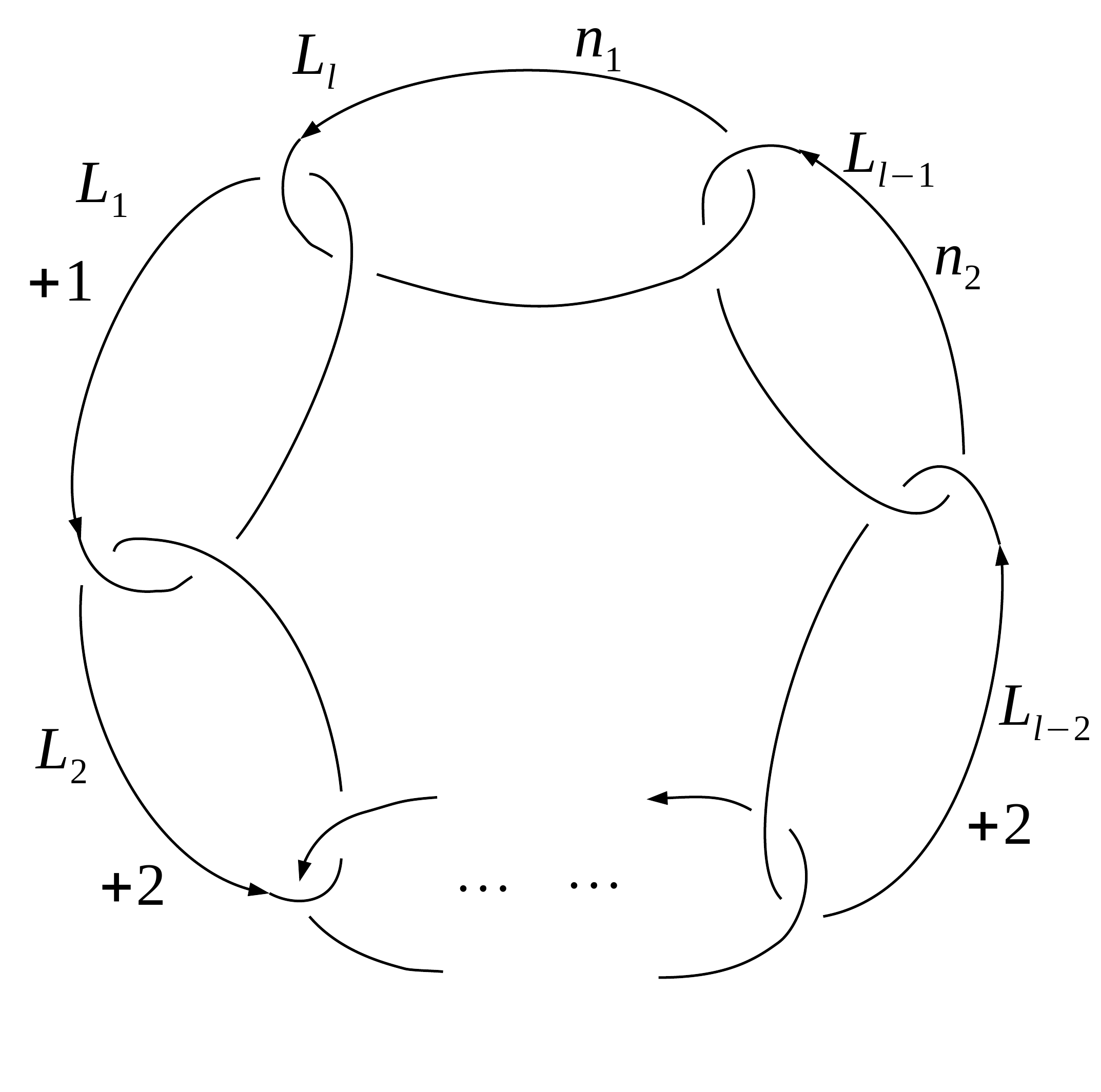}
\caption{\textbf{Another family of hyperbolic $L$-space chain links.} The surgery labelled above satisfies the positive $L$-space surgery criterion, when $n_1,n_2$ are large enough.}
\label{fig:chain_link2}
\end{figure}

\begin{example}
\label{eg:Hopf_plumbing}
The link $L^{(n)}=V_1\cup V_2\cup H_1\cup...\cup H_n$ shown in Figure \ref{fig:plumbing_L-space_link[3]} is a generalized $L$-space
link of  "++$-\cdots-$" type, for any $n\geq 1$. One can do similar induction as in Example \ref{eg:five_rings} to show the following claim.

Claim: For any $0\leq k\leq n$ and all integers $p_1\gg 0,p_2\gg 0, q_1 \ll 0,\cdots,q_k \ll 0,$ the $(p_1,p_2,q_1,\dots,q_k,-1,\dots,-1)$-surgery on $L^{(n)}$ is an $L$-space. Notice that the determinant of framing matrix
\[\det(\Lambda((p_1,p_2,q_1,\dots,q_k,-1,\dots,-1)))=(-1)^{n-k}p_1p_2q_1\cdots q_k+\text{ lower terms.}\]
The claim will follow from two induction on $n$ and on $k$.

Notice that surgeries on $L^{(n)}$ are mostly graph manifolds.
\end{example}

\begin{figure}
\centering
\includegraphics[scale=0.5]{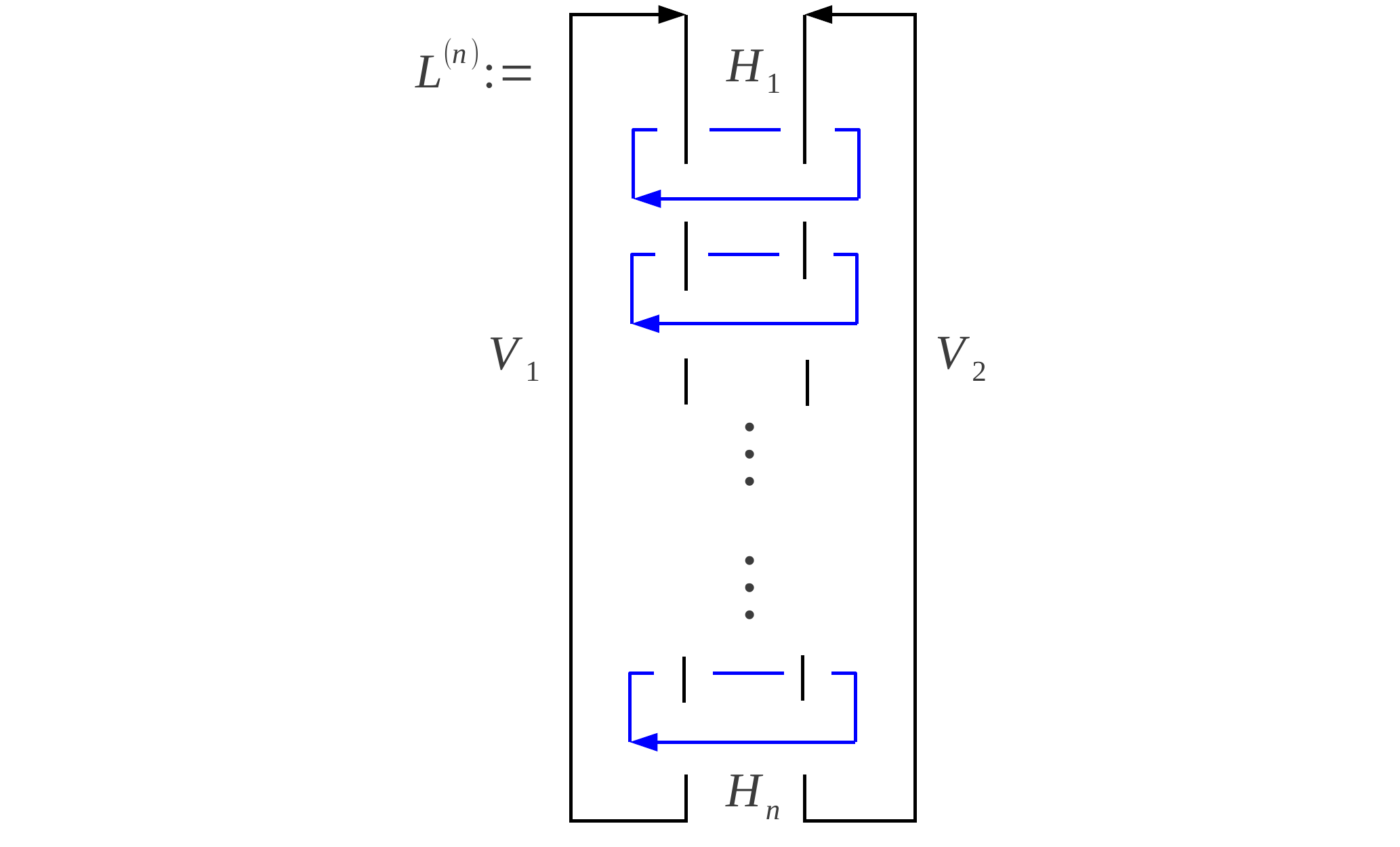}
\caption{\textbf{Another sequence of generalized $L$-space link.} Consider the link $L^{(n)}$ used in
the proof of Lemma \ref{lem:T(2,2n)}. It is in fact a generalized $L$-space link.}
\label{fig:plumbing_L-space_link[3]}
\end{figure}

\begin{example}[Thistlethwaite Link Table with crossing number $\leq7$]
\label{eg:small_links}
We examine the links in the Thistlethwaite Link Table with crossing number   $\leq 7$ and list the results in Table \ref{table:small links}.

Using the conditions of Alexander polynomials in Theorem \ref{thm: AlexPoly-of-L-space-link}, we conclude that $L6a1,$ $L7a1,$  $L7a2,$ $L7a4,$ and $L7a5$ are all non-$L$-space links.

The link $L6a2$ is the two-bridge link $b(10,7)$. Conjecture \ref{conj:L-space_two-bridge} has been verified for all two-bridge links $b(p,q)$ with $p\leq 100$ using the algorithm from \cite{[YL]Surgery_2-bridge_link}. So $L6a2$ is a non-$L$-space link.

The link $L6a5$ is the mirror of the left link in Figure  \ref{fig:(n+1,n+1)_surgery_b(6n+2,-3)} with $n=1$, on which the $(2,2,1)$-surgery is an $L$-space. Then, it quickly follows from the positive $L$-space surgery criterion that the mirror of $L6a5$ is an $L$-space link.

For the link $L6n1$, after blowing down a $+1$-framed component from it (all three components are symmetric), we get the unlink. So the $(10,10,1)$-surgery on $L6n1$ satisfies the positive surgery criterion, and thus showing that $L6n1$ is an $L$-space link.

The mirror of $L7a3$ consists of two components $L_1$ and $L_2$, where $L_1$ is the right-handed trefoil and $L_2$ is the unknot. Consider the $(n,1)$-surgery on the mirror of $L7a3$ with $n$ large. After blowing down the unknot, we have a large surgery on the right-handed torus knot $T(2,5)$. This is an $L$-space, since the right-handed torus knot $T(2,5)$ is an $L$-space knot. Then it follows from the positive surgery criterion that the mirror of $L7a3$ is an $L$-space link.

The link $L7a6$ is the two-bridge link $b(14,-9)$, and it is not $L$-space link by direct computation.

The link $L7n1$ has two components $L_1$ and $L_2$, where $L_1$ is the right-handed trefoil knot and $L_2$ is the unknot. Consider the $(10,1)$-surgery on this link. After blowing down the unknot,  the trefoil is unknotted and we obtain a lens space surgery. Since the right-handed trefoil is an $L$-space knot, from the positive surgery criterion it follows that $L7n1$ is an $L$-space link.

The link $L7n2$ is not an $L$-space link; see Proposition \ref{thm:L7n2} for the proof. Its mirror is not an $L$-space link neither, since the left-handed trefoil is not an $L$-space knot. However,  $L7n2$  is a generalized $(+-)L$-space link. The link $L7n2$
consists of two components $L_1$ and $L_2$, with $L_1$ being the right-handed trefoil and $L_2$ being the unknot. Consider the $(n,-1)$-surgery on $L7n2$ with $n$ large. After blowing down the unknot, we get the unknot, thus getting a lens space surgery. Then, since the right-handed trefoil is an $L$-space knot, $(n,-k)$-surgery is an $L$-space for all $k>0$ and large $n$ by Lemma \ref{lem:L-space_surgery_iteration}.

\end{example}

\begin{table}
\begin{spacing}{1.2}
\begin{tabular}{|p{.7in}|p{1in}|p{1.5in}|p{2.5in}|}
  \hline
  Links    & $L$-space link &  Alexander polynomial &Comments \\
  \hline
  $L2a1$ & Yes & Yes & The Hopf link\\
  \hline
  $L4a1$  & Yes   & Yes & The $T(2,4)$ torus link\\
  \hline
  $L5a1$ & Yes  & Yes & Mirror of the $L$-space Whitehead link\\
  \hline
  $L6a1$&  No  & No &\\
  \hline
  $L6a2$ &No &Yes & \\
  \hline
  $L6a3$ &Yes &Yes & The $T(2,6)$ torus link\\
  \hline
  $L6a4$&Yes & Yes &The Borromean link\\
  \hline
  $L6a5$ &Yes &Yes &The mirror is an $L$-space link\\
  \hline
  $L6n1$ &Yes &Yes &\\
  \hline
  $L7a1$ &No & No &\\
  \hline
  $L7a2$ &No & No &\\
  \hline
  $L7a3$ &Yes &Yes &The mirror is an $L$-space link\\
  \hline
  $L7a4$ &No &No &\\
  \hline
  $L7a5$ &No &No &\\
  \hline
  $L7a6$ &No &Yes  &The two-bridge link $b(14,-9)$   \\
  \hline
  $L7a7$ &Yes &Yes & The mirror is an $L$-space link  \\
  \hline
  $L7n1$ &Yes &Yes &\\
  \hline
  $L7n2$ &No &Yes &\emph{Generalized $(+-)L$-space link}\\
  \hline
\end{tabular}
\end{spacing}
\caption{\textbf{Thistlethwaite Link Table with crossing number $\leq7$.} Here, by "Yes" in the column "$L$-space link", it means either the link or its mirror is an $L$-space link; by "Yes" in the column "Alexander polynomial", it means the conditions on the multi-variable Alexander polynomial in Theorem \ref{thm: AlexPoly-of-L-space-link} are satisfied.}
\label{table:small links}
\end{table}

\section{Floer homology and Alexander polynomials of $L$-space links}
In this section, we study the link Floer homology and the multi-variable Alexander polynomials of $L$-space links with $l\geq2$ components.
The Alexander polynomial of $L$ is determined by the Euler characteristics of the link Floer homology $HFL^-(L,{\mathrm{\bf{s}}})$,
due to Equation (2) in \cite{[OS]link_FLoer}
\begin{equation}
\label{eq:chi_HFL^-}
\Delta_L(x_1,...,x_l)\overset{\centerdot}{=}\sum_{(s_1,...,s_l)\in \bb{H}(L)}\mbox{\Large$\chi$}(HFL^-(L,s_1,...,s_l))\cdot x_1^{s_1}\cdots x_l^{s_l},
\end{equation}
where $f\overset{\centerdot}{=}g$ denotes that $f$ and $g$ differ by multiplication by units.
Here, we use $CFL^-(L)$ rather than $\widehat{CFL}(L)$ as in \cite{[OS]lens_space_surgery}. Note that $CFL^-(L,s_1,s_2)$ is a finite dimensional $\bb{F}$-vector space, and thus $\bigchi(CFL^-(L,s_1,s_2))=\bigchi(HFL^-(L,s_1,s_2))$.

Now we are ready to prove Theorem \ref{thm: AlexPoly-of-L-space-link} from the introduction.

\subsection{Proof of Theorem \ref{thm: AlexPoly-of-L-space-link}.}
\begin{proof}
Fixing $\{s_i\}_i$, we denote the following successive quotient complexes by
\begin{align*}
C_k^{(1)}=& \{x\in {\bf{CF}}^-(S^3)|A_i(x)= s_i, 1\leq i\leq k, A_j(x)\leq s_j, k+1 \leq j\leq l\},\\
C_k^{(2)}=& \{x\in {\bf{CF}}^-(S^3)|A_i(x)= s_i, 1\leq i\leq k, A_{k+1}(x)\leq s_{k+1}-1, A_j(x)\leq s_j, k+2 \leq j\leq l\}.
\end{align*}
Then, $C_0^{(1)}=\mathfrak{A}^-_{\mathrm{\bf{s}}}$, $C_0^{(2)}=\mathfrak{A}^-_{s_1-1,s_2,...,s_l}$, $C_l^{(1)}=CFL^-(L,\mathrm{\bf{s}})$, and $C_{k+1}^{(1)}=C_k^{(1)}/C_k^{(2)}$.

Consider the short exact sequence of chain complexes
\[ 0\rightarrow C_0^{(2)} \xrightarrow{\iota} C_0^{(1)} \rightarrow C_1^{(1)} \rightarrow 0,\]
where the map $\iota$ is the inclusion map.  It induces another
short exact sequence of homologies
\[0 \rightarrow \mathrm{coker}(\iota_*) \rightarrow H_*(C_1^{(1)}) \rightarrow \ker(\iota_*) \rightarrow 0.\]
Since $H_*(C^{(1)}_0)=H_*(C^{(1)}_0)=\bb{F}[[U]]$, the map $\iota_*:\bb{F}[[U]]\to \bb{F}[[U]]$ is either $0$ or a multiplication of $U^k$ for some integer $k\geq 0.$ Thus, $H_*(C_1^{(1)})$ is either  $\bb{F}[[U]]/U^k$ or $\bb{F}[[U]]\oplus \bb{F}[[U]]$.

In addition, since $U_1$ acts on the chain complex $C_1^{(1)}$ as $0$, it also acts as $0$ on homology. Thus, $H_*(C_1^{(1)})$ is either $0$ or $\bb{F}[[U]]/U$, according to either $k=0$ or $k=1$.   Note here $\bb{F}[[U]]$ denotes $\bb{F}[[U_1,U_2,...,U_l]]/(U_1-U_2,...,U_1-U_l)$ as an $\bb{F}[[U_1,U_2,...,U_l]]$-module  and the $U$-action denotes any action of $U_i$.
Furthermore, $\mbox{\Large$\chi$}(H_*(C_1^{(1)}))$ is either 0 or 1. In fact, if $H_*(C_1^{(1)})=0$, then the grading of $1\in \bb{F}[[U]]=H_*(C_0^{(1)})$ equals to the grading of $1\in \bb{F}[[U]]=H_*(C_0^{(2)})$; while if $H_*(C_1^{(1)})=\bb{F}[[U]]/U$, then the grading of $1\in H_*(C_0^{(1)})$ equals to the grading of $1\in H_*(C_0^{(2)})$ plus 2, and the grading of $1\in\bb{F}[[U]]/U= H_*(C_1^{(1)})$ equals to the grading of $1\in H_*(C_1^{(1)})$. Moreover, the complex $\mathfrak{A}^-_{+\infty,...,+\infty}$ is just ${\bf{CF}}^-(S^3)$ and the absolute gradings of  elements in $H_*(\mathfrak{A}^-_{+\infty,...,+\infty})$ are all even integers.  An induction will show that all the absolute gradings of elements in the homologies of  $\mathfrak{A}^-_{s_1,s_2,...,s_l}$ are all even integers.
Thus, we have
\[\bigchi(H_*(C_{1}^{(1)}))=0 \text{ or } 1.\]
Notice that $C_k^{(1)}$ and $C_k^{(2)}$ are defined similarly, just with different $\mathrm{\bf s}$ values. Thus, we can similarly show that
\[\bigchi(H_*(C_{1}^{(2)}))=0 \text{ or } 1.\]

Since $\bigchi(H_*(C_{k+1}^{(1)}))=\bigchi(H_*(C_k^{(1)}))-\bigchi(H_*(C_k^{(2)}))$ and $\left|\bigchi(H_*(C_{1}^{(2)}))\right|\leq 1$ for any fixed $\bf{\mathrm s}$, it is not hard to see that
\begin{align*}
\left|\bigchi( H_*(C_k^{(1)}))\right|\leq 2^{k-2}, \forall k=2,...,l,\\
\left|\bigchi( H_*(C_k^{(2)}))\right|\leq 2^{k-2}, \forall k=2,...,l.
\end{align*}
Hence, we prove Inequality \eqref{eq:ineq2} by letting $k=l$.

Since $C_{k+1}^{(1)}=C_k^{(1)}/C_k^{(2)}$, we have
\[\mathrm{rank}_{\bb{F}}(C_{k+1}^{(1)})\leq \mathrm{rank}_{\bb{F}}(C_k^{(1)})+\mathrm{rank}_{\bb{F}}(C_k^{(2)}).\]
From $\mathrm{rank}_{\bb{F}}H_*(C_{1}^{(1)})\leq 1$, it follows that $\mathrm{rank}_{\bb{F}}HFL^-(L,\mathrm{\bf{s}})\leq 2^{l-1}.$ Thus, Inequality \eqref{eq:ineq1} holds.

 Let us look at the signs of the multi-variable Alexander polynomial when $l=2$.  Suppose $\mbox{\Large$\chi$}(CFL^-(L,s_1,s_2))$ and $\mbox{\Large$\chi$}(CFL^-(L,s_1+k,s_2)$ are the consecutive non-zero Euler characteristics among the horizontal Alexander gradings, that is,
\begin{itemize}
\item $\left|\mbox{\Large$\chi$}(CFL^-(L,s_1,s_2))\right|=1,$
\item $\left|\mbox{\Large$\chi$}(CFL^-(L,s_1+k,s_2))\right|=1,$
\item $\mbox{\Large$\chi$}(CFL^-(L,s_1+i,s_2))=0,\forall i=1,2,...,k-1.$
\end{itemize}
Then, we have
\begin{align*}
   \label{eq:signs_of_AlexPoly}\tag{$*$}
 &\mbox{\Large$\chi$}(H_*(\mathfrak{A}^-_{s_1+k,s_2}/\As{s_1+k}{s_2-1}))-\mbox{\Large$\chi$}(H_*(\mathfrak{A}^-_{s_1-1,s_2}/\As{s_1-1}{s_2-1}))\\
 =&\sum_{i=0}^{k}\mbox{\Large$\chi$}(CFL^-(L,s_1+i,s_2))\\
   =& \mbox{\Large$\chi$}(CFL^-(L,s_1,s_2))+\mbox{\Large$\chi$}(CFL^-(L,s_1+k,s_2)).
\end{align*}

Since $\mbox{\Large$\chi$}(H_*(\mathfrak{A}^-_{s_1,s_2}/\As{s_1}{s_2-1}))=0$ or $1$, for all $(s_1,s_2)\in \bb{H}(L)$, the top row of Equation \eqref{eq:signs_of_AlexPoly} is $0$ or $\pm1$. Whereas by the assumption, the bottom row of Equation \eqref{eq:signs_of_AlexPoly} is $0$ or $\pm2.$  Thus, we have
$$ \mbox{\Large$\chi$}(CFL^-(L,s_1,s_2))+\mbox{\Large$\chi$}(CFL^-(L,s_1-k,s_2))=0.$$
\end{proof}

\begin{cor}
A homologically thin $L$-space 2-component prime link $L=L_1\cup L_2$ has fibered link exterior.
\end{cor}

\begin{proof}
The homologically thin condition means that the homology $\widehat{HFL}(L,\mathrm{\bf s})$ is supported in a single Maslov grading, and thus is determined by its Euler characteristic. Thus, the
 link Floer homology is determined by the multi-variable Alexander polynomial. However, here we need to consider the hat version link Floer homology for the discussions of fiberedness. Let the symmetrized Alexander polynomial be
$$\Delta_L(x,y)=\sum_{i,j}a_{i,j}\cdot x^i\cdot y^j.$$
We choose
$$x_0=\max\{i|a_{i,j}\neq 0\},\quad y_0=\max\{j|a_{x_0,j}\neq 0\}.$$
Since
\[  \sum_{(s_1,s_2)\in \bb{H}(L)}\mbox{\Large$\chi$}(\widehat{HFL}(L,s_1,s_2))\cdot x^{s_1}\cdot y^{s_2}=\pm\frac{(x-1)(y-1)}{\sqrt{xy}}\Delta_L(x,y),
\]
we have that $(x_0+\frac{1}{2},y_0+\frac{1}{2})$ is an extreme point of the polytope for $\widehat{HFL}(L)$, and $\mbox{\Large$\chi$}(\widehat{HFL}(L,x_0+\frac{1}{2},y_0+\frac{1}{2}))=\pm1$. Furthermore, since $L$ is homological thin, we have that $\mathrm{rank}\widehat{HFL}(L,x_0+\frac{1}{2},y_0+\frac{1}{2})=1$, and thereby the link exterior of $L$ is fibered.
\end{proof}

\subsection{Examples.} Let us use Theorem \ref{thm: AlexPoly-of-L-space-link} to detect $L$-space links among two-bridge links.
Notice that in the knot case the Alexander polynomial gives a strong obstruction for an alternating knot to be an $L$-space knot. In \cite{[OS]lens_space_surgery}, it is shown that alternating $L$-space knots are only $(2,2n+1)$ torus knots.

\begin{prop}[Ozsv\'{a}th-Szab\'{o}, \cite{[OS]lens_space_surgery}, Proposition 4.1]
\label{prop:Alternating_Alex_knots}
If $K$ is an alternating knot with the property that all the coefficients $a_i$ of its Alexander polynomial $\Delta_{K}$ have $|a_i|\leq1$, then $K$ is the $(2,2n+1)$ torus knot.
\end{prop}

\begin{thm}[Ozsv\'{a}th-Szab\'{o}, \cite{[OS]lens_space_surgery}, Theorem 1.5]
\label{thm:alternating_L-space_knot}
If  $K\subset S^3$ is an alternating knot with the property that there is some integral surgery along $K$ is an $L$-space, then
$K$ is a $(2,2n+1)$ torus knot for some integer $n$.
\end{thm}

In contrast to the knot case, by computer experiments, we find many hyperbolic two-bridge non-$L$-space links whose Alexander polynomials satisfy the constraints in Theorem \ref{thm: AlexPoly-of-L-space-link}. We list some interesting phenomena in the two-bridge links $b(p,q)$ below, where $0<p\leq 100$. We conjecture that for any $n>0$, the two-bridge link $b(6n+4,-3)$ is a
non-$L$-space link with Alexander polynomial satisfying Theorem \ref{thm: AlexPoly-of-L-space-link}.

\
\begin{centering}
\begin{spacing}{1.5}
\begin{tabular}{|c|c|c|c|}
  \hline
  Links  &  Alexander polynomial condition:  & Hyperbolic link: & $L$-space link:\\
  \hline
  $b(2n,-1)$ & Yes & Torus link $T(2,2n)$ & Yes\\
  \hline
  $b(6n+2,-3)$  & Yes  & Hyperbolic link & Yes\\
  \hline
  $b(6n+4,-3)$ & Yes & Hyperbolic link & No, when $6n+2\leq 100$.\\
  \hline
  $b(10n\pm2,5)$& No & Hyperbolic link & No\\
  \hline
\end{tabular}
\end{spacing}
\end{centering}

\subsection{$HFL^-$ of $L$-space links.}
Let $L$ be an $L$-space link. In general, $CFL^-(L,{\mathrm{\bf{s}}})$ is an iterated quotient complex of
$\mathfrak{A}^-_{\mathrm{\bf{s}}}$.

For every subcomplex $C_1\subset C$, the quotient complex $C/C_1$  is quasi-isomorphic to the mapping cone
of the inclusion map $i:C_1\to C$. Thus, it leads to an iterated mapping cone construction of $CFL^-(L,\mathrm{\bf{s}})$ by
using $\mathfrak{A}^-_{\mathrm{\bf{s}}}$. This provides a spectral sequence converging to $HFL^-(L,\mathrm{\bf{s}})$ considered
as $\bb{F}$-vector spaces, which is stated in \cite{[Gorsky_Nemithi]algebriac_links_and_lattice_homology}.
This spectral sequence also implies Inequality \eqref{eq:ineq1}.

\section{Surgeries on $L$-space links}
Using the knot surgery formula from \cite{[OS]knot_surgery}, the graded Heegaard Floer homology of surgeries on $L$-space knots are
determined by the Alexander polynomial and the surgery coefficient. Using Manolescu-Ozsv\'{a}th link surgery formula
from \cite{[MO]link_surgery} and algebraic rigidity results from \cite{[YL]Surgery_2-bridge_link}, we  prove Theorem \ref{thm:L-space_surgery} and give some explicit formulas in this section.

The generalized Floer complexes $\mathfrak{A}^-_{\mathrm{\bf s}}$'s are $\bb{F}[[U_1,...,U_l]]$-modules, and all the $U_i$ actions are homotopic to the $U_1$ action. In fact, when $L$ is an $L$-space link,  $\mathfrak{A}^-_{\mathrm{\bf{s}}}(L)$ is chain homotopic to $\bb{F}[[U_1]]$ preserving the $\bb{Z}$-grading. This is done by restricting our scalars to $\bb{F}[[U_1]]$ and applying the algebraic rigidity results Proposition 5.5 and Corollary 5.6 in \cite{[YL]Surgery_2-bridge_link}. There is an absolute $\bb{Z}$-grading on $\mathfrak{A}^-_{\mathrm{\bf s}}$. However, the $U_1$ action decreases it by $2$, and thus it is not a chain complex of $\bb{F}[[U_1]]$-modules. So the complexes here are considered as $\bb{Z}/2\bb{Z}$-graded chain complexes of $\bb{F}[[U_1]]$-modules, together with a $\bb{Z}$-grading compatible with the $\bb{Z}/2\bb{Z}$-grading where $U_1$ lowers the $\bb{Z}$-grading by $2$.

\begin{prop}[Proposition 5.5, \cite{[YL]Surgery_2-bridge_link}]
\label{prop:homotopies on the diagonal}Let $A_{*},B_{*}$ be $\bb{Z}$-graded complexes
of $\mathbb{F}$-modules with $U$-action dropping grading by
2 and commuting with the differential. Suppose $A,B$ are both free $\bb{F}[[U]]$-modules, and $H_{*}(A)=H_{*}(B)=\mathbb{F}[[U]]$, precisely, $H_{2k}(A)\cong H_{2k}(B)\cong\mathbb{F}$
for all $k\leq0$ and $H_{i}(A)=H_{i}(B)=0$ otherwise, where $U\cdot H_{2k}(A)=H_{2k-2}(A),U\cdot H_{2k}(B)=H_{2k-2}(B)$.

Then, if $F,G:A\to B$ are both quasi-isomorphisms of $\bb{F}[[U]]$-modules, then $F,G$ are chain homotopic as maps of
$\bb{F}[[U]]$-modules. Moreover, if $H,K$ are both chain homotopies as homomorphisms of $\mathbb{F}[[U]]$-modules
between any two chain maps $f,g:A\rightarrow B$, i.e. $H\partial+\partial H=K\partial+\partial K=f-g$,
then $H-K=\partial T+T\partial$, for some $\mathbb{F}[[U]]$-module
homomorphism $T:A_{*}\rightarrow B_{*+2}$.
\end{prop}

Using these chain homotopy equivalences, we replace $\mathfrak{A}^-_{\mathrm{\bf{s}}}(L)$ by $\bb{F}[[U_1]]$ in the Manolescu-Ozsv\'{a}th link surgery complex and replace the maps up to homotopies. In \cite{[YL]Surgery_2-bridge_link}, we call this new complex the \emph{perturbed surgery formula.} Thus, we only need to determine the map $\Phi^{\V{M}}_{\mathrm{\bf s}}$ in the perturbed surgery formula, where are either $0$ or multiplications of $U^k$.
For the definition of those $\Phi$ maps, one can see \cite{[MO]link_surgery} Section 7 or \cite{[YL]Surgery_2-bridge_link} Section 4.

Combining this with conjugation symmetry, we determine the maps $\Phi^{\pm L_i}_{\mathrm{\bf{s}}}$ by the coefficients in the multi-variable Alexander polynomials of the sublinks in $L$ and the linking numbers. We also show that in the perturbed surgery complex, $\Phi^{\pm L_1 \cup \pm{L_2}}_{\mathrm{\bf{s}}}=0$ for all ${\mathrm{\bf{s}}}\in \bb{H}(L).$ For higher diagonal maps, more information is needed. For 2-component case, we write down explicit formulas.

\subsection{Conjugation symmetry of inclusion maps.}
\begin{defn}[$p^{\overrightarrow{M}}(\mathrm{\bf s})$] For $\mathrm{\bf s}\in \bar{\bb{H}}(L)$ and $\overrightarrow{M}\subset L$, we define $p^{\V{M}}(\mathrm{\bf s})=(p_1^{\V{M}}(s_1),...,p_l^{\V{M}}(s_l))$ by the following formulas
\[p_i^{\V{M}}(s)=\begin{cases}
+\infty &\mbox{if } L_i\subset M \text{ has the induced orientation from } L;\\
-\infty &\mbox{if } L_i\subset M \text{ has the opposite orientation from } L;\\
s       &\mbox{if }  L_i\not\subset M.
\end{cases}
\]
\end{defn}

\begin{defn}[$n_{\mathrm{\bf{s}}}^{\overrightarrow M}(L)$]
 Suppose $\overrightarrow{L}$ is an oriented $l$-component $L$-space link and $\overrightarrow{M} \subset \overrightarrow{L}$ is a sublink which might not have the induced orientation.  Choose a
Heegaard diagram $\mathcal{H}$ of $L$. The inclusion map $I_{\mathrm{\bf{s}}}^{\overrightarrow M}:\mathfrak{A}^-(\mathcal{H},\mathrm{\bf{s}})\to \mathfrak{A}^-(\mathcal {H},p^{\overrightarrow M}(\mathrm{\bf{s}}))$ is a chain map shifting the $\bb{Z}$-grading by a definite amount, which is explicitly expressed in Equation (57) in \cite{[MO]link_surgery}. Thus,  the map induced on
homologies $(\Is{\mathrm{\bf{s}}}{M})_*:H_*\left(\mathfrak{A}^-(\mathcal{H},\mathrm{\bf{s}})\right)\to H_*\left(\mathfrak{A}^-(\mathcal{H},p^{\overrightarrow M}(\mathrm{\bf{s}}))\right)$ is a multiplication of a monomial $U^{k}:\bb{F}[[U]]\to \bb{F}[[U]]$ or $0$ rather than a multiplication of a polynomial. In fact, this map is not 0. Consider the short exact sequence
\[0\to \mathfrak{A}^-_{p^{\V{M}}(\mathrm{\bf s})}\to \mathfrak{A}^-_{\mathrm{\bf s}}\to\mathfrak{A}^-_{\mathrm{\bf s}}/\mathfrak{A}^-_{p^{\V{M}}(\mathrm{\bf s})}\to 0\]
and the induced exact triangle on homology.
 The homology $\mathfrak{A}^-_{\mathrm{\bf s}}/\mathfrak{A}^-_{p^{\V{M}}(\mathrm{\bf s})}$ is a torsion $U_1$ module, which is argued similarly as in the proof of Theorem \ref{thm: AlexPoly-of-L-space-link}.

 The integer $k$ does not depend on the choice of $\mathcal H$, and thus we define it to be $n_{\mathrm{\bf{s}}}^{\overrightarrow M}(L)$. When the context is clear, we simply denote it by
 $n_{\mathrm{\bf{s}}}^{\overrightarrow M}$.
\end{defn}

\begin{rem}
When $L$ is a $L$-space knot $K$, these $n_{s}^{\pm K}(K)$'s are just the same as $V_{s}$'s and $H_s$'s defined for knots in \cite{[Wu-Ni]CosmeticSurgery}.
\end{rem}

\begin{lem}[Conjugation symmetry of $n_{\mathrm{\bf{s}}}^{\overrightarrow M}(L)$]
\label{lem:conjugation_symmetry}
Suppose $L$ is an oriented $n$-component $L$-space link. Then
 $$n_{\mathrm{\bf{s}}}^{\overrightarrow M}=n_{-\mathrm{\bf{s}}}^{-\overrightarrow M}, \quad \forall \mathrm{\bf{s}} \in \bb{H}(L), \forall \overrightarrow M\subset L.$$
\end{lem}

\begin{proof}
Choose an admissible basic Heegaard diagram $\mathcal{H}=(\Sigma,\mathbf{\alpha},\mathbf{\beta},{\mathbf{w}}^{H},{\mathbf{z}}^{H})$ for $\V{L}$. In order to distinguish the basepoints in different Heegaard diagrams, we put a superscript $H$ on  $w$ and $z$. Then, $\mathcal{H}'=(-\Sigma,\mathbf{\beta},\mathbf{\alpha},{\mathbf{w}}^{H'},{\mathbf{z}}^{H'})$ is also a Heegaard diagram for $\V{L}$, where ${\mathbf{w}}^H={\mathbf{z}}^{H'},{\mathbf{z}}^H={\mathbf{w}}^{H'}$.

There is an $\bb{F}[[U_1,\dots,U_n]]$-linear isomorphism of chain complexes
\begin{eqnarray*}
h_{\mathrm{\bf{s}}}:  &\AAs{H}{\bf s} \xrightarrow{\quad\quad} &\AAs{H'}{-\bf s},\\
   &\quad\bf{x}\quad   \xmapsto{\phantom{\AAs{H}{\bf s}}} & \bf{x},  \quad\quad\quad\quad  \forall \bf{x}\in\bb{T}_{\alpha}\cap\bb{T}_{\beta}.
\end{eqnarray*}
Actually, for any $\bf{x},\bf{y}\in \bb{T}_{\alpha}\cap\bb{T}_{\beta}$ and a class $\phi\in \pi_2(\bf{x},\bf{y})$, the moduli space of holomorphic disks $M(\phi,\mathcal{H})$ is identical to $M(\phi,\mathcal{H}')$. See Theorem 2.4 in \cite{[OS]HF2}. Moreover, it is not hard to see that the Alexander gradings are of opposite signs
\[ A(\mathbf{x},\mathcal H)= -A(\mathbf{x},\mathcal{H}').
\]
Thus, we just need to show $h_{\mathrm{\bf{s}}}$ is a chain map, i.e.
\begin{align*}
\partial_{\mathrm{-s}}^{\mathcal{H}'}(h_{\mathrm{\bf{s}}}({\bf{x}}))&= \sum_{{\bf{y}}\in \bb{T}_{\alpha} \cap \bb{T}_{\beta}}\sum_{\phi \in \pi_2(x,y),\mu(\phi)=1} \#(M(\phi)/{\bb{R}})\cdot U_1^{E_{-s_1}^{\mathcal{H}'}(\phi)}\cdots U_n^{E_{-s_n}^{\mathcal{H}'}(\phi)}\cdot{\bf{y}}\\
=h_{\mathrm{\bf{s}}}(\partial_{\mathrm{\bf{s}}}^{\mathcal{H}}{(\bf{x}}))&= \sum_{{\bf{y}}\in \bb{T}_{\alpha} \cap \bb{T}_{\beta}}\sum_{\phi \in \pi_2(x,y),\mu(\phi)=1} \#(M(\phi)/{\bb{R}})\cdot U_1^{E_{s_1}^{\mathcal{H}}(\phi)} \cdots  U_n^{E_{s_n}^{\mathcal H}(\phi)}\cdot{\bf{y}}.
\end{align*}
In fact, by Equation \eqref{eq:E_s}, $\forall \phi\in \pi_2({\bf x}, {\bf y}), \ \forall 1\leq i\leq n,$
\begin{align*}
E_{-s_i}^{\mathcal{H}'}(\phi)&=\max(-s_i-A_i^{\mathcal{H}'}({\bf{x}}),0)-\max(-s_i-A_i^{\mathcal{H}'}({\bf{y}}),0)
+n_{z_i^{\mathcal{H}'}}(\phi)\\
&=\max(-s_i+A_i^{\mathcal{H}}({\bf{x}}),0)-\max(-s_i+A_i^{\mathcal{H}}({\bf{y}}),0)+n_{w_i^{\mathcal{H}}}(\phi)\\
&=E_{s_i}^{\mathcal{H}}(\phi).
\end{align*}

Moreover, by direct computation, we have the following commuting diagram
\[
\xymatrix{
\mathfrak{A}^-(\mathcal{H},\mathrm{\bf{s}})\ar[d]_{I_{\mathrm{\bf{s}}}^{\overrightarrow M}(\mathcal H)}\ar[r]^{h_{\mathrm{\bf{s}}}} &\mathfrak{A}^-(\mathcal{H}',\mathrm{-\bf{s}})\ar[d]^{I_{-\mathrm{\bf{s}}}^{-\overrightarrow M}({\mathcal H}')}\\
\mathfrak{A}^-(\mathcal{H},\mathrm{p^{\overrightarrow M}(\mathrm{\bf{s}})})\ar[r]_{h_{p^{\overrightarrow M}(\mathrm{\bf{s}})}} &\mathfrak{A}^-(\mathcal{H}',\mathrm{-p^{\overrightarrow M}(\mathrm{\bf{s}})}).}
\]
Thus, it follows that \[n_{\mathrm{\bf{s}}}^{\overrightarrow M}=n_{-\mathrm{\bf{s}}}^{-\overrightarrow M}, \quad \forall \mathrm{\bf{s}} \in \bb{H}(L), \ \forall \overrightarrow M\subset L.\]
\end{proof}

\subsection{Perturbed link surgery formula for $2$-component $L$-space links.}
We review the link surgery formula of Manolescu-Ozsv\'{a}th for a 2-component link $L$. See \cite{[MO]link_surgery} and Section 4 in \cite{[YL]Surgery_2-bridge_link}. We need some notations. Denote the set of orientations on a link $N$ by $\Omega(N)$.  We define some projection maps by  $p^{\pm L_1}(s_1,s_2)=(\pm \infty, s_2),$ $p^{\pm L_2}(s_1,s_2)=(s_1,\pm\infty),$ and $p^{\pm L_1\cup\pm L_2}(s_1,s_2)=(\pm\infty,\pm\infty).$

Choose an admissible basic Heegaard diagram $\mathcal{H}$ and denote $\mathfrak{A}^-(\mathcal{H},\mathrm{\bf{s}})$ by $\mathfrak{A}^-_{\mathrm{\bf{s}}}$.  Then, the Manolescu-Ozsv\'{a}th surgery complex $(C^{-}(\mathcal{H},\Lambda),D^{-}(\Lambda))$ is as follows:
\begin{equation}
\label{eq:Manolescu-Ozsv\'{a}th_link_surgery_formula}
(C^{-}(\mathcal{H},\Lambda),D^{-}(\Lambda)):=\xyC{2pc}\xyR{2pc}\xymatrix{{\displaystyle \prod_{(s_{1},s_{2})\in\mathbb{H}(L)}}\mathfrak{A}_{s_{1},s_{2}}^{-}\ar[d]_{D_{00}^{01}(\Lambda)}\ar[r]^{D_{00}^{10}(\Lambda)}\ar[dr]|-{D_{00}^{11}(\Lambda)} & {\displaystyle \prod_{(s_{1},s_{2})\in\mathbb{H}(L)}}\mathfrak{A}_{+\infty,s_{2}}^{-}\ar[d]^{D_{10}^{01}(\Lambda)}\\
{\displaystyle \prod_{(s_{1},s_{2})\in\mathbb{H}(L)}}\mathfrak{A}_{s_{1},+\infty}^{-}\ar[r]_{D_{01}^{10}(\Lambda)} & {\displaystyle \prod_{(s_{1},s_{2})\in\mathbb{H}(L)}}\mathfrak{A}_{+\infty,+\infty}^{-},
}
\end{equation}
where $\forall \delta_1,\delta_2,\varepsilon_1,\varepsilon_2\in \{0,1\},$
\begin{equation}
\label{eq:D}
D^{\delta_1\delta_2}_{\varepsilon_1\varepsilon_2}(\Lambda)
=\prod_{(s_1,s_2)\in\bb{H}(L)}
\left(\sum_{\overrightarrow{M}\in\Omega(\delta_1L_1\cup\delta_2L_2)}\Phi^{\overrightarrow{M}}_{p^{+\varepsilon_1L_1\cup+\varepsilon_2L_2}(s_1,s_2)}\right).
\end{equation}

The Manolescu-Ozsv\'{a}th surgery complex is in the category of complexes of $\bb{F}[[U_1]]$-modules, ${\bf{Ch}}$. Inspired by the idea of  homotopy category ${\bf{K}}$ of $\bb{F}[[U_1]]$-modules, we can replace the complexes on the vertices of the hypercube by its chain homotopy type and replace the maps on the edges by its homotopy type. Then, the Manolescu-Ozsv\'{a}th surgery complex becomes a \emph{perturbed surgery formula}.

\begin{lem}
\label{lem:n_s}
Let $\overrightarrow{L}=\overrightarrow{L_1}\cup \overrightarrow{L_2}$ be an $L$-space link. Then the Heegaard Floer homologies on all the surgeries ${\bf{HF}}^-(S^3_{\Lambda}(L))$ and their absolute gradings are determined by $\{n_{\mathrm{\bf{s}}}^{+L_1}(L)\}_{\mathrm{\bf{s}}\in \bb{H}(L)}$ and $\{n_{\mathrm{\bf{s}}}^{+L_2}(L)\}_{\mathrm{\bf{s}}\in \bb{H}(L)}$.
\end{lem}

\begin{proof}
We restrict our scalars to $\bb{F}[[U_1]]$ from now on. Consider the chain complex $\bb{F}[[U_1]]$, which is freely generated by a single element over $\bb{F}[[U_1]]$ with 0 differential.  Since $L$ is an $L$-space link, i.e. $H_*(\mathfrak{A}^-_{\mathrm{\bf{s}}}(L))=\bb{F}[[U]], \forall \mathrm{\bf{s}}\in \bb{H}(L)$, $\mathfrak{A}^-_{\mathrm{\bf{s}}}(L)$ is in fact chain homotopic to $\bb{F}[[U_1]]$ by Corollary 5.6 in \cite{[YL]Surgery_2-bridge_link} as a $\bb{Z}$-graded $\bb{F}[[U_1]]$-module with $U_1$ lowering grading by 2.

Thus, we can replace every $\mathfrak{A}^-_{\mathrm{\bf{s}}}$ by $\tilde{\mathfrak{A}}^-_{\mathrm{\bf{s}}}$ which is isomorphic to $\bb{F}[[U_1]]$ with 0 differentials and replace the maps correspondingly so as to get a new complex
$(\tilde{C}^-(\mathcal{H},\Lambda),\tilde{D}^-(\Lambda))$.  We call it the \emph{perturbed surgery complex}, and it is chain homotopic to the original one.

More concretely, we first replace the edge maps in the squares in Equation \eqref{eq:Manolescu-Ozsv\'{a}th_link_surgery_formula}
$\Phi_{\mathrm{\bf{s}}}^{\pm L_i}$
by
\[
\label{eq:Phi-tilde}
\tilde{\Phi}_{\mathrm{\bf{s}}}^{\pm L_i}=U_1^{n_{\mathrm{\bf{s}}}^{\pm L_i}}:\bb{F}[[U_1]]\to \bb{F}[[U_1]].
\]
Next, we replace the diagonal maps $\Phi_{\mathrm{\bf{s}}}^{\pm L_1\cup \pm L_2}$ by
\[\tilde{\Phi}_{\mathrm{\bf{s}}}^{\pm L_1\cup \pm L_2}=0.\]
The reason we replace the diagonal maps by 0 is that, in the link surgery complex, the $\bb{F}[[U_1]]$-linear diagonal maps always shift the $\bb{Z}$-gradings by an odd number.

Finally, we get the new perturbed surgery complex $\tilde{C}(\Lambda)$ as follows:
\begin{equation}
\label{eq:perturbed_surgery_complex}
(\tilde{C}^{-}(\mathcal{H},\Lambda),\tilde{D}^{-}(\Lambda)):=
\xyC{2pc}\xyR{2pc}\xymatrix{
{\displaystyle\prod_{(s_{1},s_{2})\in\mathbb{H}(L)}}\tilde{\mathfrak{A}}_{s_{1},s_{2}}^{-}
\ar[d]_{\tilde{D}_{00}^{01}(\Lambda)}\ar[r]^{\tilde{D}_{00}^{10}(\Lambda)} \ar[dr]|-{\tilde{D}_{00}^{11}(\Lambda)}
& {\displaystyle \prod_{(s_{1},s_{2})\in\mathbb{H}(L)}}\tilde{\mathfrak{A}}_{+\infty,s_{2}}^{-}\ar[d]^{\tilde{D}_{10}^{01}(\Lambda)}\\
{\displaystyle \prod_{(s_{1},s_{2})\in\mathbb{H}(L)}}\tilde{\mathfrak{A}}_{s_{1},+\infty}^{-}\ar[r]_{\tilde{D}_{01}^{10}(\Lambda)} & {\displaystyle \prod_{(s_{1},s_{2})\in\mathbb{H}(L)}}\tilde{\mathfrak{A}}_{+\infty,+\infty}^{-},
}
\end{equation}
where
\begin{equation}
\label{eq:D-tilde}
\tilde{D}^{\delta_1\delta_2}_{\varepsilon_1\varepsilon_2}(\Lambda)
=\prod_{(s_1,s_2)\in\bb{H}(L)}
\left(\sum_{\overrightarrow{M}\in\Omega(\delta_1L_1\cup\delta_2L_2)}
\tilde{\Phi}^{\overrightarrow{M}}_{p^{+\varepsilon_1L_1\cup+\varepsilon_2L_2}(s_1,s_2)}\right), \quad \delta_1,\delta_2,\varepsilon_1,\varepsilon_2\in \{0,1\}.
\end{equation}

The perturbed complex $\tilde{C}(\Lambda)$ is chain homotopy equivalent to the original surgery complex as $\bb{F}[[U_1]]$-modules. Moreover, this chain homotopy equivalence is preserving the $\bb{Z}$-grading on it.  For more details, see Section 5.6 in \cite{[YL]Surgery_2-bridge_link}.

Hence, we have $H_*(\tilde{C}(\Lambda))\cong {\bf{HF}}^-\left(S^3_{\Lambda}(L)\right)$ as an $\bb{F}[[U_1]]$-module. By Link Surgery Theorem in \cite{[MO]link_surgery}, we have $U_i$ actions on the homology of ${\bf{HF}}^-\left(S^3_{\Lambda}(L)\right)$ are all the same, i.e.
\[  {\bf{HF}}^-\left(S^3_{\Lambda}(L)\right)=H_*(\tilde{C}(\Lambda))\otimes_{\bb{F}[[U_1]]}\bb{F}[[U_1,U_2]]/(U_1-U_2).
\]

All the inputs of $\tilde{C}(\Lambda)$ are $\{n_{\mathrm{\bf{s}}}^{\pm L_1}(L)\}_{s\in \bb{H}(L)}$ and $\{n_{\mathrm{\bf{s}}}^{\pm L_2}(L)\}_{s\in \bb{H}(L)}$. Thus,  the proof is done by Lemma \ref{lem:conjugation_symmetry}. To compute the absolute grading for $\bf{HF}^-$, we only need to shift the absolute $\bb{Z}$-grading by $\frac{\mathrm{c_1}^2(\mathrm{\bf{s}})-2\chi-3\sigma}{4}$ which can be computed from $\Lambda.$
\end{proof}

\subsection{Redefining knot Floer homology.}

We redefine the knot Floer homology by using slightly generalized Heegaard diagrams with extra basepoints. The reason we consider these diagrams is that they are used in the proof of Theorem \ref{thm:L-space_surgery}. In \cite{[M]HFK}, there are many generalized versions of knot Floer complex and homology discussed. Since the version in this subsection is not presented in \cite{[M]HFK}, we define it here.

\begin{enumerate}
\item Heegaard diagram: We choose a Heegaard diagram  $\mathcal{H}=(\Sigma,{\mathbf{\alpha}},{\mathbf{\beta}},\{w_1,...,w_k\},\{z_1\})$.
\item Alexander grading: For any ${\bf x}\in \bb{T}_{\alpha}\cap \bb{T}_{\beta}$,
      \[
       A(U_1^{n_1}\cdots U_k^{n_k}{\bf x})=A({\bf x})-n_1.
       \]
\item Alexander filtration: The complex ${\bf{CF}}^-(S^3)$ is freely generated by ${\bf x}\in \bb{T}_{\alpha}\cap \bb{T}_{\beta}$ over $\bb{F}[[U_1,U_2,...,U_k]]$ and the differentials are counting holomorphic disks. For $\forall {\bf x}\in \bb{T}_{\alpha}\cap\bb{T}_{\beta}$,  we have $A(\partial {\bf x})\leq A({\bf x})$. This is because for a pseudo-holomorphic disk in $ \phi\in \pi_2({\bf x},{\bf y})$, $n_{z_1}(\phi)\geq 0$ and
    \[A({\bf x})=A({\bf y})+n_{z_1}(\phi)-n_{w_1}(\phi)=A(U_1^{n_{w_1}(\phi)}\cdot\dots U_k^{n_{w_k}(\phi)}\cdot {\bf y})+n_{z_1}(\phi).
    \]
    Thus, the Alexander grading induces a filtration on ${\bf{CF}}^-(S^3)$. We define the subcomplex
    $$\mathfrak{A}^-_s(K):=\{x\in {\bf{CF}}^-(S^3)|A(x)\leq s\}.$$
\item The \emph{filtered minus  knot Floer homology}: We define the chain complex $CFK^-(K,s)=\mathfrak{A}^-_s/{\mathfrak{A}^-_{s-1}}$ and $HFK^-(K,s)=H_*(CFK^-(K,s)).$
\item The \emph{total minus knot Floer homology}: We define the chain complex $gCFK^-(K)$ to be freely generated by $\bb{T}_\alpha\cap \bb{T}_\beta$, and $\forall {\mathbf x}\in \bb{T}_\alpha \cap \bb{T}_\beta$
    $$\partial {\mathbf x}=\sum_{{\mathbf y}\in \bb{T}_\alpha \cap \bb{T}_\beta}\sum\limits_{\begin{array}{c}\phi\in \pi_2({\mathbf x},{\mathbf y})\\ \mu(\phi)=1,n_{z_1}(\phi)=0\end{array}} \#(M(\phi)/\bb{R})\cdot {U_1^{n_{w_1}(\phi)}}{\cdots}{U_k^{n_{w_k}(\phi)}}\cdot {\mathbf y}.$$
    The homology $HFK^-(K)$ is defined to be the homology of $gCFK^-(K)$.
\end{enumerate}

\begin{rem}
Considered only as $\bb{F}$-vector spaces, $HFK^-(K)=\bigoplus_{s\in \bb{Z}}HFK^-(K,s)$. However, considered as $\bb{F}[[U_1,...,U_k]]$-modules, $HFK^-(K,s)$ is the associated graded of a filtration on $HFK^-(K)$. Note that $HFK^-(K,s)$'s are always torsion modules.
\end{rem}

\begin{prop}
\label{prop:KnotFloerMinus}
Suppose $K\subset S^3$ is a knot. For a multi-pointed Heegaard diagram $\mathcal{H}=(\Sigma,{\mathbf{\alpha}},{\mathbf{\beta}},\{w_1,...,w_k\},\{z_1\})$ for $K$, we have the following:
\begin{enumerate}
\item The knot Floer homology $HFK^-(K,s)$ is an $\bb{F}[[U]]:=\bb{F}[[U_1,...,U_k]]/(U_2,...,U_k)$-module, and does not depend on $\mathcal H$ considered as an $\bb{F}[[U]]$-module.
\item We have the following identity
\begin{equation}
\label{eq:chi_HFK^-}
\sum_{s\in \bb{Z}}\mbox{\Large$\chi$}\left(HFK^-(K,s)\right)\cdot t^s \overset{\centerdot}{=}\frac{1}{t-1}\Delta_K(t).
\end{equation}
\end{enumerate}
\end{prop}

\begin{proof} This is actually a direct corollary of Theorem 4.10 in \cite{[MO]link_surgery}.
 There are six types of Heegaard moves  according to \cite{[MO]link_surgery},
\begin{enumerate}[(i)]
\item  3-manifold isotopy;
\item  $\alpha$-curve isotopy and $\beta$-curve isotopy;
\item  $\alpha$-handleslide and $\beta$-handleslide;
\item  index one/two stabilizations;
\item  free index zero/three stabilizations;
\item  free index zero/three link stabilizations.
\end{enumerate}

By Proposition 4.13 in \cite{[MO]link_surgery}, we only need to check how the knot Floer homology changes under these Heegaard moves and their inverses.

The Heegaard moves of types (i) to (iv) are chain homotopy equivalences preserving the Alexander filtration, and thus do not change the knot Floer homology.

A Heegaard move of type (v) changes the chain complex ${\bf{CF}}^-(\mathcal{H})$ into ${\bf{CF}}^-(\mathcal{H}')$, which is the mapping cone ${\bf{CF}}^-(\mathcal{H})[[U_{k+1}]]\xrightarrow{U_{k+1}-U_{i_0}} {\bf{CF}}^-({\mathcal H})[[U_{k+1}]]$. Notice that $U_{k+1}$ does not change the
 Alexander grading. Thus, if $i_0\neq 1$, then $CFK^-(\mathcal{H}',s)$ is the mapping cone
\[
CFK^-(\mathcal{H},s)[[U_{k+1}]]\xrightarrow{U_{k+1}-U_{i_0}} CFK^-(\mathcal{H},s)[[U_{k+1}]].
\]
If $i_0=1$, then $CFK^-(\mathcal{H}',s)$ is the mapping cone
\[
CFK^-(\mathcal{H},s)[[U_{k+1}]]\xrightarrow{U_{k+1}} CFK^-(\mathcal{H},s)[[U_{k+1}]].
\]
In both cases, we have that the homology of the mapping cone is $$HFK^-(\mathcal{H},s)\otimes_{\mathcal R}{\bb{F}}[[U_1,...,U_{k+1}]]/(U_2,...,U_{k+1}),$$ where ${\mathcal R}=\bb{F}[[U_1,...,U_k]]$.

The Heegaard move of type (vi) changes the complex $CFK^-(\mathcal{H},s)$ by
\[ CFK^-(\mathcal{H},s)\otimes H_{*}(S^1)\cong CFK^-(\mathcal{H},s)
\oplus CFK^-(\mathcal{H},s).
\]
However, if  $\mathcal{H}_1$ and $\mathcal{H}_2$ are equivalent Heegaard diagrams
both with a single pair of basepoints on $K$, then total number
of copies of $HFK^-({\mathcal H}_1,s)$'s in $HFK^-(\mathcal{H}_2,s)$ is one.
\end{proof}

\subsection{Reduction of Heegaard diagrams.}

Let $\mathcal{H}$ be Heegaard diagram for a link $L$. Then there are several Heegaard diagrams $r_{\overrightarrow M}({\mathcal H})$ of the sublinks of $L$ reduced
from $\mathcal{H}$. See Definition 4.17 in \cite{[MO]link_surgery}.

\begin{lem}
\label{lem:reduction of links}
Let $\V{L}=\V{L_1}\cup \V{L_2}$ be a link and $\mathcal{H}=(\Sigma,\mathbf{\alpha},\mathbf{\beta},\{w_1,w_2\},\{z_1,z_2\})$
be a Heegaard diagram for $\V{L}$. Denote $\mathfrak{A}^-(\mathcal{H},(s_1,s_2))$ by $\mathfrak{A}^-_{s_1,s_2}$, for all $(s_1,s_2)\in \bb{H}(L)$. Then
$$H_*(\mathfrak{A}^-_{+\infty,s_2}/\mathfrak{A}^-_{+\infty,s_2-1})=HFK^-(L_2,s_2-\frac{\mathrm{lk}}{2}). $$
In particular, $\mbox{\Large$\chi$}\left(H_*(\mathfrak{A}^-_{+\infty,s_2}/\mathfrak{A}^-_{+\infty,s_2-1})\right)$ is determined by the Alexander polynomial $\Delta_{L_2}(t)$.
\end{lem}

\begin{proof}By Proposition \ref{prop:KnotFloerMinus}, we can use $\mathfrak{A}^-_{+\infty,s_2}/\mathfrak{A}^-_{+\infty,s_2-1}$ to compute the knot Floer homology of $L_2$. The only issue is on the Alexander grading.
From Equation (36) in \cite{[MO]link_surgery},  there is an identification
$$\mathfrak{A}^-(\mathcal{H},p^{\overrightarrow M}(\mathrm{\bf s}))\xrightarrow{\cong}\mathfrak{A}^-(r_{\overrightarrow M}({\mathcal H}),\psi^{\overrightarrow M}(\mathrm{\bf s})).$$
Note that the definition of $\psi^{\V{M}}(\mathrm{\bf s})$ involves the linking numbers.
Thus, we have the following commuting diagram
\begin{equation*}
\xymatrix{\mathfrak{A}^-_{+\infty,s_2-1}(L)\ar[r]^{\cong}\ar[d]_{\iota^{+L_2}_{+\infty,s_2-1}} &\mathfrak{A}^-_{s_2-1-\frac{\mathrm{lk}}{2}}(L_2)\ar[d]^{\iota^{+L_2}_{s_2-1-\frac{\mathrm{lk}}{2}}}\\
\mathfrak{A}^-_{+\infty,s_2}(L)\ar[r]^{\cong}                   &\mathfrak{A}^-_{s_2-\frac{\mathrm{lk}}{2}}(L_2),}
\end{equation*}
where $\iota^{+L_2}_{+\infty,s_2-1}$ and $\iota^{+L_2}_{s_2-1-\frac{\mathrm{lk}}{2}}$ are both the inclusions of subcomplex. Thus,  we have

\[
\frac{\mathfrak{A}_{+\infty,s_{2}}^{-}(L)}{\mathfrak{A}_{+\infty,s_{2}-1}^{-}(L)}\cong
\frac{\mathfrak{A}_{s_{2}-\frac{\mathrm{lk}}{2}}^{-}(L_2)}{\mathfrak{A}_{s_{2}-1-\frac{\mathrm{lk}}{2}}^{-}(L_2)}
=CFK^-(L_2,s_{2}-\frac{\mathrm{lk}}{2}).
\]
Thus, the lemma follows.
\end{proof}

\subsection{Proof of Theorem \ref{thm:L-space_surgery}.}

\begin{proof}
Consider the following factorization of inclusion maps of subcomplexes
\[
I_{s_{1},s_{2}}^{+L_{2}}:\mathfrak{A}_{s_{1},s_{2}}^{-}\xrightarrow{\iota_{s_{1},s_{2}}^{+L_{2}}}\mathfrak{A}_{s_{1},s_{2}+1}^{-}
\xrightarrow{I_{s_{1},s_{2}+1}^{+L_{2}}}\mathfrak{A}_{s_{1},+\infty}^{-}.
\]
It induces a factorization of the maps on homology $(I_{s_{1},s_{2}}^{+L_{2}})_{*}=(I_{s_{1},s_{2}+1}^{+L_{2}})_{*}\circ(\iota^{+L_2}_{s_{1},s_{2}})_{*}.$
As is discussed in the proof of Theorem \ref{thm: AlexPoly-of-L-space-link}, we see $(\iota_{s_{1},s_{2}}^{+L_{2}})_{*}$
is a multiplication of $U^{k^{+L_2}_{s_1,s_2}}$, where
\[
k^{+L_2}_{s_1,s_2}=n_{s_{1},s_{2}}^{+L_{2}}-n_{s_{1},s_{2}+1}^{+L_{2}}.
\]

Moreover, $k=0$ if and only if $H_{*}(\mathfrak{A}_{s_{1},s_{2}+1}^{-}/\mathfrak{A}_{s_{1},s_{2}}^{-})=0$,
and $k=1$ if and only if $H_{*}(\mathfrak{A}_{s_{1},s_{2}+1}^{-}/\mathfrak{A}_{s_{1},s_{2}}^{-})=\mathbb{F}$ with an even grading.
 Then, we have
\[ \mbox{\Large$\chi$}\left(H_{*}(\mathfrak{A}_{s_{1},s_{2}+1}^{-}/\mathfrak{A}_{s_{1},s_{2}}^{-})\right)
=n_{s_{1},s_{2}}^{+L_{2}}-n_{s_{1},s_{2}+1}^{+L_{2}}.
\]
Whereas,
\begin{align*}
&\mbox{\Large$\chi$}\left(H_{*}(\mathfrak{A}_{s_1+k,s_{2}+1}^{-}/\mathfrak{A}_{s_1+k,s_{2}}^{-})\right)\\
=&\mbox{\Large$\chi$}\left(H_{*}(\mathfrak{A}_{s_{1},s_{2}+1}^{-}/\mathfrak{A}_{s_{1},s_{2}}^{-})\right)
+\sum_{i=1}^{k}\mbox{\Large$\chi$}\left(HFL^-(L,s_1+i,s_2+1)\right), \forall k>0.
\end{align*}
 Let $k\to \infty$, and then we have $\mbox{\Large$\chi$}\left(H_{*}(\mathfrak{A}_{k,s_{2}}^{-}/\mathfrak{A}_{k,s_{2}-1}^{-})\right)=
 \mbox{\Large$\chi$}\left(H_{*}(\mathfrak{A}_{+\infty,s_{2}}^{-}/\mathfrak{A}_{+\infty,s_{2}-1}^{-})\right)$
 determined by $\Delta_{L_{2}}(t)$, by Lemma \ref{lem:reduction of links}. Thus, all the $n_{s_1,s_2}^{+L_2}$ are determined by the Alexander polynomials. Similar results hold for $L_1.$ The theorem follows from Lemma \ref{lem:n_s} and Theorem \ref{thm:AlexPolyMurasugi}.
\end{proof}

In fact, when the linking number is not $0$, the Alexander polynomials of $L_1$ and $L_2$ are determined by the Alexander polynomial of $L=L_1\cup L_2$ and the linking number:

\begin{thm}[Murasugi, Proposition 4.1 in \cite{[Murasugi]on_Periodic_knots}]
\label{thm:AlexPolyMurasugi}
Let $\Delta_{L}(x,y)$ and $\Delta_{L_1}(t)$ be the Alexander polynomial of a link $L=L_1\cup L_2$ and $L_1$ respectively in $S^3$. Then \[ \Delta_{L}(t,1)\overset{\centerdot}{=}\frac{1-t^{\mathrm{lk}}}{1-t}\Delta_{L_1}(t),
\]
where $\mathrm{lk}$ is the linking number of $L$.
\end{thm}

\subsection{Formulas for $n_{\mathrm{\bf{s}}}^{\pm L_i}(L)$'s.}
Using the Alexander polynomials of $L,L_1,L_2$, we can get formulas for $n_{\mathrm{\bf{s}}}^{\pm L_i}(L)$'s.

First of all, we fix the overall signs of these Alexander polynomials to get normalization of  Equation \eqref{eq:chi_HFK^-} and Equation \eqref{eq:chi_HFL^-}:
\begin{eqnarray}
\label{eq:normalized_chi_HFK^-}
&\sum_{s\in \bb{Z}}\mbox{\Large$\chi$}\left(HFK^-(K,s)\right)\cdot t^s=\frac{t}{t-1}\Delta_K(t),\\
\label{eq:normalized_chi_HFL^-}
&\sum_{(s_1,s_2)\in \bb{H}(L)}\mbox{\Large$\chi$}(HFL^-(L,s_1,s_2))\cdot x_1^{s_1}\cdot x_2^{s_2}=x_1^{\frac{1}{2}}x_2^{\frac{1}{2}}\Delta_L(x_1,x_2).
\end{eqnarray}

For an $L$-space knot $K$, to get Equation \eqref{eq:normalized_chi_HFK^-}, we require that $\frac{t}{t-1}\Delta_K(t)$ has finitely many non-zero positive powers and all the non-zero coefficients of $\frac{t}{t-1}\Delta_K(t)$ are $1$, which is equivalent to $\Delta_K(1)=1$.

\begin{thm}
\label{thm:normalization_of_L-space_link}
Suppose $L=L_1\cup L_2$ is an $L$-space link. Let $\Delta_{L_1}(t)$, $\Delta_{L_2}(t)$, and $\Delta_{L}(x_1,x_2)$ be the symmetrized Alexander polynomials, such that $\Delta_{L_1}(1)=\Delta_{L_2}(1)=1.$ Let
\begin{align*}
&\frac{t}{t-1}\Delta_{L_1}(t)=\sum_{k\in \bb{Z}}a^{L_1}_k\cdot t^k,\\
&\frac{t}{t-1}\Delta_{L_2}(t)=\sum_{k\in \bb{Z}}a^{L_2}_k\cdot t^k,\\
&\Delta_{L}(x_1,x_2)=\sum_{i,j}a^L_{i,j}\cdot x_1^{i}\cdot x_2^{j}.
\end{align*}

Suppose $(i_0,j_0)$ satisfies that $a^L_{i_0,j_0}\neq 0,$  $a^L_{i,j_0}=0$ for all $i>i_0,$ and $a^L_{i_0,j}=0$  for all $j>j_0.$
Then,
\begin{itemize}
\item  $\mbox{\Large$\chi$}\left(HFL^-(L,i_0+\frac{1}{2},j_0+\frac{1}{2})\right)=1$ if and only if   $a^{L_1}_{i_0+\frac{1}{2}-\frac{\mathrm{lk}}{2}}=a^{L_2}_{j_0+\frac{1}{2}-\frac{\mathrm{lk}}{2}}=1;$
\item  $\mbox{\Large$\chi$}\left(HFL^-(L,i_0+\frac{1}{2},j_0+\frac{1}{2})\right)=-1$ if and only if  $a^{L_1}_{i_0+\frac{1}{2}-\frac{\mathrm{lk}}{2}}=a^{L_2}_{j_0+\frac{1}{2}-\frac{\mathrm{lk}}{2}}=0.$
\end{itemize}
\end{thm}
\begin{proof}
Notice that $\mbox{\Large$\chi$}\left(\mathfrak{A}^-_{s_1,s_2}/\mathfrak{A}^-_{s_1,s_2-1}\right)$ can only be $0$ or $1$ for all $(s_1,s_2)\in \bb{H}(L).$
By Equation \eqref{eq:L-space_link_AlexPoly}, we have two possible cases:
\begin{enumerate}[(a)]
\item $\bigchi\left( HFL^-(L,s_1,s_2)\right)=1$ if and only if $\bigchi\left(\mathfrak{A}^-_{s_1,s_2}/\mathfrak{A}^-_{s_1,s_2-1}\right)=1$ and $\bigchi\left(\mathfrak{A}^-_{s_1-1,s_2}/\mathfrak{A}^-_{s_1-1,s_2-1}\right)=0$;
\item $\bigchi\left( HFL^-(L,s_1,s_2)\right)=-1$ if and only if  $\bigchi\left(\mathfrak{A}^-_{s_1,s_2}/\mathfrak{A}^-_{s_1,s_2-1}\right)=0$ and
    $\bigchi\left(\mathfrak{A}^-_{s_1-1,s_2}/\mathfrak{A}^-_{s_1-1,s_2-1}\right)=1$.
\end{enumerate}
In addition, we have $$\bigchi\left(\mathfrak{A}^-_{i_0+\frac{1}{2},j_0+\frac{1}{2}}/\mathfrak{A}^-_{i_0+\frac{1}{2},j_0-\frac{1}{2}}\right)
=\bigchi\left(\mathfrak{A}^-_{+\infty,j_0+\frac{1}{2}}/\mathfrak{A}^-_{+\infty,j_0-\frac{1}{2}}\right)
=\bigchi\left(HFK^-(L_2,j_0+\frac{1}{2}-\frac{\mathrm{lk}}{2})\right).$$
So $\bigchi\left( HFL^-(L,s_1,s_2)\right)=1$ if and only if $a^{L_2}_{j_0+\frac{1}{2}-\frac{\mathrm{lk}}{2}}=1.$ Symmetrically, we have $\bigchi\left( HFL^-(L,s_1,s_2)\right)=1$ if and only if $a^{L_1}_{i_0+\frac{1}{2}-\frac{\mathrm{lk}}{2}}=1.$ Similar argument applies to the case (b).
\end{proof}

\begin{defn}[Normalization of Alexander polynomials for $L$-space links]
\label{defn:noralization of ALexPoly}
Suppose $L=L_1\cup L_2$ is an $L$-space link. Let the symmetrized  Alexander polynomial of $L$ be
\[\Delta_{L}(x_1,x_2)=\sum_{i,j}a^{L}_{i,j}\cdot x_1^{i}\cdot x_2^{j},
\]
where $x_i$ corresponds to the component $L_i$ for $i=1,2.$ Let the symmetrized Alexander polynomials of $L_1,L_2$ be  $\Delta_{L_1}(t),\Delta_{L_2}(t)$ in the forms of
\[\frac{t}{t-1}\Delta_{L_1}(t)=\sum_{k\in \bb{Z}}a^{L_1}_k\cdot t^k,\quad \frac{t}{t-1}\Delta_{L_2}(t)=\sum_{k\in \bb{Z}}a^{L_2}_k\cdot t^k.
\]

Let  $(i_0,j_0)$ be such that
\[
j_0=\max\{j\in \bb{Z}+\frac{\mathrm{lk}-1}{2}\ |\ a^L_{i,j}\neq 0\} \text{ and }
i_0=\max\{i\in \bb{Z}+\frac{\mathrm{lk}-1}{2}\ |\ a^L_{i,j_0}\neq 0\} .
\]
 Then, these Alexander polynomials are called \textbf{normalized}, if
\begin{enumerate}
  \item the leading coefficient of $\Delta_{L_i}(t)$ is $1$ for both $i=1,2$, which is equivalent to $\Delta_{L_i}(1)=1$;
  \item if $a^{L_2}_{j_0-\frac{\mathrm{lk}}{2}+\frac{1}{2}}=1$, then $a^L_{i_0,j_0}=1$; while if $a^{L_2}_{j_0-\frac{\mathrm{lk}}{2}+\frac{1}{2}}=0$, then  $a^L_{i_0,j_0}=-1.$
\end{enumerate}
\end{defn}
\

After normalization, we have $\mbox{\Large$\chi$}(HFL^-(L,s_1,s_2))=a^{L}_{s_1-\frac{1}{2},s_2-\frac{1}{2}}$ and $\mbox{\Large$\chi$}(HFK^-(L_i,s))=a_{s}^{L_i}$ for $i=1,2.$ Therefore,
\[
\mbox{\Large{$\chi$}}(H_*(\mathfrak{A}^-_{s_1,s_2}/\mathfrak{A}^-_{s_1,s_2-1}))
=a^{L_2}_{s_2-\frac{\mathrm{lk}}{2}}-\sum_{i=1}^{\infty}a^L_{s_1-\frac{1}{2}+i,s_2-\frac{1}{2}}=0 \ \text{or} \ 1.
\]
Hence, we have
\begin{equation}
\label{eq:n^+L_2}
n_{s_1,s_2}^{+L_2}=
\sum_{j=1}^{\infty}\left(a_{s_2+j-\frac{\mathrm{lk}}{2}}^{L_2}-\sum_{i=1}^{\infty}a_{s_1+i-\frac{1}{2},s_2+j-\frac{1}{2}}^{L}\right).
\end{equation}
Similarly, we have
\begin{equation}
\label{eq:n^+L_1}
n_{s_1,s_2}^{+L_1}=
\sum_{i=1}^{\infty}\left(a_{s_1+i-\frac{\mathrm{lk}}{2}}^{L_1}-\sum_{j=1}^{\infty}a_{s_1+i-\frac{1}{2},s_2+j-\frac{1}{2}}^{L}\right).
\end{equation}

\begin{thm}
\label{thm:U_powers of inclusion_maps}
Suppose $L=L_1\cup L_2$ is an $L$-space link. Under the normalization in Definition \ref{defn:noralization of ALexPoly}, we have that the formulas in Equation \eqref{eq:n^+L_2} and Equation \eqref{eq:n^+L_1} are non-negative for all $(s_1,s_2)\in \bb{H}(L).$
\end{thm}

In fact, both of Theorem \ref{thm:normalization_of_L-space_link} and Theorem \ref{thm:U_powers of inclusion_maps} give additional constraints for the Alexander polynomials of an $L$-space 2-component link.

\begin{prop}
\label{thm:L7n2}
The link $L7n2$ is not an $L$-space link.
\end{prop}

\begin{proof} We give two proofs based on Theorem \ref{thm:normalization_of_L-space_link} and Theorem \ref{thm:U_powers of inclusion_maps} respectively.
 Suppose $L=L7n2$ is an $L$-space link with components $L_1$ and $L_2$, where $L_1$ is the unknot and $L_2$ is the right-handed trefoil.  Then, we get the normalized Alexander polynomials of $L_1$ and $L_2$:
\begin{align*}
&\frac{t}{t-1}\Delta_{L_1}(t)=1+t^{-1}+t^{-2}+\cdots,\\
&\frac{t}{t-1}\Delta_{L_2}(t)=t+t^{-1}+t^{-2}+\cdots.
\end{align*}
Since $\Delta_{L}(x,y)=\frac{(x-1)(y-1)}{\sqrt{xy}}$ and $\mathrm{lk}=0$, by Theorem \ref{thm:normalization_of_L-space_link}, we have $a^{L_1}_{1}=a^{L_2}_{1}$. This is a contradiction to $a_1^{L_1}=0$ and $a_1^{L_2}=1$.

Another proof is as follows. If we used the normalization in Definition \ref{defn:noralization of ALexPoly} for $L7n2$, then we get $n^{+L_1}_{0,0}=-1$ by Equation \eqref{eq:n^+L_1}.  This is a contradiction to Theorem \ref{thm:U_powers of inclusion_maps}.
\end{proof}

\section{Applications}
Classifying $L$-space surgeries on an $L$-space link $L$ is usually challenging. One difficulty is the lack of criterion for hyperbolic $L$-spaces. In \cite{[Gorsky_Nemithi]algebriac_links}, Gorsky and N\'{e}methi studied $L$-space surgeries on the  torus links $T(pr,qr)$ with $p,q>1$ and $r\geq 1$ using Lisca-Stipsicz characterization of Seifert $L$-spaces.  Let us look at the case where $p=1$ and $r=2$, i.e. the torus links $L=T(2,2n)$.  We assume $n\geq 2$, since the $T(2,2)$ torus link is the Hopf link and its surgeries are lens spaces.

When both of $p$ and $q$ are not equal to $n$, the $(p,q)$-surgery on $T(2,2n)$ is
a Seifert manifold with three singular fibers over the base $S^2$.
Using the notational convention in \cite{[L-S]HF_&tight_contact_str}, we can write $S^3_{p,q}(T(2,2n))=-M(0;\frac{1}{n},\frac{1}{p-n},\frac{1}{q-n})$.
In \cite{[L-S]HF_&tight_contact_str}, Lisca and Stipsicz give a characterization of $L$-space Seifert manifolds.

\begin{thm}[Theorem 1.1, \cite{[L-S]HF_&tight_contact_str}]
Suppose $M$ is an oriented rational homology sphere which is  Seifert fibered over  $S^2$. Then, $M$ is
an $L$-space if and only if either $M$ or $-M$ carries no positive, transverse contact structures.
\end{thm}

\begin{thm}[\cite{[Lisca-Matic]Transverse_Contact_str_Seifert}]
An oriented Seifert rational homology sphere $M=M(e_0;r_1,...,r_k)$ with $1>r_1\geq r_2 \geq \cdots\geq r_k>0$ admits no positive,
transverse contact structure if and only if
\begin{itemize}
\item $e_0(M)\geq 0$, or
\item $e_0(M)=-1$ and there are no relatively prime integers $m>a$ such that
   $$ m r_1<a<m(1-r_2),\text{ and }mr_i<1,\ i=3,...,k.$$
\end{itemize}
\end{thm}

While the $(n,q)$-surgery on $T(2,2n)$ is usually a graph manifold.  The $(n,q)$-surgeries are discussed in Corollary \ref{cor:L-space surgery on T(2,2n)}.  Direct computation gives the following result.

\begin{prop}[Classification of $L$-space surgeries on $T(2,2n)$ with $n\geq 2$]
\label{prop:classification of L-space surgeries on T(2,2n)}
For all $q\neq n$, the $(n,q)$-surgery on $T(2,2n)$ is an $L$-space.

When $p\neq n,q\neq n$ and $p\geq q$,   $S^3_{p,q}(T(2,2n))$ is an $L$-space with  if and only if one of the following conditions holds:
\begin{enumerate}
 \item $n+2\leq p, n+1\leq q$;
 \item $2n\leq p, n-2\geq q$, and there are no relatively prime integers $m>a>0$ such that
 \[m \frac{n-q-1}{n-q}<a<m(1-\frac{1}{n})\text{   and   }\frac{m}{p-n}<1;\]
 \item $n+2\leq p\leq 2n, q\leq n-2$, and there  are no relatively prime integers $m>a>0$ such that
 \[m \frac{n-q-1}{n-q}<a<m(1-\frac{1}{p-n})\text{   and   }\frac{m}{n}<1;\]
 \item $p=n+1,q\leq n+1,$ and  $q\neq n$;
 \item $p=n-1,q\leq n-1$;
 \item $p\leq n-2, q\leq p,$ and there are no relatively prime integers $m>a>0$ such that
 \[m(1-\frac{1}{n})<a<m\frac{1}{n-p}\text{   and   }\frac{m}{n-q}<1.\]
\end{enumerate}
\end{prop}

See Figure \ref{fig:Seifert_L-space_surgery on T(2,20)} for the example of $T(2,20)$. Compare this result with Theorem 7 in \cite{[Gorsky_Nemithi]algebriac_links}.

\begin{figure}
  \centering
  \centering
  \includegraphics[scale=0.5]{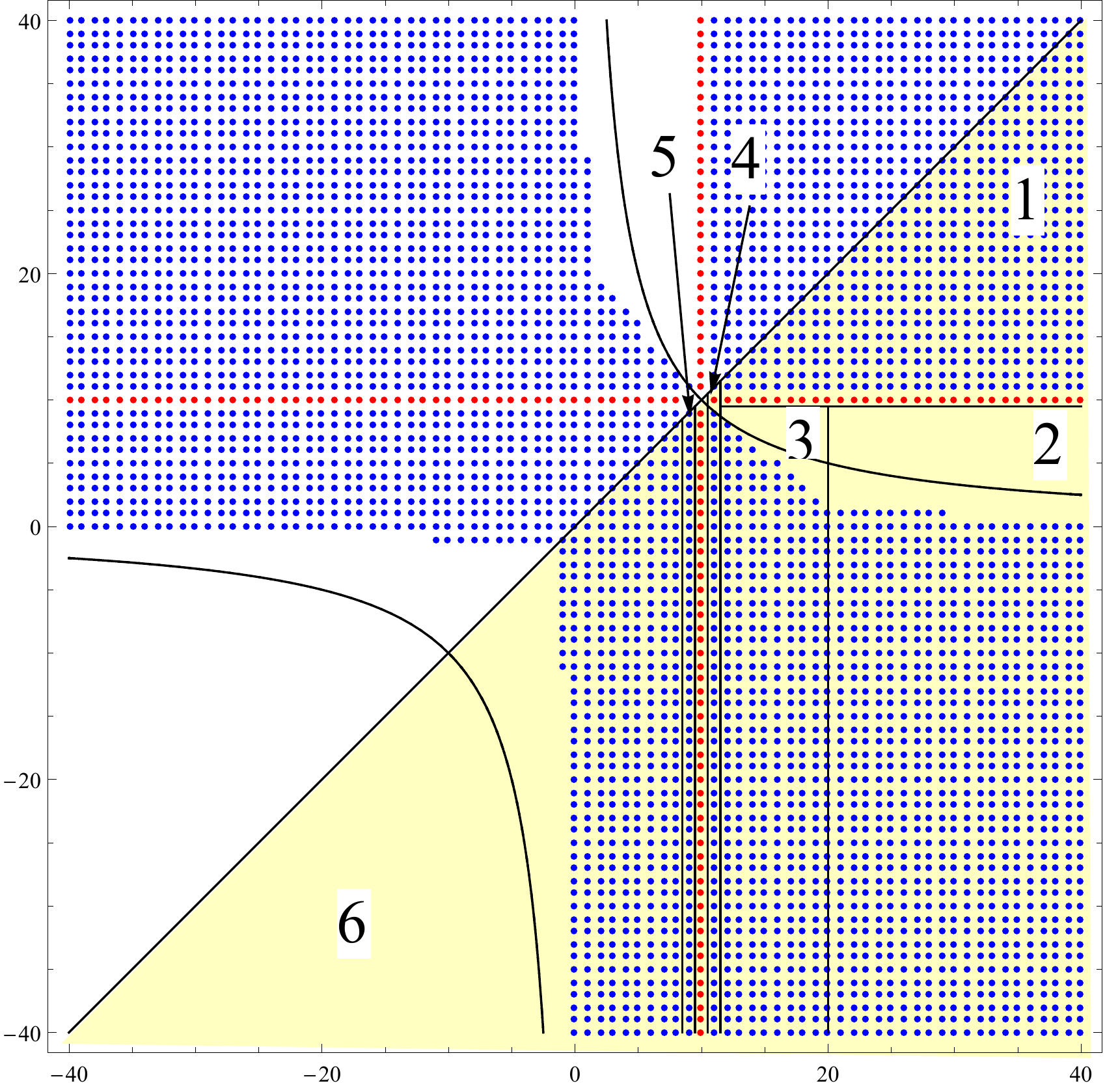}
  \caption{\textbf{The $L$-space surgeries on $T(2,20)$.} We draw the $L$-space surgeries
  of $T(2,20)$ on the x-y plane within the range $[-40,40]\times[-40,40]$.
   Every dot $(p,q)$ represents an $L$-space surgery $(p,q)$. The blue points are Seifert $L$-space surgeries determined by the characterization of Lisca-Stipsicz, while the red points are determined by induction. The six labelled
  regions  correspond to the six conditions (1) to (6) in Proposition \ref{prop:classification of L-space surgeries on T(2,2n)}.
  The drawn hyperbola indicates the positions of the surgeries with $b_1=1$.}
  \label{fig:Seifert_L-space_surgery on T(2,20)}
\end{figure}

Nevertheless, the links $T(2,2n)$ are the simplest two-bridge links. In order to generally study $L$-space surgeries on $L$, we give an algorithm computing $\widehat{HF}(S^3_{\Lambda}(L))$ using the Alexander polynomials.

Another example is the Whitehead link. By the results in Section 6 in \cite{[YL]Surgery_2-bridge_link} or the method introduced in this section, we can obtain the following proposition. In order to distinguish it with its mirror, we call it the \emph{$L$-space Whitehead link}.

\begin{prop}
\label{prop:L-space_surgeries_on_Wh}
The $(p_1,p_2)$-surgery on the $L$-space Whitehead link is an $L$-space if and only if $p_1>0,p_2>0.$
\end{prop}
\subsection{Truncated perturbed surgery complex.}
The link surgery formula is an infinitely generated $\bb{F}[[U_1,U_2]]$-module. A truncation procedure is introduced in Section 8.3 in \cite{[MO]link_surgery} to reduce it to finitely generated $\bb{F}[[U_1,U_2]]$-module. It is called \emph{horizontal truncation} in \cite{[MO]link_surgery}, and we just call it \emph{truncation} here. A truncation for the $\Lambda$-surgery on a 2-component link $L$ is described by four finite subsets of
$\bb{H}(L)$,
$$S^{00}(\Lambda),S^{01}(\Lambda),S^{10}(\Lambda),S^{11}(\Lambda).$$
The way of doing truncation is not unique. Later, we will describe an explicit way which depends on $L$ and $\Lambda$.

Define $$\bar{C}^{\delta_1\delta_2}(\Lambda)=\displaystyle{\bigoplus_{\mathrm{\bf{s}}\in S^{\delta_1\delta_2}}}\mathfrak{A}^-_{p^{+\delta_1L_1\cup+\delta_2L_2}(\mathrm{\bf{s}})}, \quad\delta_1,\delta_2\in \{0,1\}.$$
Then, the truncated perturbed complex $\bar{C}(\Lambda)$ for an $L$-space link is defined as follows:
\begin{equation}
\label{eq:truncated_complex}
(\bar{C}^{-}(\mathcal{H},\Lambda),\bar{D}^{-}(\Lambda)):=\xyC{3pc}\xyR{3pc}\xymatrix{
\bar{C}^{00}(\Lambda)\ar[d]_{\bar{D}_{00}^{01}(\Lambda)}\ar[r]^{\bar{D}_{00}^{10}(\Lambda)}
&\bar{C}^{10}(\Lambda)\ar[d]^{\bar{D}_{10}^{01}(\Lambda)}\\
\bar{C}^{01}(\Lambda)\ar[r]_{\bar{D}_{01}^{10}(\Lambda)}
&\bar{C}^{11}(\Lambda),
}
\end{equation}
where
$\bar{D}^{\delta_1\delta_2}_{\varepsilon_1\varepsilon_2}(\Lambda)$ are the restrictions of $\tilde{D}^{\delta_1\delta_2}_{\varepsilon_1\varepsilon_2}(\Lambda)$ on the truncated complexes. See Equation \eqref{eq:Phi-tilde} and \eqref{eq:D-tilde} for the definitions of $\tilde{D}^{\delta_1\delta_2}_{\varepsilon_1\varepsilon_2}(\Lambda)$. They are determined by the set of integers $n_{\mathrm{\bf{s}}}^{\pm L_i}$.

The surgery complex naturally splits as a direct sum corresponding to $\text{Spin}^c$ structures. For the $\Lambda$-surgery on $L$, there is an identification $\text{Spin}^c(S^3_{\Lambda}(L))=\bb{H}(L)/H(L,\Lambda)$, where $H(L,\Lambda)$ is the lattice spanned by $\Lambda$. For $\mathfrak{u}\in \bb{H}(L)/H(L,\Lambda)$, choose $\s=(s_1,s_2)\in \mathfrak{u}.$
Denote
$$\bar{C}^{\delta_1\delta_2}(\Lambda,\mathfrak{u})=\underset{i\in\bb{Z}}{\displaystyle\bigoplus} \underset{
\scriptsize{\begin{array}{c} j\in\bb{Z}\\ s+i\Lambda_1+j\Lambda_2\in S^{\delta_1\delta_2}\end{array}}}{\displaystyle\bigoplus}\tilde{\mathfrak{A}}^-_{s+i\Lambda_1+j\Lambda_2}.$$
Then, the summand $\bar{C}(\Lambda,\mathfrak{u})$ is as follows:
\begin{equation}
\label{eq:truncated_complex_in Spin^c}
(\bar{C}^{-}(\mathcal{H},\Lambda,\mathfrak{u}),\bar{D}^{-}(\Lambda,\mathfrak{u})):=\xyC{3pc}\xyR{3pc}\xymatrix{
\bar{C}^{00}(\Lambda,\mathfrak{u})\ar[d]_{\bar{D}_{00}^{01}(\Lambda,\mathfrak{u})}\ar[r]^{\bar{D}_{00}^{10}(\Lambda,\mathfrak{u})} & \bar{C}^{10}(\Lambda,\mathfrak{u})\ar[d]^{\bar{D}_{10}^{01}(\Lambda,\mathfrak{u})}\\
\bar{C}^{01}(\Lambda,\mathfrak{u})\ar[r]_{\bar{D}_{01}^{10}(\Lambda,\mathfrak{u})} & \bar{C}^{11}(\Lambda,\mathfrak{u}).
}
\end{equation}

By putting $U_1=0$, we can get the chain complex of $\bb{F}$-vector spaces $\bar{C}\hat{\ }(\Lambda,\mathfrak{u})$, whose homology is isomorphic to $\widehat{HF}(S^3_{\Lambda}(L),\mathfrak{u})$.

\begin{lem}
Suppose $A,B,C,D$ are finite dimensional $\bb{F}$-vector spaces and the following diagram commutes
\[\xymatrix{A\ar[r]^{h_1} \ar[d]_{v_1}&B\ar[d]^{v_2}\\
C\ar[r]_{h_2}&D.}
\]
We form a chain complex $(R_*,d_*)$ supported on degrees $0,1,2,$ $R:A\xrightarrow{d_2} B\oplus C\xrightarrow{d_1} D$ with $d_2=h_1+v_1$ and $d_1=h_2\oplus v_2.$ Then, we have the following conclusions
\begin{enumerate}[(1)]
\item $\dim H_*(R)=2\dim(\Ker h_1\cap \Ker v_1)-2\dim(\im v_2+\im h_2)-\dim A+\dim B+\dim C+\dim D; $
\item $\dim H_*(R)=1$ iff one of the following is true
\begin{enumerate}[a.]
\item $\bigchi=\dim A-\dim B-\dim C+ \dim D=1,$ and $\dim(\Ker(h_1) \cap \Ker(v_1))+\dim \mathrm{Coker}(v_2+h_2)= 1.$
\item $\bigchi=\dim A-\dim B-\dim C+ \dim D=-1,$ and $\dim(\Ker (h_1) \cap \Ker(v_1))+\dim \mathrm{Coker}(v_2+h_2)=0.$
\end{enumerate}
\end{enumerate}
\end{lem}

\begin{proof}
Part (1) is a straightforward computation. Notice that $H_0=\mathrm{Coker}(h_2\oplus v_2)$, $H_2=\Ker(h_1+v_1)$.

For Part (2), there are only three cases when $H_*(R)=\bb{F}$ happens,
\begin{enumerate}
\item $H_0(R)=\bb{F},H_1(R)=H_2(R)=0;$
\item $H_1(R)=\bb{F},H_0(R)=H_2(R)=0;$
\item $H_2(R)=\bb{F},H_0(R)=H_1(R)=0.$
\end{enumerate}
In cases (1) and (3), we have that $\bigchi=1$ and $\dim H_0+\dim H_2=1$; in case (2), we have that $\bigchi=-1$ and $\dim H_0+\dim H_1=0.$ It is not hard to check the converse.
\end{proof}

If $\Ker (v_1),\Ker (h_1)$ are both known, then computing $\dim(\Ker (v_1) \cap \Ker (h_1))$ is equivalent to computing $\dim(\Ker(v_1)+ \Ker( h_1))$, which can be done by Gauss Elimination.

While computing $\mathrm{Coker}(v_2+h_2)$ is the dual question for computing $\mathrm{Ker}(v_2^*)\cap\mathrm{Ker}(h_2^*)$. While  the dual maps
$v_2^*$ and $h_2^*$ can be obtained by reversing the arrows, since we are working over $\bb{F}$.

We can directly apply the above lemma for each truncated perturbed complex $\bar{C}\hat{\ }(\Lambda,\mathfrak{u})$ for each $\text{Spin}^c$ structure. Thus, we only
need to describe the truncated regions $S^{00}(\Lambda),S^{01}(\Lambda),S^{10}(\Lambda),S^{11}(\Lambda)$ and the kernels of the maps
$\bar{D}\hat{\ }_{**}^{**}(\Lambda,\mathfrak{u})$ and their dual.

\begin{prop}
\label{prop:L-space_surgeries_on L-space links}
Suppose $L$ is an $L$-space link.  Fix a surgery framing $\Lambda$ and a $\text{Spin}^c$ structure $\mathfrak{u}$. Then, $\widehat{HF}(\Lambda,\mathfrak{u})=\bb{F}$ iff in the truncated complex $\bar{C}\hat{\ }(\Lambda,\mathfrak{u})$, one of the following is true,
\begin{enumerate}[(A)]
\item $\#S^{00}(\Lambda,\mathfrak{u})-\#S^{01}(\Lambda,\mathfrak{u})-\#S^{10}(\Lambda,\mathfrak{u})+\#S^{11}(\Lambda,\mathfrak{u})= 1,$ and $\dim(\Ker (\hat{D}_{00}^{01}) \cap \Ker (\hat{D}_{00}^{10}))+\dim \mathrm{Coker}(\hat{D}_{01}^{10}+\hat{D}_{10}^{01})= 1.$
\item $\#S^{00}(\Lambda,\mathfrak{u})-\#S^{01}(\Lambda,\mathfrak{u})-\#S^{10}(\Lambda,\mathfrak{u})+\#S^{11}(\Lambda,\mathfrak{u})= -1,$ and $\dim(\Ker (\hat{D}_{00}^{01}) \cap \Ker (\hat{D}_{00}^{10}))+\dim \mathrm{Coker}(\hat{D}_{01}^{10}+\hat{D}_{10}^{01})= 0.$
\end{enumerate}
\end{prop}

\subsection{Truncations.}
We explicitly describe the truncated regions $S^{00}(\Lambda),S^{01}(\Lambda),S^{10}(\Lambda),S^{11}(\Lambda)$ here.
Let us briefly recall the procedure to form these truncated regions for a general two-component link $L$ in Section 8.3 \cite{[MO]link_surgery}.
 \begin{enumerate}
 \item Choose a number $b\in \bb{N}$, such that the inclusion maps $I^{\pm L_i}_{s_1,s_2}$'s are  quasi-isomorphisms whenever $\pm s_i\geq b.$
 \item Determine a parallelogram $Q$ in the plane, with vertices $P_1,P_2,P_3,P_4$ counterclockwise labelled, satisfying the following condition: The point $P_i$ has the coordinate $(x_i,y_i)$ such that
 \begin{equation}
 \label{eq:quadrilateral_conditions}
 \begin{cases}
x_{1}>b\\
y_{1}>b,
\end{cases}
 \begin{cases}
x_{2}<-b\\
y_{2}>b,
\end{cases}
 \begin{cases}
x_{3}<-b\\
y_{3}<-b,
\end{cases}
\begin{cases}
 x_{4}>b\\
y_{4}<-b.
\end{cases}
 \end{equation}
We also require that every edge is either  parallel to the vector $\Lambda_1$  with length greater than $\Vert\Lambda_1\Vert$ or parallel to  $\Lambda_2$ with length greater than $\Vert\Lambda_2\Vert$.
 \item Decide which is the case among the six cases of the surgeries described in Figure 22 in \cite{[MO]link_surgery}. Then, we can decide the corresponding truncated regions according to Section 8.3 in \cite{[MO]link_surgery}.
 \end{enumerate}

\begin{figure}
  \centering
  \centering
  \includegraphics[scale=0.35]{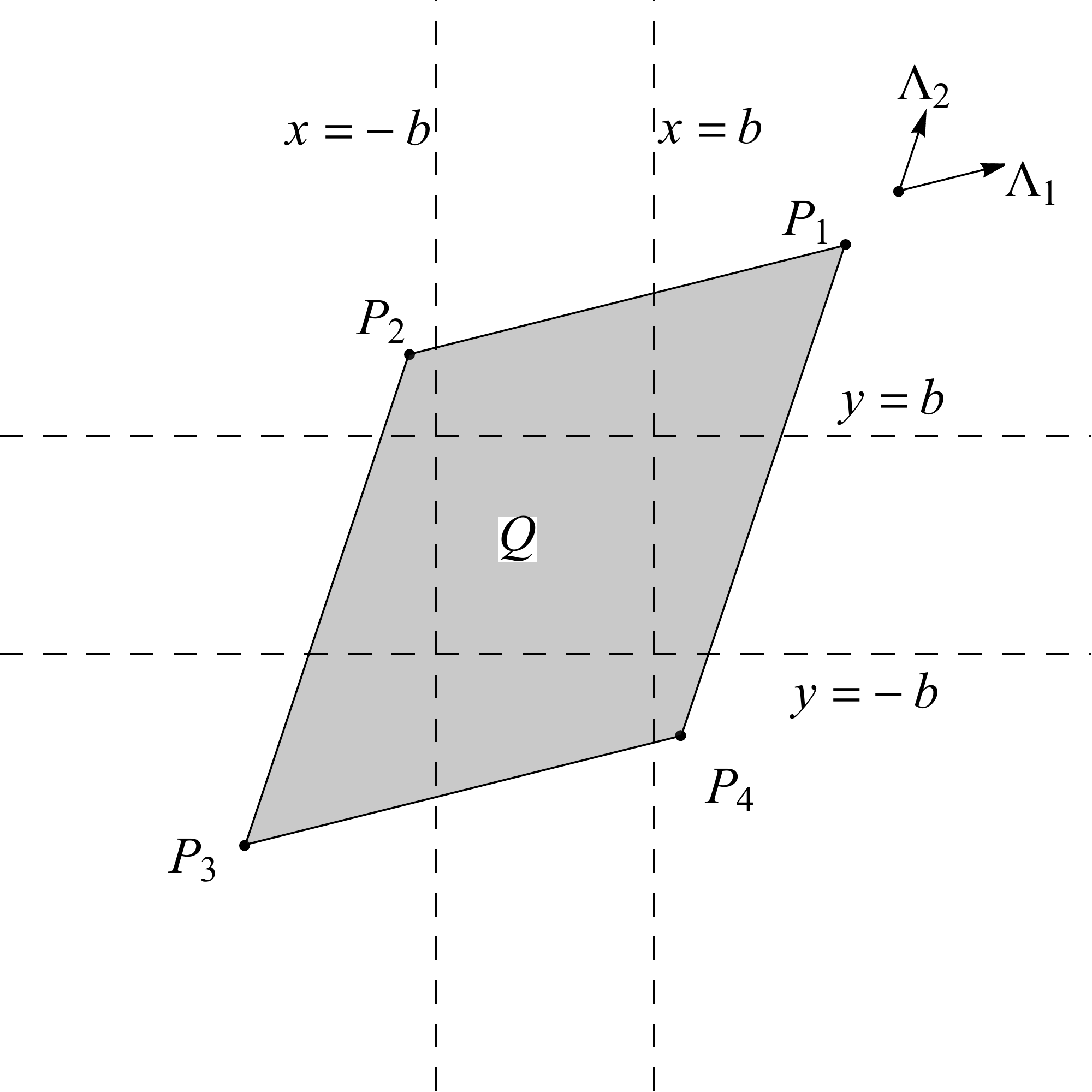}
  \caption{\textbf{The truncation.} The vectors $\Lambda_1$ and $\Lambda_2$ are determined by the surgery framing matrix. The edges of the parallelogram $Q$ are parallel to $\Lambda_1$ and $\Lambda_2$, and they indicate the border lines of various acyclic subcomplexes or quotient complexes.  Thus, the parallelogram $Q$ roughly indicates the support of the truncated complex.}
  \label{fig:truncation}
\end{figure}

The way of doing truncation is not unique. One explicit way to choose the parallelogram $Q$ to be centered at the origin as follows. See Figure \ref{fig:truncation}.

Let
$$\left\{P_1,P_2,P_3,P_4\right\}=\left\{\frac{i_0\Lambda_1+j_0\Lambda_2}{2},\frac{-i_0\Lambda_1+j_0\Lambda_2}{2},
\frac{i_0\Lambda_1-j_0\Lambda_2}{2},\frac{-i_0\Lambda_1-j_0\Lambda_2}{2}\right\},$$
with $i_0,j_0$ being positive integers, such that Equations \eqref{eq:quadrilateral_conditions} hold.

Fix $\Lambda$ and $\mathfrak{u}\in\bb{H}(L)/H(L,\Lambda)$. Suppose
$$\mathrm{\bf{s}}=\theta_1\Lambda_1+\theta_2\Lambda_2\in \mathfrak{u}, \quad P_1=a_1\Lambda_1+a_2\Lambda_2.$$
We denote
\begin{align*}
&A_1=\lceil-\theta_{1}-|a_{1}|\rceil, &A_2=\lfloor-\theta_{1}+|a_{1}|\rfloor,\\
&B_1=\lceil-\theta_{2}-|a_{2}|\rceil, &B_2=\lfloor-\theta_{2}+|a_{2}|\rfloor.
\end{align*}
Then, the truncated regions in the six cases are as follows.
\begin{description}
\item[Case I]
\begin{align*}
& S^{00}(\Lambda,\mathfrak{u})=\mathfrak{u}\cap Q,\\
& S^{10}(\Lambda,\mathfrak{u})=\mathfrak{u}\cap Q\cap(Q+\Lambda_{1}),\\
& S^{01}(\Lambda,\mathfrak{u})=\mathfrak{u}\cap Q\cap(Q+\Lambda_{2}),\\
& S^{11}(\Lambda,\mathfrak{u})=\mathfrak{u}\cap Q\cap(Q+\Lambda_{1}+\Lambda_{2}).	
\end{align*} In other words, for $\delta_1,\delta_2\in\{0,1\},$
\[S^{\delta_{1}\delta_{2}}(\Lambda,\mathfrak{u})=\left\{s+i\Lambda_{1}+j\Lambda_{2}|A_1+\delta_{1}\leq i\le A_2, B_1 +\delta_2\leq j\le B_2.\right\}.
\]
\item[Case II]
\begin{align*}
& S^{00}(\Lambda,\mathfrak{u})=\mathfrak{u}\cap Q,\\	
& S^{10}(\Lambda,\mathfrak{u})=\mathfrak{u}\cap \{Q\cup(Q+\Lambda_{1})\},\\
& S^{01}(\Lambda,\mathfrak{u})=\mathfrak{u}\cap \{Q\cup(Q+\Lambda_{2})\},\\	
& S^{11}(\Lambda,\mathfrak{u})=\mathfrak{u}\cap \{Q\cup(Q+\Lambda_{1})\cup(Q+\Lambda_{2})\cup(Q+\Lambda_{1}+\Lambda_{2})\}.	
\end{align*} In other words, for $\delta_1,\delta_2\in\{0,1\},$
\[S^{\delta_{1}\delta_{2}}(\Lambda,\mathfrak{u})=\left\{s+i\Lambda_{1}+j\Lambda_{2}|A_1-\delta_{1}\leq i\le A_2, B_1-\delta_{2}\leq j\le B_2.\right\}.
\]
\item[Case III]
\begin{align*}
& S^{00}(\Lambda,\mathfrak{u})=\mathfrak{u}\cap Q,\\	
& S^{10}(\Lambda,\mathfrak{u})=\mathfrak{u}\cap \{Q\cap(Q+\Lambda_{1})\},\\
& S^{01}(\Lambda,\mathfrak{u})=\mathfrak{u}\cap \{Q\cup(Q+\Lambda_{2})\},\\	
& S^{11}(\Lambda,\mathfrak{u})=\mathfrak{u}\cap \{[Q\cup(Q+\Lambda_{2})]\cap([Q\cup(Q+\Lambda_{2})]+\Lambda_1)\}.	
\end{align*} In other words, for $\delta_1,\delta_2\in\{0,1\},$
\[S^{\delta_{1}\delta_{2}}(\Lambda,\mathfrak{u})=\left\{s+i\Lambda_{1}+j\Lambda_{2}|A_1+\delta_{1}\leq i\le A_2, B_1\leq j\le B_2+\delta_{2}.\right\}.
\]
\item[Case IV]
\begin{align*}
& S^{00}(\Lambda,\mathfrak{u})=\mathfrak{u}\cap Q,\\	
& S^{10}(\Lambda,\mathfrak{u})=\mathfrak{u}\cap \{Q\cup(Q+\Lambda_{1})\},\\
& S^{01}(\Lambda,\mathfrak{u})=\mathfrak{u}\cap \{Q\cap(Q+\Lambda_{2})\},\\	
& S^{11}(\Lambda,\mathfrak{u})=\mathfrak{u}\cap \{[Q\cap(Q+\Lambda_{2})]\cup([Q\cap(Q+\Lambda_{2})]+\Lambda_1)\}.	
\end{align*} In other words, for $\delta_1,\delta_2\in\{0,1\},$
\[S^{\delta_{1}\delta_{2}}(\Lambda,\mathfrak{u})=\left\{s+i\Lambda_{1}+j\Lambda_{2}|A_1\leq i\le A_2+\delta_{1}, B_1+\delta_{2}\leq j\le B_2.\right\}.
\]
\item[Case V] This case is similar to {\bf{Case I}}, but the regions $S^{10}(\Lambda,\mathfrak{u}),S^{01}(\Lambda,\mathfrak{u})$ have two more points at the corners.
\begin{align*}
& S^{00}(\Lambda,\mathfrak{u})=\mathfrak{u}\cap Q,\\	
& S^{10}(\Lambda,\mathfrak{u})=(\mathfrak{u}\cap Q\cap(Q+\Lambda_{1}))\cup T^{10},\\
& S^{01}(\Lambda,\mathfrak{u})=(\mathfrak{u}\cap Q\cap(Q+\Lambda_{2}))\cup T^{01},\\	
& S^{11}(\Lambda,\mathfrak{u})=\mathfrak{u}\cap Q\cap(Q+\Lambda_{1}+\Lambda_{2}),	
\end{align*}
where $T^{10}=\{\mathrm{\bf{s}}+A_2\Lambda_1+B_1\Lambda_2\},T^{10}=\{\mathrm{\bf{s}}+A_1\Lambda_1+B_2\Lambda_2\}.$

In other words, for $\delta_1,\delta_2\in\{0,1\},$
\[S^{\delta_{1}\delta_{2}}(\Lambda,\mathfrak{u})=\left\{s+i\Lambda_{1}+j\Lambda_{2}|A_1+\delta_1\leq i\le A_2, B_1+\delta_{2}\leq j\le B_2.\right\}\cup T^{\delta_{1}\delta_{2}},
\]
where $T^{00}=T^{11}=\emptyset.$
\item[Case VI] This is similar to {\bf{Case V}}.
\begin{align*}
& S^{00}(\Lambda,\mathfrak{u})=\mathfrak{u}\cap Q\cap (Q-\Lambda_1-\Lambda_2),\\	
& S^{10}(\Lambda,\mathfrak{u})=(\mathfrak{u}\cap Q\cap(Q-\Lambda_{1}))\cup T^{10},\\
& S^{01}(\Lambda,\mathfrak{u})=(\mathfrak{u}\cap Q\cap(Q-\Lambda_{2}))\cup T^{01},\\	
& S^{11}(\Lambda,\mathfrak{u})=\mathfrak{u}\cap Q,	
\end{align*}
where $T^{10}=\mathrm{\bf{s}}+A_1\Lambda_1+B_2 \Lambda_2, T^{01}=\mathrm{\bf{s}}+B_1\Lambda_2 +A_1\Lambda_1.$

In other words, for $\delta_1,\delta_2\in\{0,1\},$
\[S^{\delta_{1}\delta_{2}}(\Lambda,\mathfrak{u})=\left\{s+i\Lambda_{1}+j\Lambda_{2}|A_1+1-\delta_1\leq i\le A_2, B_1+1-\delta_{2}\leq j\le B_2.\right\}\cup T^{\delta_{1}\delta_{2}},
\]
where $T^{00}=T^{11}=\emptyset$.
\end{description}

\begin{rem}
In all of the above cases, $\#S^{00}(\Lambda,\mathfrak{u})-\#S^{01}(\Lambda,\mathfrak{u})-\#S^{10}(\Lambda,\mathfrak{u})+\#S^{11}(\Lambda,\mathfrak{u})=\pm 1.$
\end{rem}

\subsection{Kernel of $\bar{D}\hat{\ }_{**}^{**}(\Lambda,\mathfrak{u})$} In fact, all the mapping cones of $\bar{D}\hat{\ }_{**}^{**}(\Lambda,\mathfrak{u})$ split as a direct sum of mapping cones in a common form. They look like the mapping cones in computing $+1$-surgery on knots. Since this type of mapping cones looks like zigzags, we just call them "zigzags".
We denote the set of integers in $[a,b]$ by $[a;b]$, where we allow $a=b$.

\begin{defn}[Zigzags]
A \emph{zigzag mapping cone} $C$ is a mapping cone of $\bb{F}$-vector spaces:
$$\bigoplus_{a_1\leq s\leq a_2}A_s\xrightarrow{f+g}\bigoplus_{b_1\leq t \leq b_2} B_t,$$
where
\begin{align*}
&A_s=\bb{F},\  \forall a_1\leq s\leq a_2,\\
&B_t=\bb{F},\  \forall b_1\leq t\leq b_2,\\
&f=\bigoplus f_s, \quad f_s:A_s\to B_s,\\
&g=\bigoplus g_s,\quad g_s:A_s\to B_{s+1}.
\end{align*}
The \emph{code} of the zigzag $C$ is a set of data $\{[a_1;a_2],[b_1;b_2],S_1,S_2\}$, where
\begin{align}
S_1=\{s\in\bb{Z}|f_s\neq0\},\\ S_2=\{s\in\bb{Z}|g_s\neq0\}.
\end{align}
We define $\Ker(C)$ (resp. $\mathrm{Coker}(C)$) to be $\mathrm{Ker}(f+g)$ (resp. $\mathrm{Coker}(f+g)$).
\end{defn}

\begin{defn}
For any element $x$ in $\bigoplus_{a_1\leq s\leq a_2}\bb{F}\cdot e_s$, we can represent it uniquely by $x=\sum_{s\in \Gamma}e_s.$ We call $\Gamma$ the \emph{support} of $x$, and denote it by $\mathrm{Supp}(x)$. Similarly, for $X=\{x_1,...,x_n\}$, we denote
$\{\mathrm{Supp}(x_1),...,\mathrm{Supp}(x_n)\}$ by $\mathrm{Supp}(X)$.
\end{defn}

\begin{prop}
For a zigzag $C$ with  the code  $\{[a_1,a_2],[b_1,b_2],S_1,S_2\}$, we represent $S_1\cap S_2$ by a minimal disjoint unions
\[S_1\cap S_2=\coprod_{i\in [1;K]}[\alpha_i;\beta_i],
\]
with $\beta_{i}\leq \alpha_{i+1}+2,\forall i.$
Then, $\Ker(C)$ has a basis with the following support
$$\Big\{\{s\}\Big|s\in [a_1,a_2]\backslash(S_1\cup S_2) \Big\}\cup\Big\{[\alpha_j-1,\beta_j+1]\Big|\alpha_j-1\in S_2,\beta_j+1\in S_1\Big\}.$$
\end{prop}

\begin{proof} Straightforward.
\end{proof}
\begin{defn}
Let $L=L_1\cup L_2$ be an $L$-space link. For all $s_1\in \bb{H}_1(L),s_2\in \bb{H}_2(L),$ we define
\begin{align}
 \nu^{+L_2}_{s_1}(L)&=\min\{s_2\in\bb{H}_2(L)\big|n^{+L_2}_{s_1,s_2}\neq 0 \},\\
 \nu^{+L_1}_{s_2}(L)&=\min\{s_1\in\bb{H}_1(L)\big|n^{+L_1}_{s_1,s_2}\neq 0 \}.
\end{align}
\end{defn}

It is easy to see that in Section 6.3 we can let $b=\max\{\max\{ \nu^{+L_2}_{s_1}(L)\}_{s_1},\max\{ \nu^{+L_1}_{s_2}(L)\}_{s_2}\}$.  Moreover, the truncated perturbed complex $\bar{C}\hat{\ }(\Lambda)$ is determined by these $\nu^{+L_2}_{s_1}(L)$'s and $\nu^{+L_1}_{s_2}(L)$'s, and thereby so are the zigzag mapping cones corresponding to $\bar{D}\hat{\ }_{**}^{**}(\Lambda,\mathfrak{u})$'s.

For example, suppose
$$\Lambda=\left(\begin{array}{cc}
p_{1} & \mathrm{lk}\\
\mathrm{lk} & p_{2}
\end{array}\right)
$$
and choose $\mathrm{\bf{s}}=(s_1,s_2)\in\mathfrak{u}.$ Before truncation, we have
$$\mathrm{cone}(\hat{D}_{00}^{01}(\Lambda,\mathfrak{u}))=\prod_{i\in \bb{Z}}\mathrm{cone}(\prod_{j\in \bb{Z}}(\hat{\Phi}^{+L_2}_{\mathrm{\bf{s}}+i\Lambda_1+j\Lambda_2}+\hat{\Phi}^{-L_2}_{\mathrm{\bf{s}}+i\Lambda_1+j\Lambda_2})).$$

After truncation, $\mathrm{cone}(\bar{D}\hat{\ }_{00}^{01}(\Lambda,\mathfrak{u}))$ splits into direct sums of zigzags in form of
$$\mathrm{cone}(\prod_{j\in\bb{Z}}(\hat{\Phi}^{+L_2}_{\mathrm{\bf{s}}+i\Lambda_1+j\Lambda_2}
+\hat{\Phi}^{-L_2}_{\mathrm{\bf{s}}+i\Lambda_1+j\Lambda_2}))\cap \bar{C}\hat{\ }(\Lambda).$$

Let us figure out the codes of these zigzags. Suppose the code  of the above zigzag is
$$\{[a_1;a_2],[b_1;b_2],S_1,S_2\}.$$

Then, it is not hard to get the following formulas for the code,
\[\begin{cases}
\quad[a_1;a_2]=\{j\in \bb{Z}|\s+i \Lambda_1+j \Lambda_2\in S^{00}(\Lambda,\mathfrak{u}) \},\\
\quad[b_1;b_2]=\{j\in \bb{Z}|\s+i \Lambda_1+j \Lambda_2\in S^{01}(\Lambda,\mathfrak{u}) \},\\
\quad S_1=\{j\in \bb{Z}| s_2+i\cdot \mathrm{lk}+j\cdot p_2 \geq \nu^{+L_2}_{s_1+i\cdot p_1+j\cdot\mathrm{lk}}(L)\},\\
\quad S_2=\{j\in \bb{Z}| s_2+i\cdot \mathrm{lk}+j\cdot p_2 \leq -\nu^{+L_2}_{-s_1-i\cdot p_1-j\cdot\mathrm{lk}}(L)\}.
\end{cases}
\]

\subsection{Examples: $L$-space surgeries on two-bridge links}  From Proposition \ref{prop:2-component Link surgery}, we see that if a two-bridge link has an $L$-space surgery, then it is a generalized $L$-space link. By taking mirrors, we can reduce these links to two types: $L$-space links and generalized $(+-)L$-space links.  We have discussed two-bridge $L$-space links in Section 3. By the method in this section, it is convenient to make computer programs for computing $\widehat{HF}$ of their surgeries and give characterizations of $L$-space surgeries.  For example, regarding the surgeries on the Whitehead link,  we can do truncations as in Proposition 6.9 in \cite{[YL]Surgery_2-bridge_link} and then use the method of zig-zags in Section 6.3 to recover the results in Proposition 6.9 \cite{[YL]Surgery_2-bridge_link} for the hat version. Thus, we can obtain  Proposition  \ref{prop:L-space_surgeries_on_Wh}.  However, to find a general formula of $\widehat{HF}$ is not easy.

In fact, finding $L$-space homology spheres is more interesting. Let us try some examples here, by looking at the $(1,1)$-surgeries on a sequence of two-bridge links $L_n=b(4n^2+4n,-2n-1)$ for all positive integers $n$. This sequence of $L$-space links have linking numbers 0. Note that $L_1$ is the Whitehead link.

\begin{prop}
For all $n\geq2$, the $(1,1)$-surgery on $b(4n^2+4n,-2n-1)$ is not an $L$-space.
\end{prop}

With the help of a computer program, we get the Alexander polynomials of $L_n$:
\[\Delta_{L_n}(x,y)=\sum_{j=-n}^{n-1}\sum_{i=-n-\frac{1}{2}+|j+\frac{1}{2}|}^{n-\frac{1}{2}-|j+\frac{1}{2}|}
(-1)^{i+j}x^{i+\frac{1}{2}}y^{j+\frac{1}{2}}.\]
After normalizing $\Delta_{L_n}(x,y)$ by Definition \ref{defn:noralization of ALexPoly}, we can get formulas for $n^{+L_1}_{s_1,s_2}$ by Equation \eqref{eq:n^+L_1}, \eqref{eq:n^+L_2}. We list the numbers $\{n^{+L_2}_{s_1,s_2}(L_n)\}_{-4\le s_1\le4,-4\le s_2\le 4}$ for $n=1,2,3,4$ as follows:

\begin{align*}
\{n_{s_{1},s_{2}}^{+L_{2}}(L_{1})\}:\left\{\begin{array}{ccccccccc}
0 & 0 & 0 & 0 & 0 & 0 & 0 & 0 & 0\\
0 & 0 & 0 & 0 & 0 & 0 & 0 & 0 & 0\\
0 & 0 & 0 & 0 & 0 & 0 & 0 & 0 & 0\\
0 & 0 & 0 & 0 & 0 & 0 & 0 & 0 & 0\\
0 & 0 & 0 & 0 & 1 & 0 & 0 & 0 & 0\\
1 & 1 & 1 & 1 & 1 & 1 & 1 & 1 & 1\\
2 & 2 & 2 & 2 & 2 & 2 & 2 & 2 & 2\\
3 & 3 & 3 & 3 & 3 & 3 & 3 & 3 & 3\\
4 & 4 & 4 & 4 & 4 & 4 & 4 & 4 & 4
\end{array}\right\}; &
\{n_{s_{1},s_{2}}^{+L_{2}}(L_{2})\}:\left\{\begin{array}{ccccccccc}
0 & 0 & 0 & 0 & 0 & 0 & 0 & 0 & 0\\
0 & 0 & 0 & 0 & 0 & 0 & 0 & 0 & 0\\
0 & 0 & 0 & 0 & 0 & 0 & 0 & 0 & 0\\
0 & 0 & 0 & 0 & 1 & 0 & 0 & 0 & 0\\
0 & 0 & 0 & 1 & 1 & 1 & 0 & 0 & 0\\
1 & 1 & 1 & 1 & 2 & 1 & 1 & 1 & 1\\
2 & 2 & 2 & 2 & 2 & 2 & 2 & 2 & 2\\
3 & 3 & 3 & 3 & 3 & 3 & 3 & 3 & 3\\
4 & 4 & 4 & 4 & 4 & 4 & 4 & 4 & 4
\end{array}\right\}.
\end{align*}
\begin{align*}
\{n_{s_{1},s_{2}}^{+L_{2}}(L_{3})\}:\left\{\begin{array}{ccccccccc}
0 & 0 & 0 & 0 & 0 & 0 & 0 & 0 & 0\\
0 & 0 & 0 & 0 & 0 & 0 & 0 & 0 & 0\\
0 & 0 & 0 & 0 & 1 & 0 & 0 & 0 & 0\\
0 & 0 & 0 & 1 & 1 & 1 & 0 & 0 & 0\\
0 & 0 & 1 & 1 & 2 & 1 & 1 & 0 & 0\\
1 & 1 & 1 & 2 & 2 & 2 & 1 & 1 & 1\\
2 & 2 & 2 & 2 & 3 & 2 & 2 & 2 & 2\\
3 & 3 & 3 & 3 & 3 & 3 & 3 & 3 & 3\\
4 & 4 & 4 & 4 & 4 & 4 & 4 & 4 & 4
\end{array}\right\}; &
\{n_{s_{1},s_{2}}^{+L_{2}}(L_{4})\}:\left\{\begin{array}{ccccccccc}
0 & 0 & 0 & 0 & 0 & 0 & 0 & 0 & 0\\
0 & 0 & 0 & 0 & 1 & 0 & 0 & 0 & 0\\
0 & 0 & 0 & 1 & 1 & 1 & 0 & 0 & 0\\
0 & 0 & 1 & 1 & 2 & 1 & 1 & 0 & 0\\
0 & 1 & 1 & 2 & 2 & 2 & 1 & 1 & 0\\
1 & 1 & 2 & 2 & 3 & 2 & 2 & 1 & 1\\
2 & 2 & 2 & 3 & 3 & 3 & 2 & 2 & 2\\
3 & 3 & 3 & 3 & 4 & 3 & 3 & 3 & 3\\
4 & 4 & 4 & 4 & 4 & 4 & 4 & 4 & 4
\end{array}\right\}.
\end{align*}

In particular, we get the following formulas for all $s_1\in \bb{Z}$,
$$\nu_{s_1}^{+L_2}(L_n)=\begin{cases}n-|s_1|, &|s_1|\leq n,\\ 0,&|s_1|\geq n.\end{cases}$$

Since $L_n$ is a two-bridge link, we have the symmetry $\nu^{+L_1}_{s_2}=\nu^{+L_2}_{s_1}$, when $s_1=s_2$.

Thus, we can let $b(L_n)=n.$ Then, as described in Section 6.3, the truncation regions are determined by the parallelogram $Q$, with
vertices $P_1=(n,n),P_2=(-n,n),P_3=(-n,-n),P_4=(n,-n).$ The surgery framing is in Case I, so we have the truncated regions
\[S^{\delta_{1}\delta_{2}}=\left\{(i,j)\in\bb{Z}^2\Big{|} -n +\delta_1\leq i\le n, -n +\delta_2\leq j\le n\right\}.\]

Now we can see
$$\hat{\Phi}^{\pm L_i}_{s_1,s_2}=0, \forall -n+1\leq s_1\leq n-1, -n+1\leq s_2\leq n-1, i=1,2.$$
So $\hat{\mathfrak{A}}_{s_1,s_2}\in \bar{C}\hat{}^{00}$ with $-n<s_1<n,-n<s_2<n$ are all in the kernel of $\hat{D}_{00}^{10}$
and $\hat{D}_{00}^{01}.$ So when $n\geq 2,$ we have that $\Ker(\hat{D}_{00}^{01})\cap \Ker(\hat{D}_{00}^{10})$ has rank at least
$n^2+(n-1)^2>1.$ Thus, by Proposition \ref{prop:L-space_surgeries_on L-space links}, the $(1,1)$-surgeries on $L_n$ with $n\geq2$ are never $L$-spaces. Similar arguments apply to $(\pm1,\pm1)$-surgeries on these links.

\begin{prop}
On the two-bridge $L$-space links $L_n=b(4n^2+4n,-2n-1)$ with $n\geq 2$, there are no $L$-space homology sphere surgeries.
\end{prop}

In fact, direct computations using the zigzags give that $\widehat{HF}(S^3_{1,1}(L_n))$ has dimension $(2n-1)^2.$

\end{document}